\documentclass[final,leqno,onefignum,onetabnum]{siamart190516}
\usepackage{amsmath}
\usepackage{enumitem}
\usepackage{amssymb}
\usepackage{mathrsfs}
\usepackage{hyperref}
\usepackage{algorithm}
\usepackage{algorithmic}
\usepackage{mathtools}
\usepackage{float}
\usepackage{multirow}
\usepackage{subcaption}
\usepackage{subfloat}
\usepackage[version=4]{mhchem}

\hypersetup{colorlinks=true,citecolor=blue,linkcolor=blue,urlcolor=black}

\newtheorem{remark}[theorem]{Remark}

\newtheorem{assumption}[theorem]{Assumption}

\graphicspath{{figure/}}

\providecommand{\inner}[1]{\left\langle#1\right\rangle} 

\begin{document}
	\title{Mathematical Analysis and Numerical Approximations of Density Functional Theory Models for Metallic Systems\thanks{This work was supported by  the National Natural Science Foundation of China under grant 12021001, the National Key R \& D Program of China under grants 2019YFA0709600 and 2019YFA0709601, and  the CAS President’s International Fellowship for Visiting Scientists under grants 2019VMA0029. de Gironcoli also acknowledges support from the European Union’s Horizon 2020 research and innovation program (Grant No. 824143, MaX ``MAterials design at the eXascale'' Centre of Excellence).}}
	\author{Xiaoying Dai\footnotemark[2] \and Stefano de Gironcoli\footnotemark[3] \and Bin Yang\footnotemark[4] \and Aihui Zhou\footnotemark[2]}
	
	\maketitle
	\renewcommand{\thefootnote}{\fnsymbol{footnote}}
	\footnotetext[2]{LSEC, Institute of Computational Mathematics and Scientific/Engineering Computing, Academy of Mathematics and Systems Science, Chinese Academy of Sciences, Beijing 100190, China; and School of Mathematical Sciences, University of Chinese Academy of Sciences, Beijing 100049, China. \{daixy, azhou\}@lsec.cc.ac.cn.}
	\footnotetext[3]{Scuola Internazionale Superiore di Studi Avanzati (SISSA) and CNR-IOM DEMOCRITOS Simulation Centre, Via Bononea 265, 34146 Trieste, Italy. degironc@sissa.it.}
	\footnotetext[4]{NCMIS, Academy of Mathematics and Systems Science, Chinese Academy of Sciences, Beijing 100190, China. binyang@lsec.cc.ac.cn.} 
	\renewcommand{\thefootnote}{\arabic{footnote}}
	
	\begin{abstract}
		In this paper, we investigate the energy minimization model of the ensemble Kohn-Sham density functional theory for metallic systems, in which a pseudo-eigenvalue matrix and a general smearing approach are involved. We study the invariance and the existence of the minimizer of the energy functional. We propose an adaptive double step size strategy and the corresponding preconditioned conjugate gradient methods for solving the energy minimization model. Under some mild but reasonable assumptions, we prove the global convergence of our algorithms. Numerical experiments show that our algorithms are efficient, especially for large scale metallic systems. In particular, our algorithms produce convergent numerical approximations for some metallic systems, for which the traditional self-consistent field iterations fail to converge.
	\end{abstract}

	\section{Introduction} 
	The Kohn-Sham density functional theory (DFT) is widely used in the electronic structure calculations \cite{becke2014perspective,bris2003special,martin2020electronic,parr1994densityfunctional}. The underlying mathematical model is often formulated as either a nonlinear eigenvalue problem or an energy minimization problem with an unitary constraint. The most commonly used approach for computing the Kohn-Sham DFT model is to solve the nonlinear eigenvalue problem by using the self-consistent field (SCF) iterations. However, the convergence of the SCF iterations is not guaranteed and the performance of the SCF iterations is unpredictable, especially for large scale systems. Consequently, people turn to pay attention to investigating the constrained energy minimization problem (see, e.g., \cite{dai2017conjugate,gao2018new,schneider2009direct,zhang2014gradient,zhao2015riemannian} and references therein).
	
	We particularly note that the efficient numerical methods for the classical Kohn-Sham DFT model, in which occupation numbers are either $1$ or $0$, are inefficient or even invalid for metallic systems. The main reason is that the gap between the highest occupied state and the lowest unoccupied state for metallic systems is very small or absent. More precisely, the classical Kohn-Sham DFT model becomes ill-posed due to its difficulty to separate the occupied states and unoccupied states.
	
	To provide a well-posed and efficient mathematical model for metallic systems, the unoccupied states have been incorporated into the classical Kohn-Sham DFT model and the fractional occupancies has been applied in computations. For instance, the ensemble Kohn-Sham DFT (or the finite-temperature Kohn-Sham DFT) is developed (see, e.g., \cite{kresse1996efficiency}), in which the associated total energy is a nonlinear functional of wavefunctions and pseudo-eigenvalues (or occupation numbers). We see that the ensemble Kohn-Sham DFT can be formulated as a nonlinear eigenvalue problem or a constrained energy minimization problem. It is not difficult to apply the SCF iteration approach for the classical Kohn-Sham DFT model to the ensemble Kohn-Sham DFT model. We understand that some preconditioners have been also constructed to accelerate the SCF iterations \cite{herbst2021blackbox,kresse1996efficiency,lin2013elliptic,zhou2018applicability}. Unfortunately, the convergence of the SCF iterations for the ensemble Kohn-Sham DFT is not guaranteed yet.
	
	In the context of solving the constrained energy minimization problem of the ensemble Kohn-Sham DFT, different from the classical Kohn-Sham DFT, we need to treat the occupation numbers as additional variables. There are more challenges for designing and analyzing an efficient algorithm. For example, we observe that the unitary invariance of the energy functional is not clear and applying the unitary transformation to the Kohn-Sham orbitals may not produce the ground states. We also understand that it is necessary to calculate the Kohn-Sham orbitals exactly \cite{kresse1996efficiency} and it is usually required to choose a good unitary transformation of the wavefunctions when designing an optimization algorithm. We refer to \cite{gillan1989calculation,grumbach1994initio,kresse1996efficiency} for constructing the unitary transformation of the wavefunctions to make energy approximations decay. Ismail-Beigi et al. \cite{ismail-beigi2000new} suggested expressing the unitary transformation as $P=e^{\mathrm{i}B}$ and minimizing the energy functional with respect to the Hermitian matrix $B$. However, the unitary transformation is incorporated into the model when some matrix representations are applied. Marzari et al. \cite{marzari1997ensemble} proposed an optimization algorithm by adopting a matrix representation of the occupation numbers, which we call the occupation matrix, and they got an unitarily invariant functional of wavefunctions by minimizing the occupation matrix. It is shown in \cite{marzari1997ensemble} that it is not necessary to construct the unitary transformation. Later on, Freysoldt et al. \cite{freysoldt2009direct} introduced the so-called pseudo-Hamitonian matrix and proposed a preconditioned conjuagte gradient (PCG) algorithm to minimize the energy functional with respect to the wavefunctions and the pseudo-Hamiltonian matrix, in which the unitary transformation is constructed automatically by minimizing the energy functional with respect to the pseudo-Hamiltonian matrix. Recently, Ulbrich et al. \cite{ulbrich2015proximal} studied a proximal gradient method for the ensemble Kohn-Sham DFT with the Fermi-Dirac smearing. We may refer to \cite{baarman2016direct,ruiz-serrano2013variational} for more works on the direct minimization algorithms for the ensemble Kohn-Sham DFT model. To our knowledge, there is little mathematical analysis on the ensemble Kohn-Sham DFT and its approximations. In this paper, we investigate the energy minimization model of the ensemble Kohn-Sham DFT from a mathematical aspect, and design and analyze the associated optimization algorithms.
	
	The rest of this paper is organized as follows. In the next section, we introduce some basic notation and the energy minimization model of the ensemble Kohn-Sham density functional theory with the pseudo-eigenvalue matrix and the general smearing method. In section \ref{sec:math}, we study the invariance and the exsistence of the minimizer for the ensemble Kohn-Sham energy functional. In section \ref{sec:numerical-approx}, we propose an adaptive double step size strategy and the corresponding preconditioned conjugate gradient (PCG) algorithms to solve the energy minimization problem. Under some mild but reasonable assumptions, we then prove the global convergence of the PCG algorithms based on the adaptive double step size strategy we proposed. We report several numerical experiments in section \ref{sec:numeraical} to demonstrate our theory and show the superiority of our algorithms over the traditional SCF iterations. We give some concluding remarks in section \ref{sec:concluding}. Finally, we provide some details of the gradient of the energy functional in Appendix \ref{appx:gradient} and the derivation process to get the standard Kohn-Sham equation in Appendix \ref{appx:Kohn-Sham}.
	
	\section{Preliminaries}
	\subsection{Basic notation}
	Throughout this paper, we consider periodic systems. Since we usually apply a large enough unit cell when calculating isolated systems, our definitions and conclusions are applicable to the isolated systems in practice. Let $\Omega=\{x_1\xi_1+x_2\xi_2+x_3\xi_3:x_1,x_2,x_3\in[0,1)\}$ be the unit cell, where $\xi_1,\xi_2,\xi_3\in\mathbb{R}^3$ are three non-coplanar vectors. Then the associated Bravais lattice and the reciprocal lattice  are $\mathcal{R}=\{n_1\xi_1+n_2\xi_2+n_3\xi_3: n_1,n_2,n_3\in\mathbb{Z}\}$ and $\mathcal{R}^*=\{m_1\zeta_1+m_2\zeta_2+m_3\zeta_3: m_1,m_2,m_3\in\mathbb{Z}\}$, respectively. Here, $\mathbb{Z}$ represents the set of all integers and
	\[
	\zeta_1=2\pi\frac{\xi_2\times\xi_3}{\xi_1\cdot(\xi_2\times\xi_3)},~\zeta_2=2\pi\frac{\xi_3\times\xi_1}{\xi_2\cdot(\xi_3\times\xi_1)},~\zeta_3=2\pi\frac{\xi_1\times\xi_2}{\xi_3\cdot(\xi_1\times\xi_2)}.
	\]
	
	For $\mathrm{G}\in \mathcal{R}^*$, we denote by $e_{\mathrm{G}} (r)=|\Omega|^{-1/2}e^{\mathrm{i}\mathrm{G}\cdot r}$ the planewave with wavevector $\mathrm{G}$, where $|\Omega|$ is the volume of $\Omega$. The family $\{e_{\mathrm{G}}\}_{G\in\mathcal{R}^*}$ forms an orthonormal basis of the complex valued $\mathcal{R}$-periodic functions space
	\begin{equation}\label{eq:L2C-p}
		L_{\#}^2(\Omega,\mathbb{C})=\left\{\psi\in L_{\mathrm{loc}}^2(\mathbb{R}^3,\mathbb{C}):\psi~\text{is $\mathcal{R}$-periodic}\right\},
	\end{equation}
	and for any $\psi\in L_\#^2(\Omega,\mathbb{C})$,
	\[
	\psi(r)=\sum_{\mathrm{G}\in \mathcal{R}^*}\hat{\psi}_{\mathrm{G}}e_{\mathrm{G}}(r)\quad\text{with}\quad \hat{\psi}_{\mathrm{G}}=\frac{1}{|\Omega|^{\frac12}}\int_{\Omega}\psi(r)e^{-\mathrm{i}\mathrm{G}\cdot r}\textup{d}r.
	\]
	We define the Sobolev space of complex valued $\mathcal{R}$-periodic functions as
	\[
	H_\#^s(\Omega,\mathbb{C})=\left\{\psi\in L^2_{\#}(\Omega,\mathbb{C}):\sum_{\mathrm{G}\in\mathcal{R}^*}(1+|\mathrm{G}|^2)^s|\hat{\psi}_{\mathrm{G}}|^2<\infty\right\}
	\]
	with $s\in\mathbb{R}$, endowed with the inner product
	\[
	(\psi,\phi)_{H_\#^s}=\sum_{\mathrm{G}\in\mathcal{R}^*}(1+|\mathrm{G}|^2)^s\bar{\hat{\psi}}_{\mathrm{G}}\hat{\phi}_{\mathrm{G}},
	\]
	and the induced norm
	\[
	\|\psi\|_{H_\#^s}^2=\sum_{\mathrm{G}\in\mathcal{R}^*}(1+|\mathrm{G}|^2)^s|\hat{\psi}_{\mathrm{G}}|^2.
	\]
	For convenience, unless otherwise specified, $(\cdot,\cdot)$ and $\|\cdot\|$ always represent the inner product and the norm of $L_{\#}^2(\Omega,\mathbb{C})$, respectively.
	
	Let $\Psi=(\psi_1,\ldots,\psi_N)\in (L_{\#}^2(\Omega,{\mathbb{C}}))^N,\Phi=(\phi_1,\ldots,\phi_N)\in (L_{\#}^2(\Omega,\mathbb{C}))^N$. Here $N$ is some positive integer. We can view $\Psi$ and $\Phi$ as vectors with elements being functions. Then we have
	\[
	\Psi\Phi^* = \sum_{i=1}^N \psi_i\bar{\phi}_i,~\Psi^*\Phi = (\bar{\psi}_i\phi_j)_{i,j=1}^{N}.
	\]
	For any $A=(A_{ij})_{i,j=1}^N\in \mathbb{C}^{N\times N}$, we denote by
	\[
	A\Psi^* = \left(\sum_{j=1}^N A_{1j}\bar{\psi}_j, \ldots, \sum_{j=1}^N A_{Nj}\bar{\psi}_j\right)^T,~\Psi A = \left(\sum_{i=1}^N A_{i1}\psi_i,\ldots,\sum_{i=1}^N A_{iN}\psi_i \right).
	\]
	Define
	\[
	\langle \Psi^* \Phi\rangle=((\psi_i,\phi_j ))_{i,j=1}^N\in\mathbb{C}^{N\times N}.
	\]
	
	For any positive integer $n$ and any $\Psi=(\Psi_1,\Psi_2,\ldots,\Psi_n),~\Phi=(\Phi_1,\Phi_2,\ldots,\Phi_n)\in((L^2(\Omega,\mathbb{C}))^N)^n$, we define its inner product as $\langle\Psi,\Phi \rangle=\sum_{i=1}^n\operatorname{tr}\langle\Psi^*_i\Phi_i\rangle$. The induced norm is $\|\Psi\|=\sqrt{\langle\Psi,\Psi\rangle}$. We shall use the notation
	\[
	\|\Psi\|_{\infty}=\max_{i=1,2,\ldots,n}\|\Psi_i\|
	\]
	for convenience.
	
	For any $A=(A_1,A_2,\ldots,A_n),B=(B_1,B_2,\ldots,B_n)\in(\mathbb{C}^{N\times N})^n$, we define its inner product as $\langle A,B\rangle=\sum_{i=1}^n\operatorname{tr}(A_i^* B_i)$. And the induced norm is Frobenius norm, denoted by $\|\cdot\|_F$. We shall use the notation $\|\cdot\|_{sF}$ defined as 
	\begin{equation*}
		\|A\|_{sF}=\min_{c\in\mathbb{C}}\|cI_N-A\|_F,
	\end{equation*}
	where $cI_N-A\coloneqq(cI_N-A_1,cI_N-A_2,\ldots,cI_N-A_n)$. It is easy to obtain
	\begin{equation}\label{eq:c-for-sf}
		\|A\|_{sF}=\left\|\frac{\sum_{i=1}^n\operatorname{tr}A_i}{nN}I_N-A\right\|_F.
	\end{equation}
	Define 
	\[
	\|A\|_{sF,\infty}=\min_{i=1,2,\ldots,n}\|c_A I_N-A_i\|_{F}, 
	\]
	where $c_A = \frac{\sum_{i=1}^n\operatorname{tr}A_i}{nN}$. It is easy to get the following properties for $\|\cdot\|_{sF}$ by \eqref{eq:c-for-sf}.
	\begin{proposition}\label{prop:sF}
		Let $A,B\in(\mathbb{C}^{N\times N})^n$, then the following properties of $\|\cdot\|_{sF}$ hold true:
		\begin{enumerate}
			\item $\|A-B\|_{sF}=0$ if and only if there exists $c\in\mathbb{C}$ such that $A=B+cI_N$;
			\item $\|\cdot\|_{sF}$ satisfies the triangle inequality, i.e., $\|A+B\|_{sF}\le\|A\|_{sF}+\|B\|_{sF}$;
			\item $\|\cdot\|_{sF}$ satisfies the absolute homogeneity, i.e., $\|\alpha A\|_{sF}=|\alpha|\|A\|_{sF}$ for any $\alpha\in\mathbb{C}$;
			\item\label{prop:sF-inner} if $\displaystyle\sum_{i=1}^n\operatorname{tr}A_i=0$, then
			\[
			|\inner{A,B}|\le\|A\|_{sF}\|B\|_{sF}.
			\]
		\end{enumerate}
	\end{proposition}
	It follows from Proposition \ref{prop:sF} and \eqref{eq:c-for-sf} that $\|A\|_{sF}$ is the norm of the following linear space
	\[
	\left\{A=(A_1,A_2,\ldots,A_n)\in(\mathbb{C}^{N\times N})^n:\sum_{i=1}^n\operatorname{tr}A_i=0\right\}.
	\]
	
	The Stiefel manifold is defined by
	$$\mathcal{M}_{\mathcal{B},\mathbb{C}}^N=\{{\Psi}\in (H^1_{\#}(\Omega,\mathbb{C}))^N:\langle\Psi^*\mathcal{B}\Psi\rangle=I_N\},$$
	where $\mathcal{B}$: $(L^2(\Omega,\mathbb{C}))^N\to(L^2(\Omega,\mathbb{C}))^N$ is a bounded and self-adjoint operator. Let
	\[
	\mathcal{O}^{N\times N}_{\mathbb{C}}=\{P\in \mathbb{C}^{N\times N}:P^* P=I_N\},~
	\mathcal{S}^{N\times N}_{\mathbb{C}}=\{A\in \mathbb{C}^{N\times N}:A^* =A\}.
	\] 
	
	If only real values are taken into account, we then remove $\mathbb{C}$ or replace $\mathbb{C}$ with $\mathbb{R}$ and replace the conjugate transpose symbol $*$ by the transpose symbol $T$ in the above notation. We note that the Fourier coefficients of real valued $\mathcal{R}$-periodic functions have some symmetry, more precisely, 
	\begin{equation}\label{eq:HsR-p}
		H_\#^s(\Omega)=\left\{\psi\in H_\#^s(\Omega,\mathbb{C}):\forall \mathrm{G}\in\mathcal{R}^*,~\hat{\psi}_{-\mathrm{G}}=\bar{\hat{\psi}}_{\mathrm{G}}\right\}.
	\end{equation}

	We then introduce some projections of wavefunctions. Let $\Psi\in\mathcal{M}_{\mathcal{B},\mathbb{C}}^N$. We know that the tangent space of $\mathcal{M}_{\mathcal{B},\mathbb{C}}^N$ at $\Psi$ is
	\[
	\mathcal{T}_{\Psi} \mathcal{M}_{\mathcal{B}}^{N}=\{ \Phi\in(H^1_{\#}(\Omega,\mathbb{C}))^N:\langle \Phi^*\mathcal{B}\Psi\rangle+\langle\Psi^*\mathcal{B}\Phi\rangle=0\in\mathbb{C}^{N\times N}\}.
	\]
	Let $$K_{\Psi}=\{\Phi\in(H^1_{\#}(\Omega,\mathbb{C}))^N:\langle \Phi^*\Psi\rangle+\langle\Psi^*\Phi\rangle=0\in\mathbb{C}^{N\times N}\}.$$
	It is clear that $\mathcal{T}_{\Psi} \mathcal{M}_{\mathcal{B},\mathbb{C}}^{N}=K_{\Psi}$ provided $\mathcal{B}=\mathcal{I}$, where $\mathcal{I}$ is the identity operator. For any $\alpha\in\mathbb{R}$, we define the linear operator onto $K_{\Psi}$ by
	\begin{equation}\label{eq:proj-alpha}
		P_{\alpha,\Psi}(\Phi) = (\Phi-\mathcal{B}\Psi\langle\Psi^*\Phi\rangle)+\alpha\mathcal{B}\Psi(\langle\Psi^*\Phi\rangle-\langle\Phi^*\Psi\rangle),~\forall\Phi\in(H^1_{\#}(\Omega,\mathbb{C}))^N.
	\end{equation}
	We see that
	\begin{equation*}
		P_{\alpha,\Psi}^2(\Phi)=P_{\alpha,\Psi}(\Phi)+\alpha(2\alpha-1)\mathcal{B}\Psi(\langle\Psi^*\Phi\rangle-\langle\Phi^*\Psi\rangle),~\forall\Phi\in(H^1_{\#}(\Omega,\mathbb{C}))^N,
	\end{equation*}
	which indicates that $P_{\alpha,\Psi}$ is a projection if and only if $\alpha=0$ or $1/2$. Define
	\begin{equation*}
		P_{\alpha,\Psi}^*(\Phi)=(\Phi-\Psi\langle\Psi^*\mathcal{B}\Phi\rangle)+\alpha\Psi(\langle\Psi^*\mathcal{B}\Phi\rangle-\langle\Phi^*\mathcal{B}\Psi\rangle),~\forall\Phi\in(H^1_{\#}(\Omega,\mathbb{C}))^N.
	\end{equation*}
	We have that for any $\Phi_1,\Phi_2\in (H_{\#}^1(\Omega,\mathbb{C}))^N$, $\inner{P_{\alpha,\Psi}(\Phi_1),\Phi_2}$ and $\inner{\Phi_1,P_{\alpha,\Psi}^*(\Phi_2)}$ have the same real part. Thus $P_{\alpha,\Psi}^*$ is the adjoint operator of $P_{\alpha,\Psi}$ if only real functions are involved. We mention that $P_{0,\Psi}(\Phi)$ is orthogonal to $\Psi$ for any $\Phi\in (H_{\#}^1(\Omega,\mathbb{C}))^N$. 
	
	\subsection{Ensemble Kohn-Sham DFT model for metallic systems}\label{sec:E-KSDFT}
	We consider the ensemble Kohn-Sham density functional theory, in which we adopt the matrix representation of occupations \cite{freysoldt2009direct,marzari1997ensemble}. We see from Bloch's theorem \cite{martin2020electronic} that the kinetic energy and the electronic density are given by the integral over the Brillouin zone (BZ). If BZ sampling is used to discrete the integral over BZ, the ensemble Kohn-Sham energy functional with a general smearing approach can be formulated as 
	\begin{equation}\label{eq:free-energy}
		\mathcal{F}(\Psi,\eta)=\mathcal{E}(\Psi,\eta)-\sigma\sum_{k\in\mathcal{K}}w_{\mathrm{k}}\operatorname{tr}S\left(\frac{1}{\sigma}(\eta_{\mathrm{k}}-\mu I_N)\right)
	\end{equation}
	with wavefunctions $\Psi=(\Psi_{\mathrm{k}})_{\mathrm{k}\in\mathcal{K}}\in ((H^1_{\#}(\mathbb{R}^3,\mathbb{C}))^N)^{|\mathcal{K}|}$ and the pseudo-eigenvalue matrices $\eta=(\eta_{\mathrm{k}})_{\mathrm{k}\in\mathcal{K}}\in(\mathcal{S}_{\mathbb{C}}^{N\times N})^{|\mathcal{K}|}$, where
	\begin{align*}
		\mathcal{E}(\Psi,\eta)&=\sum_{\mathrm{k}\in\mathcal{K}}w_{\mathrm{k}}\operatorname{tr}\left(\left\langle \Psi^*_{\mathrm{k}}\left(-\frac12(\mathrm{i}\mathrm{k}+\nabla)^2+V_{\text{nl}}\right)\Psi_{\mathrm{k}}\right\rangle F_{\eta_{\mathrm{k}}}\right)+\int_{\Omega} V_\text{loc}(r)\rho_{\Psi,\eta}(r)\textup{d}r\\ 
		&\quad+\frac12\int_{\Omega}\int_\Omega\frac{\rho_{\Psi,\eta}(r)\rho_{\Psi,\eta}(r')}{|r-r'|}\textup{d}r \textup{d}r'+\mathcal{E}_{\text{xc}}(\rho_{\Psi,\eta}).
	\end{align*}
	Here $\mathcal{K}$ is a finite subset of BZ, $w_{\mathrm{k}}$ is the weight associated to k-points $\mathrm{k}\in\mathcal{K}$ satisfying 
	 \[
	 \sum_{\mathrm{k}\in\mathcal{K}}w_{\mathrm{k}}=2,
	 \]
	 $N$ is the number of wavefunctions for one k-point, $\sigma=k_B T$ with the Boltzmann constant $k_B$ and the temperature $T$,
	 \[
	 F_{\eta_{\mathrm{k}}} = f\left(\frac{1}{\sigma}(\eta_{\mathrm{k}}-\mu I_N) \right),
	 \]
	 $f$ is a function which is sometimes called the smearing function, and $\mu$ is a function of $\eta$ which will be determined later, $S$ is a function associated to the entropy term. The electronic density $\rho_{\Psi,\eta}$ is
	 \[
	 \rho_{\Psi,\eta}=\sum_{\mathrm{k}\in\mathcal{K}}w_{\mathrm{k}}\operatorname{tr}((\Psi^*_{\mathrm{k}}\Psi_{\mathrm{k}}+\langle\Psi^*_{\mathrm{k}} M \rangle\mathcal{Q}\langle M^*\Psi_{\mathrm{k}}\rangle) F_{\eta_{\mathrm{k}}})
	 \]
	 with $M=(\varphi_1,\ldots,\varphi_{K})\in (L^2_{\#}(\Omega,\mathbb{C}))^{K}$ and the Hermitian-matrix-valued function $\mathcal{Q}=(\mathcal{Q}_{ij})_{i,j=1}^{K}\in(L^2_{\#}(\Omega,\mathbb{C}))^{K\times K}$. Sometimes we shall simply denote $\rho_{\Psi,\eta}$ by $\rho$. $V_{\textup{loc}}\in L^2_{\#}(\Omega,\mathbb{C})$ is the local pseudopotential and $V_\text{nl}$ is the nonlocal pseudopotential defined by $\Psi_{\mathrm{k}}\mapsto V_\text{nl}(\Psi_{\mathrm{k}})=M D\langle M^*\Psi_{\mathrm{k}}\rangle$ with $D\in\mathcal{S}^{K\times K}_{\mathbb{C}}$. Note that the form of \eqref{eq:free-energy} is suitable for the full potential calculations, the pseudopotential approximations \cite{troullier1991efficient,vanderbilt1990soft} and the projector augmented wave (PAW) method \cite{blochl1994projector}. For instance, if the norm-conserving pseudopotential is applied, then $\mathcal{Q}=0$ and $\rho_{\Psi,\eta}=\sum_{\mathrm{k}\in\mathcal{K}}w_{\mathrm{k}}\operatorname{tr}(\Psi^*_{\mathrm{k}}\Psi_{\mathrm{k}})$. In theory, $N$ should be $+\infty$ for the ensemble Kohn-Sham DFT. However, $N$ has to be set to be finite in practice. We require $N>N_b$ where $N_b$ is the number or the half number of electrons. For example, in Quantum ESPRESSO, $N$ is set to $N_b+\lfloor0.2N_b\rfloor$ by default, where $\lfloor x\rfloor$ is the greatest integer not larger than $x$.
	
	Now we address the function $\mu$ of $\eta$ in detail. Assume that $f$ and $S$ satisfy the following properties:
	\begin{enumerate}[leftmargin=40pt,label=A.\Roman*]
		\item\label{asp:fS1} $f$ and $S$ are analytic functions on $\mathbb{R}$ satisfying $S'(x)=xf'(x)$.
		\item\label{asp:fS2} $\lim\limits_{x\to-\infty}f(x)=1$ and $\lim\limits_{x\to+\infty}f(x)=0$.
		\item\label{asp:fS3} $\displaystyle\lim_{x\to+\infty} S(x)$ and $\displaystyle\lim_{x\to-\infty}S(x)$ exist.
		\item\label{asp:fS4} $f$ is strictly monotonically decreasing.
	\end{enumerate}
	Under these assumptions, for given $\eta\in\left(\mathcal{S}_{\mathbb{C}}^{N\times N}\right)^{|\mathcal{K}|}$, there is one and only one $\mu\in\mathbb{R}$ satisfying $\sum\limits_{\mathrm{k}\in\mathcal{K}}w_{\mathrm{k}}\operatorname{tr}F_{\eta_{\mathrm{k}}}=N_e$. Here $N_e$ is the number of electrons. Thus, we choose $\mu$ in \eqref{eq:free-energy} as the unique function of $\eta$ from $\left(\mathcal{S}_{\mathbb{C}}^{N\times N}\right)^{|\mathcal{K}|}$ to $\mathbb{R}$ such that $\sum\limits_{\mathrm{k}\in\mathcal{K}}w_{\mathrm{k}}\operatorname{tr}F_{\eta_{\mathrm{k}}}=N_e$.
	
	We list several possible choices for the smearing function used in the literature.
	\begin{itemize}
		\item the Fermi-Dirac smearing \cite{callaway1984density}:
		\[
		f_{\textup{FD}}(x)=\frac{1}{1+e^x},~S_{\textup{FD}}(x)=-[f_{\textup{FD}}(x)\ln f_{\textup{FD}}(x)+(1-f_{\textup{FD}}(x))\ln(1-f_{\textup{FD}}(x))].
		\]
		\item the Gaussian smearing \cite{elsasser1994densityfunctional,fu1983firstprinciples}:
		\[
		f_{\textup{GS}}(x)=\frac{1}{2}(1-\operatorname{erf}(x)),~S_{\textup{GS}}(x) = \frac{1}{2\sqrt{\pi}} e^{-x^2}.
		\]
		\item the Methfessel-Paxton smearing \cite{methfessel1989highprecision}:
		\[
		f_{\textup{MP},m}(x) = f_{\textup{GS}}(x) + \sum_{i=1}^m A_i H_{2i-1}(x)e^{-x^2},~S_{\textup{MP},m}(x) = \frac12 A_m H_{2m}(x)e^{-x^2},
		\]
		where $H_i$ are the Hermite polynomials (defined as $H_0(x)=1,\,H_{i+1}(x)=2xH_i(x)-H_i'(x)$) and 
		\[
		A_i=\frac{(-1)^i}{i!4^i\sqrt{\pi}}.
		\] 
		\item the Marzari-Vanderbilt smearing \cite{marzari1996abinitio,marzari1999thermal}:
		\[
		f_{\textup{MV}}(x)=f_{\text{GS}}(x)+\frac{1}{4\sqrt{\pi}}\left(-\frac12 aH_2(x)+H_1(x)\right)e^{-x^2},
		\]
		\[
		S_{\textup{MV}}(x)=\frac{1}{4\sqrt{\pi}}\left(-\frac12 H_2(x)+ax^2 H_1(x)\right)e^{-x^2},
		\]
		where $a$ is a free parameter such that $f_{\textup{MV}}(x)$ is nonnegative for any $x\in\mathbb{R}$. Marzari suggests choosing $a=-0.5634$ or $a=-\sqrt{2/3}$ in \cite{marzari1996abinitio}.
	\end{itemize}
		
	We see that the assumptions \ref{asp:fS1}-\ref{asp:fS2} imply the existence of $\mu\in\mathbb{R}$ such that $\sum\limits_{\mathrm{k}\in\mathcal{K}}w_{\mathrm{k}}\operatorname{tr}F_{\eta_{\mathrm{k}}}=N_e$ for any given $\eta\in\left(\mathcal{S}_{\mathbb{C}}^{N\times N}\right)^{|\mathcal{K}|}$. Further, if \ref{asp:fS4} is satisfied, then $\mu$ is unique. Thus, $\mu$ is a function of $\eta$ when the Fermi-Dirac smearing and the Gaussian smearing are applied. But for some other smearing such as the Methfessel-Paxton smearing and the Marzari-Vanderbilt smearing, it is still open whether $\mu$ is unique. In practice, we will always assume that $\mu$ is a function of $\eta$ such that $\sum\limits_{\mathrm{k}\in\mathcal{K}}w_{\mathrm{k}}\operatorname{tr}F_{\eta_{\mathrm{k}}}=N_e$. 
	
	According to the ensemble Kohn-Sham DFT, we solve the following constrained minimization problem
	\begin{equation}\label{opt-eta-sym}
		\inf_{({\Psi},\eta)\in\left(\mathcal{M}_{\mathcal{B},\mathbb{C}}^N\right)^{|\mathcal{K}|}\times\left(\mathcal{S}_{\mathbb{C}}^{N\times N}\right)^{|\mathcal{K}|}}\mathcal{F}({\Psi},\eta)
	\end{equation}
	to obtain the ground state of the system, where $\mathcal{B}$ is an operator defined by $\Psi\mapsto\mathcal{B}\Psi=\Psi + M Q \langle M^*\Psi\rangle$ with $Q=\displaystyle\int_{\Omega}\mathcal{Q}(r)\textup{d}r$. Note that $\mathcal{B}$ is bounded and self-adjoint. The associated Lagrange functional is 
	\begin{equation}\label{Lagrange-multik}
		\mathcal{L}(\Psi,\eta,\Lambda)=\mathcal{F}(\Psi,\eta)-\sum_{\mathrm{k}\in\mathcal{K}}w_{\mathrm{k}}\operatorname{tr}[\Lambda_{\mathrm{k}}^*(\langle \Psi_{\mathrm{k}}^*\mathcal{B}\Psi_{\mathrm{k}}\rangle-I_N)]
	\end{equation}
	with the Lagrange multiplier  $\Lambda=(\Lambda_{\mathrm{k}})_{\mathrm{k}\in\mathcal{K}}\in\left(\mathbb{C}^{N\times N}\right)^{|\mathcal{K}|}$. Note that throughout this paper, since our discussion with respect to $\eta$ is in the linear space $\left(\mathcal{S}_{\mathbb{C}}^{N\times N}\right)^{|\mathcal{K}|}$ over $\mathbb{R}$, there is no term associated with the constraint $\eta\in\left(\mathcal{S}_{\mathbb{C}}^{N\times N}\right)^{|\mathcal{K}|}$ in the Lagrange functional \eqref{Lagrange-multik}.
	
	Assume that the exchange-correction functional $\mathcal{E}_{\text{xc}}$ is differentiable. We regard $\Psi_{\mathrm{k}}$ and $\bar{\Psi}_{\mathrm{k}}$ as two independent variables for all $\mathrm{k}\in\mathcal{K}$ and view $\mathcal{F}$ as a functional of $\Psi$, $\bar{\Psi}$ and $\eta$. Then we get (see Appendix \ref{appx:gradient})
	\[
	\mathcal{F}_{\Psi_{\mathrm{k}}}(\Psi,\eta)=w_{\mathrm{k}}H_{\mathrm{k}}(\rho_{\Psi,\eta})\Psi_{\mathrm{k}} F_{\eta_{\mathrm{k}}}
	\]
	and
	\begin{equation}\label{eq:L_Psi}
	\mathcal{L}_{\Psi_{\mathrm{k}}}(\Psi,\eta,\Lambda)=w_{\mathrm{k}}(H_{\mathrm{k}}(\rho_{\Psi,\eta})\Psi_{\mathrm{k}} F_{\eta_{\mathrm{k}}} - \mathcal{B}\Psi_{\mathrm{k}}\Lambda_{\mathrm{k}}),
	\end{equation}
	where $\mathcal{F}_{\Psi_{\mathrm{k}}}$ and $\mathcal{L}_{\Psi_{\mathrm{k}}}$ are Wirtinger derivatives, 
	\[
	H_{\mathrm{k}}(\rho) = -\frac12(\mathrm{i}\mathrm{k}+\nabla)^2 + \tilde{V}_\text{loc}(\rho)+ \tilde{V}_{\text{nl}}(\rho)
	\]
	with $\displaystyle\tilde{V}_{\text{loc}}(\rho)=V_\text{loc} + \int_{\Omega}\frac{\rho(r)}{|\cdot-r|}\textup{d}r+V_{\text{xc}}(\rho)$, $\tilde{V}_{\text{nl}}(\rho):\Psi_{\mathrm{k}}\mapsto V_{\text{nl}}(\Psi_{\mathrm{k}})+ M\tilde{D}\inner{M^*\Psi_{\mathrm{k}}}$, $\displaystyle V_{\text{xc}}(\rho)=\frac{\delta \mathcal{E}_\text{xc}}{\delta\rho}$, and
	\[
	\tilde{D}=\int_{\Omega}\tilde{V}_{\text{loc}}(\rho)(r)\mathcal{Q}(r)\operatorname{d\!}r\in\mathcal{S}_{\mathbb{C}}^{K \times K}.
	\]
	Here we use the convenient notation $\mathcal{F}(\Psi,\eta)=\mathcal{F}(\Psi,\bar{\Psi},\eta)$ and $\mathcal{L}(\Psi,\eta,\Lambda)=\mathcal{L}(\Psi,\bar{\Psi},\eta,\Lambda)$. Set
	\[
	\nabla_{\Psi_{\mathrm{k}}}\mathcal{F}(\Psi,\eta)=2w_{\mathrm{k}}(H_{\mathrm{k}}(\rho_{\Psi,\eta})\Psi_{\mathrm{k}} - \mathcal{B}\Psi_{\mathrm{k}}\langle\Psi_{\mathrm{k}}^* H(\rho_{\Psi,\eta})\Psi_{\mathrm{k}}\rangle)F_{\eta_{\mathrm{k}}}
	\]
	and $\nabla_{\Psi}\mathcal{F}=(\nabla_{\Psi_{\mathrm{k}}}\mathcal{F})_{\mathrm{k}\in\mathcal{K}}$. Given $\eta$, we denote by $\nabla_{\eta_{\mathrm{k}}}\mathcal{F}=\mathcal{F}_{\eta_{\mathrm{k}}}^T$ and  $\nabla_\eta\mathcal{F}=(\nabla_{\eta_{\mathrm{k}}}\mathcal{F})_{\mathrm{k}\in\mathcal{K}}$, where
	\[
	\mathcal{F}_{\eta_{\mathrm{k}}}=\left(\frac{\partial\mathcal{F}}{\partial\eta_{\mathrm{k}ij}}\right)_{i,j=1}^N.
	\]
	When all $\eta_{\mathrm{k}}$ are diagonal matrices, $\displaystyle\frac{\partial\mathcal{F}}{\partial\eta_{\mathrm{k}ij}}$ is given by
	\begin{align*}
		\frac{\partial \mathcal{F}}{\partial\eta_{\mathrm{k}ij}}&=w_{\mathrm{k}}\bigg( (\langle\psi_{\mathrm{k}i},H_{\mathrm{k}}(\rho_{\Psi,\eta})\psi_{\mathrm{k}i}\rangle-\epsilon_{\mathrm{k}i})\frac{1}{\sigma}f'\left(\frac{\epsilon_{\mathrm{k}i}-\mu}{\sigma}\right)\delta_{ij}\\
		&\quad-\frac{ f'\left(\frac{\epsilon_{\mathrm{k}'i}-\mu}{\sigma}\right)\delta_{ij}}{\sum_{\mathrm{k}'}w_{\mathrm{k}'}\sum_{i'=1}^N f'\left(\frac{\epsilon_{\mathrm{k}'i'}-\mu}{\sigma}\right)} d_{\mu}\\
		&\quad +\langle\psi_{\mathrm{k}j},H(\rho_{\Psi,\eta})\psi_{\mathrm{k}i}\rangle\frac{f_{\mathrm{k}j}-f_{\mathrm{k}i}}{\epsilon_{\mathrm{k}j}-\epsilon_{\mathrm{k}i}}(1-\delta_{ij})\bigg)
	\end{align*}
	for any $\mathrm{k}\in\mathcal{K}$, where $\Psi_{\mathrm{k}}=(\psi_{\mathrm{k}1},\psi_{\mathrm{k}2},\ldots,\psi_{\mathrm{k}N})$, $\eta_{\mathrm{k}}=\operatorname{Diag}(\epsilon_{\mathrm{k}1},\epsilon_{\mathrm{k}2},\ldots,\epsilon_{\mathrm{k}N})$, $f_{\mathrm{ki}}=f((\epsilon_{\mathrm{k}i}-\mu)/\sigma)$, $\frac{f_{\mathrm{k}j}-f_{\mathrm{k}i}}{\epsilon_{\mathrm{k}j}-\epsilon_{\mathrm{k}i}}=\frac{1}{\sigma}f\left(\frac{\epsilon_{\mathrm{k}i}-\mu}{\sigma}\right)$
	provided $\epsilon_{\mathrm{k}j}=\epsilon_{\mathrm{k}i}$,
	\begin{equation*}
		d_{\mu}=\sum_{\mathrm{k}'\in\mathcal{K}}w_{\mathrm{k}'}\sum_{i'=1}^N (\langle\psi_{\mathrm{k}'i'},H_{\mathrm{k}'}(\rho_{\Psi,\eta})\psi_{\mathrm{k}'i'}\rangle-\epsilon_{\mathrm{k}'i'})\frac{1}{\sigma}f'\left(\frac{\epsilon_{\mathrm{k}'i'}-\mu}{\sigma}\right).
	\end{equation*}
	It is clear that $\mathcal{L}_{\Psi_\mathrm{k}}(\Psi,\eta,\Lambda)=0$ and $\mathcal{L}_{\eta_{\mathrm{k}}}(\Psi,\eta,\Lambda)=0$ for all $\mathrm{k}\in\mathcal{K}$ mean that $\nabla_{\Psi}\mathcal{F}(\Psi,\eta)=0$ and $\nabla_{\eta}\mathcal{F}(\Psi,\eta)=0$. And $\nabla_{\Psi}\mathcal{F}(\Psi,\eta)=0$ and $\nabla_{\eta}\mathcal{F}(\Psi,\eta)=0$ mean that there exists some $\Lambda$ such that $\mathcal{L}_{\Psi_\mathrm{k}}(\Psi,\eta,\Lambda)=0$ and $\mathcal{L}_{\eta_{\mathrm{k}}}(\Psi,\eta,\Lambda)=0$ for all $\mathrm{k}\in\mathcal{K}$. As for the classical Kohn-Sham DFT model, let $\mathcal{L}_{\Psi}(\Phi,\eta,\Lambda)=0$ and $\mathcal{L}_{\eta}(\Phi,\eta,\Lambda)=0$, we will obtain the standard Kohn-Sham equation (see Appendix \ref{appx:Kohn-Sham} for details).
	
	\section{Mathematical analysis}\label{sec:math}
	In this section, we investigate some basic mathematical properties of the ensemble Kohn-Sham DFT model, including the invariance and the existence of the minimizer of the energy functional.
	\subsection{Invariance}\label{subsec:invariance}
	We first have the following invariance of the energy functional.
	\begin{theorem}
		For any $c\in\mathbb{R}$, $(\Psi,\eta)\coloneqq(\Psi_{\mathrm{k}},\eta_{\mathrm{k}})_{\mathrm{k}\in\mathcal{K}}\in ((H_\#^1(\Omega,\mathbb{C}))^{N})^{|\mathcal{K}|}\times\left(\mathcal{S}_{\mathbb{C}}^{N\times N}\right)^{|\mathcal{K}|}$ and $P\coloneqq(P_{\mathrm{k}})_{\mathrm{k}\in\mathcal{K}}\in\left(\mathcal{O}^{N\times N}_{\mathbb{C}}\right)^{|\mathcal{K}|}$, there holds
		\begin{equation}\label{eq:shift-unitarily-invariant}
			\mathcal{F}(\Psi P,P^*(\eta+c I_N) P)=\mathcal{F}(\Psi,\eta),
		\end{equation}
		where $\Psi P=(\Psi_{\mathrm{k}}P_{\mathrm{k}})_{\mathrm{k}\in\mathcal{K}}$, $P^*\eta P=(P_{\mathrm{k}}^*\eta_{\mathrm{k}} P_{\mathrm{k}})_{\mathrm{k}\in\mathcal{K}}$.
	\end{theorem}
	\begin{proof}
		It is sufficient to prove that
		\begin{gather}
			\mathcal{F}(\Psi,\eta+cI_N)=\mathcal{F}(\Psi,\eta),\label{eq:shift-invariant}\\
			\mathcal{F}(\Psi P,P^*\eta P)=\mathcal{F}(\Psi,\eta)\label{eq:unitarily-invariant}
		\end{gather}
		hold true for any $c\in\mathbb{R}$, $(\Psi,\eta)\in ((H_\#^1(\Omega,\mathbb{C}))^{N})^{|\mathcal{K}|}\times\left(\mathcal{S}_{\mathbb{C}}^{N\times N}\right)^{|\mathcal{K}|}$ and $P\in\left(\mathcal{O}^{N\times N}_{\mathbb{C}}\right)^{|\mathcal{K}|}$.
	
		We first prove the equation \eqref{eq:shift-invariant}. By the uniqueness of $\mu$ that $\sum\limits_{\mathrm{k}\in\mathcal{K}}w_{\mathrm{k}}\operatorname{tr}F_{\eta_{\mathrm{k}}}=N_e$, we obtain $\mu(\eta+cI_N)=\mu(\eta)+c$ for any $c\in\mathbb{R}$. Thus, we have $(F_{\eta_{\mathrm{k}}+cI_N})_{\mathrm{k}\in\mathcal{K}}=(F_{\eta_{\mathrm{k}}})_{\mathrm{k}\in\mathcal{K}}$ and
		\begin{equation*}
			\left(S\left(\frac{1}{\sigma}(\eta_{\mathrm{k}}+cI_N-\mu(\eta+cI_N) I_N)\right)\right)_{\mathrm{k}\in\mathcal{K}}=\left(S\left(\frac{1}{\sigma}(\eta_{\mathrm{k}}-\mu(\eta) I_N)\right)\right)_{\mathrm{k}\in\mathcal{K}},
		\end{equation*}
		which lead to $\rho_{\Psi,\eta+cI_N}=\rho_{\Psi,\eta}$ and arrive at \eqref{eq:shift-invariant}.
		
		Next we prove the equation \eqref{eq:unitarily-invariant}. Since $f$ and $S$ are analytic on $\mathbb{R}$, we have
		\[
		P_{\mathrm{k}}f(\eta_{\mathrm{k}})P_{\mathrm{k}}^*=f(P_{\mathrm{k}}\eta_{\mathrm{k}} P_{\mathrm{k}}^*),~P_{\mathrm{k}} S(\eta_{\mathrm{k}})P^*_{\mathrm{k}}=S(P_{\mathrm{k}}\eta_{\mathrm{k}} P^*_{\mathrm{k}}).
		\]
		By the uniqueness of $\mu$ that $\sum\limits_{\mathrm{k}\in\mathcal{K}}w_{\mathrm{k}}\operatorname{tr}F_{\eta_{\mathrm{k}}}=N_e$, we get $\mu(P^*\eta P)=\mu(\eta)$ for any $P\in(\mathcal{O}_{\mathbb{C}}^{N\times N})^{\mathcal{K}}$. Note that 
		\begin{align*}
			\rho_{\Psi P,\eta}&=\sum_{\mathrm{k}\in\mathcal{K}}w_{\mathrm{k}}\operatorname{tr}(P_{\mathrm{k}}^*(\Psi^*_{\mathrm{k}}\Psi_{\mathrm{k}}+\langle\Psi^*_{\mathrm{k}} M\rangle\mathcal{Q}\langle M^*\Psi_{\mathrm{k}}\rangle)P_{\mathrm{k}} F_{\eta_{\mathrm{k}}})\\
			&=\sum_{\mathrm{k}\in\mathcal{K}}w_{\mathrm{k}}\operatorname{tr}((\Psi^*_{\mathrm{k}}\Psi_{\mathrm{k}}+\langle\Psi^*_{\mathrm{k}} M\rangle\mathcal{Q}\langle M^*\Psi_{\mathrm{k}}\rangle) F_{P_{\mathrm{k}}\eta_{\mathrm{k}}P_{\mathrm{k}}^*})\\
			&=\rho_{\Psi,P\eta P^*}.
		\end{align*}
		We have
		\begin{align*}
			\mathcal{F}(\Psi P,\eta)&=\sum_{\mathrm{k}\in\mathcal{K}}w_{\mathrm{k}}\operatorname{tr}\left(\left\langle (\Psi_{\mathrm{k}} P_{\mathrm{k}})^*\left(-\frac12\Delta+V_{\textup{nl}}\right)(\Psi_{\mathrm{k}} P_{\mathrm{k}})\right\rangle F_{\eta_{\mathrm{k}}}\right)\\
			&\quad+\int_{\Omega} V_\text{loc}(r)\rho_{\Psi P,\eta}(r) dr+\mathcal{E}_{\text{HXC}}(\rho_{\Psi P,\eta})-\sigma\sum_{\mathrm{k}\in\mathcal{K}}w_{\mathrm{k}}\operatorname{tr} P_{\mathrm{k}} S\left(\frac{1}{\sigma}(\eta_{\mathrm{k}}-\mu I)\right)P_{\mathrm{k}}^*\\
			&=\sum_{\mathrm{k}\in\mathcal{K}}w_{\mathrm{k}}\operatorname{tr}\left(\left\langle \Psi_{\mathrm{k}}^*\left(-\frac12(i\mathrm{k}+\nabla)^2+V_{\textup{nl}}\right)\Psi_{\mathrm{k}}\right\rangle F_{P_{\mathrm{k}}\eta_{\mathrm{k}} P_{\mathrm{k}}^*}\right)\\
			&\quad+\int_{\mathbb{R}^3} V_\text{loc}(r)\rho_{\Psi,P\eta P^*}(r)dr+\mathcal{E}_{\text{HXC}}(\rho_{\Psi,P\eta P^*})-\sigma\sum_{\mathrm{k}\in\mathcal{K}}w_{\mathrm{k}}\operatorname{tr} S\left(\frac{1}{\sigma}(P_{\mathrm{k}}\eta_{\mathrm{k}}P_{\mathrm{k}}^*-\mu I)\right),
		\end{align*}
		where
		\[
		\mathcal{E}_{\textup{HXC}}(\rho_{\Psi,\eta})=\frac12\int_{\Omega}\int_{\Omega}\frac{\rho_{\Psi,\eta}(r)\rho_{\Psi,\eta}(r')}{|r-r'|}\textup{d}r \textup{d}r'+\mathcal{E}_{\text{xc}}(\rho_{\Psi,\eta}),
		\]
		namely,
		\begin{equation}\label{eq:U-tran}
			\mathcal{F}(\Psi P,\eta)=\mathcal{F}(\Psi,P\eta P^*).
		\end{equation}
		Finally we obtain from \eqref{eq:U-tran} that
		\[
		\mathcal{F}(\Psi P,P^*\eta P)=\mathcal{F}(\Psi,P(P^*\eta P)P^*)=\mathcal{F}(\Psi,\eta).
		\]
	\end{proof}

	We may view \eqref{eq:shift-invariant} as the translation invariance and \eqref{eq:unitarily-invariant} as the quasi unitary invariance.
	
	We obtain from \eqref{eq:shift-unitarily-invariant} that
	\begin{equation}\label{eq:eta-equal-diag}
		\inf_{(\Psi,\eta)\in\left(\mathcal{M}_{\mathcal{B},\mathbb{C}}^{N}\right)^{|\mathcal{K}|}\times \left(\mathcal{D}^{N\times N}\right)^{|\mathcal{K}|}}\mathcal{F}({\Psi},\eta)=\inf_{(\Psi,\eta)\in\left(\mathcal{M}_{\mathcal{B},\mathbb{C}}^{N}\right)^{|\mathcal{K}|}\times \left(\mathcal{S}^{N\times N}_{\mathbb{C}}\right)^{|\mathcal{K}|}}\mathcal{F}({\Psi},\eta),
	\end{equation}
	where $\mathcal{D}^{N\times N}=\{A\in\mathbb{R}^{N\times N}:A\text{~is a diagonal matrix}\}$. We see that
	\[\inf\limits_{(\Psi,\eta)\in\left(\mathcal{M}_{\mathcal{B},\mathbb{C}}^{N}\right)^{|\mathcal{K}|}\times \left(\mathcal{D}^{N\times N}\right)^{|\mathcal{K}|}}\mathcal{F}({\Psi},\eta)
	\]
	is the original ensemble Kohn-Sham DFT model, which means that the model \eqref{opt-eta-sym} is equivalent to the original ensemble Kohn-Sham DFT model.
	
	We see from \eqref{eq:shift-unitarily-invariant} that the solution of \eqref{opt-eta-sym} is not unique. Thus we may turn to consider the following optimization problem
	\begin{equation}\label{optimisation-eta-hermitian-optimisation}
		\inf_{[{\Psi},\eta]\in\left(\mathcal{M}_{\mathcal{B},\mathbb{C}}^N\right)^{|\mathcal{K}|}\times\left(\mathcal{S}_{\mathbb{C}}^{N\times N}\right)^{|\mathcal{K}|}\big/\sim}\mathcal{F}({\Psi},\eta)
	\end{equation}
	which is equivalent to \eqref{opt-eta-sym}. Here $\sim$ denotes the equivalence relation defined as follows: $(\Psi,\eta)\sim(\Psi',\eta')$ if and only if there exist $P\in\left(\mathcal{O}^{N\times N}_{\mathbb{C}}\right)^{|\mathcal{K}|}$ and $c\in\mathbb{R}$ such that
	\[
		\begin{pmatrix}
			\Psi'&\\&\eta'
		\end{pmatrix}=\begin{pmatrix}
			1&\\ &P^*
		\end{pmatrix}\begin{pmatrix}
			\Psi&\\ &\eta+cI_N
		\end{pmatrix}\begin{pmatrix}
			P&\\ &P
		\end{pmatrix}.
	\]
	Therefore, the equivalence class $[\Psi,\eta]$ is
	\[
		[\Psi,\eta]=\{(\Psi P,P^*(\eta+cI_N) P):P\in\left(\mathcal{O}^{N\times N}_{\mathbb{C}}\right)^{|\mathcal{K}|},\,c\in\mathbb{R}\}.
	\]
	
	Let $P\in\left(\mathcal{O}^{N\times N}_{\mathbb{C}}\right)^{|\mathcal{K}|}$ and 
	\[
		\eta_{\mathrm{k}}=\operatorname{Diag}(\epsilon_{\mathrm{k}1} I_{N_{\mathrm{k}1}},\epsilon_{\mathrm{k}2} I_{N_{\mathrm{k}2}},\ldots,\epsilon_{\mathrm{k}d_{\mathrm{k}}} I_{N_{\mathrm{k}d_{\mathrm{k}}}})_{N\times N},~\forall \mathrm{k}\in\mathcal{K},
	\]
	then $(\Psi P,\eta)\sim(\Psi, \eta)$ if and only if  $P_{\mathrm{k}}$ has the same block structure with $\eta_{\mathrm{k}}$ for any $\mathrm{k}\in\mathcal{K}$
	\[
		P_{\mathrm{k}}=\operatorname{Diag}(P_{\mathrm{k}1},P_{\mathrm{k}2},\ldots,P_{\mathrm{k}d_{\mathrm{k}}})_{N\times N},~P_{\mathrm{k}i}\in\mathcal{O}^{N_{\mathrm{k}i}\times N_{\mathrm{k}i}}_{\mathbb{C}}.
	\]
	If $\eta=(I_N)_{k\in\mathcal{K}}$ is fixed, then $F_{\eta_{\mathrm{k}}}=I_N$ and $(\Psi P,\eta)\sim(\Psi, \eta)$ for any $P\in\left(\mathcal{O}^{N\times N}_{\mathbb{C}}\right)^{|\mathcal{K}|}$, i.e., the energy functional is unitarily invariant. It is nothing but the classical Kohn-Sham DFT model.
	
	Similarly, for the gradient of $\mathcal{F}$, we have the following theorem.
	\begin{theorem}\label{prop:grad}
		Given $c\in\mathbb{R}$, $(\Psi,\eta)\in ((H_\#^1(\Omega,\mathbb{C}))^{N})^{|\mathcal{K}|}\times\left(\mathcal{S}_{\mathbb{C}}^{N\times N}\right)^{|\mathcal{K}|}$, and $P\in\left(\mathcal{O}^{N\times N}_{\mathbb{C}}\right)^{|\mathcal{K}|}$.
		\begin{enumerate}
			\item There hold
			\begin{equation}\label{eq:grad-invariant}
				\begin{aligned}
					\mathcal{F}_{\Psi}(\Psi P,P^*(\eta+cI_N) P)&=\mathcal{F}_{\Psi}(\Psi,\eta)P,\\
					\nabla_{\Psi}\mathcal{F}(\Psi P,P^*(\eta+cI_N) P)&=\nabla_\Psi\mathcal{F}(\Psi,\eta)P,\\ \nabla_\eta\mathcal{F}(\Psi P,P^*(\eta+cI_N) P)&=P^*\nabla_\eta\mathcal{F}(\Psi,\eta)P;	
				\end{aligned}
			\end{equation}
			\item $\nabla_{\eta_{\mathrm{k}}}\mathcal{F}(\Psi,\eta)$ is Hermitian matrix for any $\mathrm{k}\in\mathcal{K}$;
			\item\label{prop:grad-trace0} $\sum_{\mathrm{k}\in\mathcal{K}}\operatorname{tr}\nabla_{\eta_{\mathrm{k}}}\mathcal{F}(\Psi,(\eta+cI_N))=0$.
		\end{enumerate}
	\end{theorem}

	The first property tells us how to apply unitary transformations to $\Psi$, $\eta$ and the associated gradients consistently. The third property is the another description of the translation invariance of $\mathcal{F}$ with respect to $\eta$ and will be used in our convergence analysis.
	
	\subsection{Existence of the minimizer}\label{subsec:minimizer}
	In this subsection, we show the existence of the minimizer of the ensemble Kohn-Sham DFT model. We consider that the sampling of k-points is at $\Gamma$ point only, for which $\Psi$, $\eta$ and other corresponding functions and spaces are of real valued. For the general sampling $\mathcal{K}$, the existence of the minimizer of the ensemble Kohn-Sham DFT model is still open.
	
	Following \cite{chen2013numerical}, we assume that $\mathcal{E}_{\text{xc}}$ is of the form
	\[
	\mathcal{E}_{\text{xc}}(\rho)=\int_{\Omega}\mathcal{N}(\rho)(r)\textup{d}r
	\]
	and
	\begin{equation}\label{eq:N-prop}
		\mathcal{N}\in\mathscr{P}(3,(c_1,c_2))\,(c_1\ge 0) \text{~or~} \mathcal{N}\in\mathscr{P}(4/3,(c_1,c_2),
	\end{equation}
	where
	\[
	\mathscr{P}\left(p,\left(c_{1}, c_{2}\right)\right)=\left\{f: \exists a_{1}, a_{2} \in \mathbb{R} \text { such that } c_{1} t^{p}+a_{1} \leq f(t) \leq c_{2} t^{p}+a_{2} \quad \forall t \geq 0\right\}
	\]
	with $c_1\in\mathbb{R}$ and $p,c_2\in[0,\infty)$. We assume that there exists a constant $\alpha>0$ such that for any $\psi\in L^2_{\#}(\Omega)$, the following inequality holds:
	\begin{equation}\label{ineq:S-coer}
		(\psi, \mathcal{B}\psi) \ge \alpha \|\psi\|^2.
	\end{equation}
	We also assume that the assumptions \ref{asp:fS1}-\ref{asp:fS4} are satisfied. Let
	\[
	\mathscr{F}_{\textup{occ}}=\{F=\operatorname{Diag}(f_1,f_2,\ldots,f_N)\in\mathcal{D}^{N\times N}:2\sum_{i=1}^N f_i=N_e, f_i\in(0,1), i=1,2,\ldots,N\}.
	\]
	Obviously,
	\[
	\overline{\mathscr{F}}_{\textup{occ}}=\{F=\operatorname{Diag}(f_1,f_2,\ldots,f_N)\in\mathcal{D}^{N\times N}:2\sum_{i=1}^N f_i=N_e, f_i\in[0,1], i=1,2,\ldots,N\}.
	\]
	
	We first have the following lemma.
	\begin{lemma}\label{lem:energy-equal-occ}
		There holds
		\[
		\inf_{({\Psi},\eta)\in\mathcal{M}_{\mathcal{B}}^N\times\mathcal{S}^{N\times N}}\mathcal{F}({\Psi},\eta)=\inf_{(\Psi,F)\in\mathcal{M}_{\mathcal{B}}^N\times\mathscr{F}_{\textup{occ}}}\widetilde{\mathcal{F}}({\Psi},F),
		\]
		where $\widetilde{\mathcal{F}}(\Psi,F)=\widetilde{\mathcal{E}}(\Psi,F)-\sigma\operatorname{tr}(S\circ f^{-1})(F)$, 
		\begin{equation*}
			\widetilde{\mathcal{E}}(\Psi,F)=\operatorname{tr}\left(\left\langle \Psi^T\left(-\frac12\Delta+V_{\text{ext}}\right)\Psi\right\rangle F\right)+\mathcal{E}_{\textup{HXC}}(\tilde{\rho}_{\Psi,F})
		\end{equation*}
		with $\displaystyle\tilde{\rho}_{\Psi,F}=2\operatorname{tr}((\Psi^T\Psi+\langle\Psi^T M\rangle\mathcal{Q}\langle M^T\Psi\rangle) F)$. 
	\end{lemma}
	\begin{proof}
		Let $(\Psi,\eta)\in\mathcal{M}_{\mathcal{B}}^N\times\mathcal{D}^{N\times N}$. We have
		\begin{equation*}
			\mathcal{F}(\Psi,\eta)=\widetilde{\mathcal{F}}(\Psi,F_\eta),
		\end{equation*}
		which together with \eqref{eq:eta-equal-diag} yields the conclusion.
	\end{proof}
	
	Let $f(-\infty)=1,\,f(+\infty)=0$ and $S(-\infty)=\displaystyle\lim_{x\to-\infty} S(x),\,S(+\infty)=\displaystyle\lim_{x\to+\infty}S(x)$, then $f$ and $S$ are continuous on $[-\infty,+\infty]$ and $f([-\infty,\infty])=[0,1]$. Thus $S\circ f^{-1}$ is continuous on $[0,1]$. By Lemma \ref{lem:energy-equal-occ}, instead of $\inf\limits_{(\Psi,F)\in\mathcal{M}_{\mathcal{B}}^N\times\mathcal{S}^{N\times N}}\mathcal{F}({\Psi},\eta)$, we consider the following minimization problem
	\begin{equation}\label{eq:occ-closure}
		\inf_{(\Psi,F)\in\mathcal{M}_{\mathcal{B}}^N\times\overline{\mathscr{F}}_{\textup{occ}}}\widetilde{\mathcal{F}}({\Psi},F).
	\end{equation}
	We shall prove that $\widetilde{\mathcal{F}}$ does indeed have a minimizer on $\mathcal{M}_{\mathcal{B}}^N\times\overline{\mathscr{F}}_{occ}$. Let
	\begin{equation*}
		\widetilde{\widetilde{\mathcal{E}}}(\Psi)=\operatorname{tr}\left(\left\langle \Psi^T\left(-\frac12\Delta+V_{\text{ext}}\right)\Psi\right\rangle \right)+\frac12\int_{\mathbb{R}^3}\frac{\rho_{\Psi}(r)\rho_{\Psi}(r')}{|r-r'|}\textup{d}r\textup{d}r'+\mathcal{E}_{\text{xc}}(\rho_{\Psi}),
	\end{equation*}
	where $\rho_\Psi =2\operatorname{tr}(\Psi^T\Psi+\langle\Psi^T M \rangle\mathcal{Q}\langle M^T\Psi\rangle)$. Then we have
	\begin{equation}\label{eq:noocc.eq.occ}
		\widetilde{\widetilde{\mathcal{E}}}(\Psi F^{1/2})=\widetilde{\mathcal{E}}(\Psi,F), \forall (\Psi,F)\in\mathcal{M}_{\mathcal{B}}^N\times\overline{\mathscr{F}}_{\textup{occ}}.
	\end{equation}
	
	To prove $\widetilde{\mathcal{F}}$ has a minimizer on $\mathcal{M}_{\mathcal{B}}^N\times\overline{\mathscr{F}}_{\textup{occ}}$, we need the lower semi-continuity of $\widetilde{\widetilde{\mathcal{E}}}$ in the weak topology of $(H_{\#}^1(\Omega))^N$ (See, e.g., \cite{chen2013numerical,chen2010numerical}). 
	\begin{proposition}\label{prop:weakcvg}
		Suppose \eqref{eq:N-prop} holds. If $\Psi^{(n)}$ converges weakly to $\Psi$ in $(H_{\#}^1(\Omega))^N$, then
		\begin{equation*}
			\widetilde{\widetilde{\mathcal{E}}}(\Psi)\le\varliminf_{n\to\infty}\widetilde{\widetilde{\mathcal{E}}}(\Psi^{(n)}).
		\end{equation*}
	\end{proposition}
	
	Using \eqref{ineq:S-coer}, Jensen's inequality and the similar arguments in \cite{chen2010numerical}, we get that $\widetilde{\mathcal{E}}(\Psi,F)$ is bounded below over $\mathcal{M}_{\mathcal{B}}^N\times\overline{\mathscr{F}}_{\text{occ}}$.
	\begin{proposition}\label{prop:boundedbelow}
		If \eqref{eq:N-prop} and \eqref{ineq:S-coer} hold, then there exist constants $C>0$ and $b>0$ such that
		\begin{equation*}
			\widetilde{\mathcal{E}}(\Psi,F)\ge C^{-1}\sum_{i=1}^N \|\Psi F^{1/2}\|_{H_{\#}^1}^2-b\quad\forall (\Psi,F)\in\mathcal{M}_{\mathcal{B}}^N\times\overline{\mathscr{F}}_{\textup{occ}}.
		\end{equation*}
	\end{proposition}
	
	Finally, we obtain the existence of a minimizer for \eqref{eq:occ-closure}.
	\begin{theorem}
		If \eqref{eq:N-prop}, \eqref{ineq:S-coer} and the assumptions \ref{asp:fS1}-\ref{asp:fS4} hold, then there exits $(\Phi_*,F_*)\in\mathcal{M}_{\mathcal{B}}^N\times\overline{\mathscr{F}}_{\textup{occ}}$ such that 
		\begin{equation*}
			\widetilde{\mathcal{F}}(\Phi_*,F_*) = \inf\limits_{(\Psi,F)\in\mathcal{M}_{\mathcal{B}}^N\times\overline{\mathscr{F}}_{\textup{occ}}}\widetilde{\mathcal{F}}({\Psi},F).
		\end{equation*}
	\end{theorem}
	\begin{proof}
		Let $\displaystyle \alpha = \inf_{(\Psi,F)\in\mathcal{M}_{\mathcal{B}}^N\times\overline{\mathscr{F}}_{occ}}\widetilde{\mathcal{F}}({\Psi},F)$. It follows from Proposition \ref{prop:boundedbelow} and $S([-\infty,+\infty])$ being bounded that $\alpha>-\infty$. It is clear that $\alpha<\infty$.
		
		Choose $\Psi^{(n)}=(\psi_1^{(n)},\ldots,\psi_N^{(n)})\in\mathcal{M}_{\mathcal{B}}^N$ and $F^{(n)}=\operatorname{Diag}(f_1^{(n)},\ldots,f_N^{(n)})\in\overline{\mathscr{F}}_{occ}$ such that 
		\begin{equation*}
			\lim_{n\to\infty}\widetilde{\mathcal{F}}(\Psi^{(n)},F^{(n)})=\alpha.
		\end{equation*}
		We then get from Proposition \ref{prop:boundedbelow} that $\Psi^{(n)}(F^{(n)})^{1/2}$ is uniformly bounded in $(H_{\#}^1(\Omega))^N$. We derive from Kakutani's Theorem (see Theorem 4.2 in page 132 of \cite{conway2007course}) that there exists a weakly convergent subsequence of $\Psi^{(n)}(F^{(n)})^{1/2}$ in $(H_{\#}^1(\Omega))^N$. Without loss of generality, let 
		\[
		\Psi^{(n)}(F^{(n)})^{1/2}\rightharpoonup \Psi_*=(\psi_{*,1},\ldots,\psi_{*,N})\quad \text{in } (H_{\#}^1(\Omega))^N,
		\]
		where $\Psi_*\in (H_{\#}^1(\Omega))^N$. Since $(H_{\#}^1(\Omega))^N$ is compactly embedded into $L^2_{\#}(\Omega)$, we see that $\Psi^{(n)}(F^{(n)})^{1/2}\to\Psi_*$ strongly in $L^2_{\#}(\Omega)$ as $n\to\infty$. Let $F_*=\langle\Psi_*^T\Psi_*\rangle$. We have
		\begin{equation}\label{cvg:occ}
			F^{(n)}=\langle(\Psi^{(n)}(F^{(n)})^{1/2})^T\Psi^{(n)}(F^{(n)})^{1/2}\rangle\to F_*,
		\end{equation}
		which shows $F_*\in\overline{\mathscr{F}}_{occ}$ and that there exists $\Phi_*\in\mathcal{M}_{\mathcal{B}}^N$ such that $\Phi_*F_*^{1/2}=\Psi_*$.  From \eqref{eq:noocc.eq.occ}, \eqref{cvg:occ}, and Proposition \ref{prop:weakcvg}, we obtain
		\begin{align*}
			\widetilde{\mathcal{F}}(\Phi_*,F_*)&=\widetilde{\widetilde{\mathcal{E}}}(\Psi_*(F_*)^{1/2})-\sigma\operatorname{tr}(S\circ f^{-1})(F_*)\\
			&\le \varliminf_{n\to\infty}\widetilde{\widetilde{\mathcal{E}}}(\Psi^{(n)}(F^{(n)})^{1/2})+\varliminf_{n\to\infty}\left(-\sigma\operatorname{tr}(S\circ f^{-1})(F^{(n)})\right)\\
			&\le \varliminf_{n\to\infty}\left(\widetilde{\widetilde{\mathcal{E}}}(\Psi^{(n)}(F^{(n)})^{1/2})-\sigma\operatorname{tr}(S\circ f^{-1})(F^{(n)})\right)\\
			&=\varliminf_{n\to\infty}\widetilde{\mathcal{F}}(\Phi^{(n)},F^{(n)})\\
			&=\alpha.
		\end{align*}
		This completes the proof.
	\end{proof}
	
	\section{Numerical approximations}\label{sec:numerical-approx}
	We apply the planewave method to discrete \eqref{opt-eta-sym}. For any $\mathrm{k}\in\mathcal{K}$, let
	\begin{equation*}
	V_{\mathrm{k},N_{G}}=\operatorname{span}\left\{e_{\mathrm{G}}:\mathrm{G} \in \mathcal{R}^*, \frac{1}{2}\left| \mathrm{k}+\mathrm{G}\right|^{2} \leq E_{\mathrm{cut}}\right\},
	\end{equation*}
	where $E_{\text{cut}}$ is a given cutoff energy, $N_G$ is the largest number of planewaves among $\mathrm{k}\in\mathcal{K}$. Consequently, a finite planewave discretization of the ensemble Kohn-Sham DFT minimization problem \eqref{opt-eta-sym} is as follows
	\begin{equation}\label{eq:d-opt-eta-sym}
		\inf_{({\Psi},\eta)\in\left(\prod\limits_{\mathrm{k}\in\mathcal{K}}\mathcal{M}_{\mathcal{B},\mathbb{C},\mathrm{k},N_G}^N\right)\times\left(\mathcal{S}_{\mathbb{C}}^{N\times N}\right)^{|\mathcal{K}|}}\mathcal{F}({\Psi},\eta),
	\end{equation}
	where $\prod$ is the Cartesian product and  $\mathcal{M}_{\mathcal{B},\mathbb{C},\mathrm{k},N_G}^N$ is the Stiefel manifold
	\[
	\mathcal{M}_{\mathcal{B},\mathbb{C},\mathrm{k},N_G}^N=\{{\Psi}\in (V_{\mathrm{k},N_{G}})^N:\langle\Psi^*\mathcal{B}\Psi\rangle=I_N\}.
	\]
	Since $\prod\limits_{\mathrm{k}\in\mathcal{K}}\mathcal{M}_{\mathcal{B},\mathbb{C},\mathrm{k},N_G}^N$ is compact for any finite sampling, we obtain the existence of a minimizer of the discrete problem \eqref{eq:d-opt-eta-sym} in the sense of section \ref{subsec:minimizer}. In addition, the invariance of the energy functional and its gradient in section \ref{subsec:invariance} also holds since $\prod\limits_{\mathrm{k}\in\mathcal{K}}\mathcal{M}_{\mathcal{B},\mathbb{C},\mathrm{k},N_G}^N\subset ((H_\#^1(\Omega,\mathbb{C}))^{N})^{|\mathcal{K}|}$.
	
	\subsection{Numerical method}\label{subsec:numerical-method}
	We understand that the line search method is widely used to solve a minimization problem, in which there are two main issues: a search direction and a step size. In our minimization problem \eqref{eq:d-opt-eta-sym}, we observe that the iterative behavior for $\Psi$ and $\eta$ may be different. Hence it is better to apply different step sizes for $\Psi$ and $\eta$ when we apply the line search method to solve the minimization problem \eqref{eq:d-opt-eta-sym}. Inspired by the adaptive step size strategy proposed in \cite{dai2020adaptive}, we propose an adaptive double step size strategy for the line search method.
	
	\subsubsection{Adaptive double step size strategy}\label{subsec:adap-step}
	An adaptive step size strategy is concluded as the following four steps \cite{dai2020adaptive}:
	\begin{center}
		\textbf{Initialize} $\rightarrow$ \textbf{Estimate} $\rightarrow$ \textbf{Judge} $\rightarrow$ \textbf{Improve}.
	\end{center}
	We suppose that the initial guess of the step sizes $(t_\Psi^{n,\text{initial}},t_\eta^{n,\text{initial}})$ at $n$-th iteration is given. Then we introduce the other three steps of our adaptive double step size strategy one by one.
	
	Let $D_{\Psi}^{(n)}=(D_{\Psi_{\mathrm{k}}}^{(n)})_{\mathrm{k}\in\mathcal{K}}\in \prod\limits_{\mathrm{k}\in\mathcal{K}}\mathcal{T}_{\Psi_{\mathrm{k}}} \mathcal{M}_{\mathcal{B},\mathbb{C},\mathrm{k},N_G}^N$, $D_{\eta}^{(n)}=(D_{\eta_{\mathrm{k}}}^{(n)})_{\mathrm{k}\in\mathcal{K}}\in\left(\mathcal{S}_{\mathbb{C}}^{N\times N}\right)^{|\mathcal{K}|}$. For the sake of convenience, omiting $\Psi^{(n)},\eta^{(n)},D_\Psi^{(n)}$ and $D_\eta^{(n)}$, we denote
	\[
	\mathcal{F}((\operatorname{ortho}(\Psi_{\mathrm{k}}^{(n)},D_{\Psi_{\mathrm{k}}}^{(n)},t_\Psi))_{\mathrm{k}\in\mathcal{K}},\eta^{(n)}+t_\eta D^{(n)}_\eta)\]
	by $\bar{\mathcal{F}}_n(t_\Psi,t_\eta)$, where $\operatorname{ortho}(\Psi_{\mathrm{k}}^{(n)},D_{\Psi_{\mathrm{k}}}^{(n)},t_\Psi)$ means one step from $\Psi^{(n)}_{\mathrm{k}}\in\mathcal{M}_{\mathcal{B},\mathbb{C},\mathrm{k},N_G}^N$ with the search direction $D^{(n)}_{\Psi_{\mathrm{k}}}$ and the step size $t_\Psi$ to the next point in $\mathcal{M}_{\mathcal{B},\mathbb{C},\mathrm{k},N_G}^N$. More introduction about $\operatorname{ortho}(\Psi_{\mathrm{k}}^{(n)},D_{\Psi_{\mathrm{k}}}^{(n)},t_\Psi)$ will be provided in section \ref{subsubsec:pcg}. By a simple calculation, we have
	$$\frac{\partial\bar{\mathcal{F}}_n}{\partial t_\Psi}(0,0)=2\operatorname{Re}\langle\mathcal{F}_{\Psi}(\Psi^{(n)},\eta^{(n)}),D_\Psi^{(n)}\rangle,\,\frac{\partial\bar{\mathcal{F}}_n}{\partial t_\eta}(0,0)=\operatorname{Re}\langle\nabla_\eta\mathcal{F}(\Psi^{(n)},\eta^{(n)}),D_\eta^{(n)}\rangle.$$
	We assume  $\inner{\left(D_{\Psi_{\mathrm{k}}}^{(n)}\right)^*\mathcal{B}\Psi_{\mathrm{k}}^{(n)}}=0$ for any $\mathrm{k}\in\mathcal{K}$ to ensure
	$$\frac{\partial\bar{\mathcal{F}}_n}{\partial t_\Psi}(0,0)=\operatorname{Re}\langle\nabla_\Psi\mathcal{F}(\Psi^{(n)},\eta^{(n)}),D_\Psi^{(n)}\rangle.$$
	We always assume that all search directions $D_{\Psi}^{(n)}$ and $D_{\eta}^{(n)}$ are descent directions, namely, 
	\begin{equation}\label{eq:descent}
		\frac{\partial\bar{\mathcal{F}}_n}{\partial t_\Psi}(0,0)\le 0,~\frac{\partial\bar{\mathcal{F}}_n}{\partial t_\eta}(0,0)\le 0,\quad n=0,1,2,\ldots,
	\end{equation}
	where $\frac{\partial\bar{\mathcal{F}}_n}{\partial t_\Psi}(0,0)=0$ if and only if $\nabla_\Psi\mathcal{F}(\Psi^{(n)},\eta^{(n)})=0$, and $\frac{\partial\bar{\mathcal{F}}_n}{\partial t_\eta}(0,0)=0$ if and only if $\nabla_\eta\mathcal{F}(\Psi^{(n)},\eta^{(n)})=0$. For simplicity, we always suppose  $\|\nabla_\Psi\mathcal{F}(\Psi^{(n)},\eta^{(n)})\|+\|\nabla_\eta\mathcal{F}(\Psi^{(n)},\eta^{(n)})\|_{sF}\ne 0$ in the adaptive double step size strategy, otherwise we have obtained the minimizer of the problem \eqref{eq:d-opt-eta-sym}.
	
	\textbf{Estimate.} The final step sizes are supposed to satisfy the following non-monotone condition:
	\begin{equation}\label{eq:armijo}
		\bar{\mathcal{F}}_n(t_\Psi^{(n)},t_\eta^{(n)})-\mathcal{C}_n\le\nu\left(t_\Psi^{(n)}\frac{\partial\bar{\mathcal{F}}_n}{\partial t_\Psi}(0,0)+t_\eta^{(n)}\frac{\partial\bar{\mathcal{F}}_n}{\partial t_\eta}(0,0)\right),\,n=0,1,2,\ldots,
	\end{equation}
	where $\nu\in(0,1)$ is a given parameter. Here $\mathcal{C}_n$ can be $\mathcal{F}(\Psi^{(n)},\eta^{(n)})$ or that introduced in \cite{zhang2004nonmonotone} as follows
	\begin{equation}\label{eq:Armijo-C}
		\begin{cases}
			\mathcal{C}_0=\mathcal{F}(\Psi^{(0)},\eta^{(0)}),\,Q_0=1,\\
			Q_n=\alpha Q_{n-1} + 1,\\
			\mathcal{C}_n=(\alpha Q_{n-1}\mathcal{C}_{n-1}+\mathcal{F}(\Psi^{(n)},\eta^{(n)}))/Q_n,
		\end{cases}
	\end{equation}
	where $\alpha\in[0,1)$  is a given parameter. We consider the approximation of the energy functional $\mathcal{F}$ around $(\Psi^{(n)},\eta^{(n)})$ as follows:
	\begin{equation}\label{eq:approx-Fn}
		\bar{\mathcal{F}}_n(t_\Psi,t_\eta)\approx \bar{\mathcal{F}}_n(0,0)+t_\Psi\frac{\partial\bar{\mathcal{F}}_n}{\partial t_\Psi}(0,0)+t_\eta\frac{\partial\bar{\mathcal{F}}_n}{\partial t_\eta}(0,0)+\frac12 c_{n,1} t_\Psi^2+\frac12 c_{n,2}t_\eta^2,
	\end{equation}
	where $c_{n,1},c_{n,2}\ge 0$ are approximations of the second derivatives, $c_{n,1}=0$ if and only if $\nabla_\Psi\mathcal{F}(\Psi^{(n)},\eta^{(n)})=0$, and $c_{n,2}=0$ if and only if $\nabla_\eta\mathcal{F}(\Psi^{(n)},\eta^{(n)})=0$. Replacing $\bar{\mathcal{F}}_n(t_\Psi^{(n)},t_\eta^{(n)})$ in \eqref{eq:armijo} by the right hand term of \eqref{eq:approx-Fn}, we obtain
	\begin{align*}
		&\mathrel{\phantom{=}}\bar{\mathcal{F}}_n(0,0)+t_\Psi\frac{\partial\bar{\mathcal{F}}_n}{\partial t_\Psi}(0,0)+t_\eta\frac{\partial\bar{\mathcal{F}}_n}{\partial t_\eta}(0,0)+\frac12 c_{n,1} t_\Psi^2+\frac12 c_{n,2}t_\eta^2-\mathcal{C}_n\\
		&\le\nu\left(t_\Psi\frac{\partial\bar{\mathcal{F}}_n}{\partial t_\Psi}(0,0)+t_\eta\frac{\partial\bar{\mathcal{F}}_n}{\partial t_\eta}(0,0)\right),
	\end{align*}
	or equivalently,
	\begin{equation*}
		\frac{\bar{\mathcal{F}}_n(0,0)+t_\Psi\frac{\partial\bar{\mathcal{F}}_n}{\partial t_\Psi}(0,0)+t_\eta\frac{\partial\bar{\mathcal{F}}_n}{\partial t_\eta}(0,0)+\frac12 c_{n,1} t_\Psi^2+\frac12 c_{n,2}t_\eta^2-\mathcal{C}_n}{t_\Psi\frac{\partial\bar{\mathcal{F}}_n}{\partial t_\Psi}(0,0)+t_\eta\frac{\partial\bar{\mathcal{F}}_n}{\partial t_\eta}(0,0)}\ge\nu.
	\end{equation*}
	Hence, we propose the following estimator
	\begin{equation}\label{eq:estimator}
		\zeta_n(t_\Psi,t_\eta)=\frac{\bar{\mathcal{F}}_n(0,0)+t_\Psi\frac{\partial\bar{\mathcal{F}}_n}{\partial t_\Psi}(0,0)+t_\eta\frac{\partial\bar{\mathcal{F}}_n}{\partial t_\eta}(0,0)+\frac12 c_{n,1} t_\Psi^2+\frac12 c_{n,2}t_\eta^2-\mathcal{C}_n}{t_\Psi\frac{\partial\bar{\mathcal{F}}_n}{\partial t_\Psi}(0,0)+t_\eta\frac{\partial\bar{\mathcal{F}}_n}{\partial t_\eta}(0,0)}
	\end{equation}
	to guide us whether to accept the step sizes or not at the $n$-th iteration.
	Since the estimator \eqref{eq:estimator} remains reliable only in a neighborhood of $(\Psi^{(n)},\eta^{(n)})$, it is reasonable to restrict $t_\Psi^{(n)}\|D_{\Psi}^{(n)}\|_{\infty}\le\theta_\Psi^{(n)}$ and $t_\eta^{(n)}\|D_\eta^{(n)}\|_{sF,\infty}\le\theta_\eta^{(n)}$ for some given small $\theta_\Psi^{(n)},\,\theta_\eta^{(n)}\in(0,1)$. Thus, we first set
	\begin{equation*}
		t_\Psi^{(n)}=\min\left(t_\Psi^{n,\text{initial}},\frac{\theta_\Psi^{(n)}}{\|D_{\Psi}^{(n)}\|_{\infty}}\right),~  t_{\eta}^{(n)}=\min\left(t_\eta^{n,\text{initial}},\frac{\theta_{\eta}^{(n)}}{\|D_{\eta}^{(n)}\|_{sF,\infty}}\right),
	\end{equation*}
	and then calculate the estimator $\zeta_n(t_\Psi^{(n)},t_\eta^{(n)})$.
	
	\textbf{Judge.} The estimator $\zeta_n(t_\Psi^{(n)},t_\eta^{(n)})$ is used to determine whether to accept the step sizes $(t_\Psi^{(n)},t_\eta^{(n)})$ or not. If $(t_\Psi^{(n)},t_\eta^{(n)})$ satisfies
	\begin{equation}\label{eq:accept}
		\zeta_n(t_\Psi^{(n)},t_\eta^{(n)})\ge\nu,
	\end{equation}
	then we accept this step sizes. Otherwise, $(t_\Psi^{(n)},t_\eta^{(n)})$ is to be improved.
	
	\textbf{Improve.} If $(t_\Psi^{(n)},t_\eta^{(n)})$ is not accepted , then we solve the minimizer of the approximation \eqref{eq:approx-Fn} of $\bar{\mathcal{F}}_n$ and set it to be the step size. Combining the restriction of approximation in the neighborhood of $(\Psi^{(n)},\eta^{(n)})$, we take
	\begin{equation}\label{eq:impro-step}
		\begin{aligned}
			t_\Psi^{(n)}&=\min\left(-\frac{1}{c_{n,1}}\frac{\partial\bar{\mathcal{F}}_n}{\partial t_\Psi}(0,0),\frac{\theta_\Psi^{(n)}}{\|D_{\Psi}^{(n)}\|_{\infty}}\right),\\ t_\eta^{(n)}&=\min\left(-\frac{1}{c_{n,2}}\frac{\partial\bar{\mathcal{F}}_n}{\partial t_\eta}(0,0),\frac{\theta_\eta^{(n)}}{\|D_{\eta}^{(n)}\|_{sF,\infty}}\right).
		\end{aligned}
	\end{equation}
	Here and hereafter, $-\frac{1}{c_{n,1}}\frac{\partial\bar{\mathcal{F}}_n}{\partial t_\Psi}(0,0)$ is replaced by $-\frac{1}{c_{n,2}}\frac{\partial\bar{\mathcal{F}}_n}{\partial t_\eta}(0,0)$ if $\nabla_\Psi\mathcal{F}(\Psi^{(n)},\eta^{(n)})=0$, and $-\frac{1}{c_{n,2}}\frac{\partial\bar{\mathcal{F}}_n}{\partial t_\eta}(0,0)$ is replaced by $-\frac{1}{c_{n,1}}\frac{\partial\bar{\mathcal{F}}_n}{\partial t_\Psi}(0,0)$ if $\nabla_\eta\mathcal{F}(\Psi^{(n)},\eta^{(n)})=0$.
	Note that we choose $\nu\in(0,1/2]$ to ensure that step sizes \eqref{eq:impro-step} satisfy \eqref{eq:accept}. To ensure the convergence of the iterations, we may do some adjustments on the above step sizes. More precisely, if
	\begin{equation}\label{eq:step-lowup-bounded}
		\underline{c}\le\frac{t_\eta^{(n)}}{t_\Psi^{(n)}}\le\bar{c}
	\end{equation}
	does not hold, we then reduce one of two step sizes to make them satisfy the above inequalities. Here $\bar{c}>1>\underline{c}>0$ are given constants.
	
	\begin{remark}\label{rmk:same-step}
		We can always choose $c_{n,1},c_{n,2}$ such that the minimizer of \eqref{eq:approx-Fn} satisfies $t_\Psi=t_\eta$, i.e.,
		$$-\frac{1}{c_{n,1}}\frac{\partial\bar{\mathcal{F}}_n}{\partial t_\Psi}(0,0)=-\frac{1}{c_{n,2}}\frac{\partial\bar{\mathcal{F}}_n}{\partial t_\eta}(0,0).$$
		In this case, the minimizer of \eqref{eq:approx-Fn} is also the minimizer of the following function
		\[
		\bar{\mathcal{F}}_n(0,0)+\left(\frac{\partial\bar{\mathcal{F}}_n}{\partial t}(0,0)+\frac{\partial\bar{\mathcal{F}}_n}{\partial t}(0,0)\right)t+\frac12 (c_{n,1} + c_{n,2}) t^2.
		\]
		Hence, the approximation of $\bar{\mathcal{F}}_n$ with the same step size $t_\Psi=t_\eta$ is a special case of the above discussion.
	\end{remark}

	We summarize the above process as Algorithm \ref{algo:adap-step}.
	\begin{algorithm}[!htbp]
		\caption{Adaptive double step size strategy}
		\label{algo:adap-step}
		\begin{algorithmic}[1]
			\renewcommand{\algorithmicrequire}{\textbf{Input:}}
			\REQUIRE $\Psi,\,\eta,\,D_\Psi,\,D_\eta,\,t_\Psi^{\text{initial}},\,t_\eta^{\text{initial}},\,t_\Psi^{\text{min}},\,t_\eta^{\text{min}},\,\nu,\,c_1,\,c_2,\,\theta_\Psi,\,\theta_\eta,\,\mathcal{C}$
			\STATE Set $$t_\Psi=\min\left(\max(t_\Psi^{\text{initial}},t_\Psi^{\text{min}}),\frac{\theta_\Psi}{\|D_{\Psi}\|_{\infty}}\right),$$ $$t_\eta=\min\left(\max(t_\eta^{\text{initial}},t_\eta^{\text{min}}),\frac{\theta_\eta}{\|D_{\eta}\|_{sF,\infty}}\right);$$
			\STATE Calculate the estimator
			\[
			\zeta(t_\Psi,t_\eta)=\frac{\bar{\mathcal{F}}(0,0)+t_\Psi\frac{\partial\bar{\mathcal{F}}}{\partial t_\Psi}(0,0)+t_\eta\frac{\partial\bar{\mathcal{F}}}{\partial t_\eta}(0,0)+\frac12 c_1 t_\Psi^2+\frac12 c_2 t_\eta^2-\mathcal{C}}{t_\Psi\frac{\partial\bar{\mathcal{F}}}{\partial t_\Psi}(0,0)+t_\eta\frac{\partial\bar{\mathcal{F}}}{\partial t_\eta}(0,0)},
			\]
			where $\bar{\mathcal{F}}(t_\Psi,t_\eta)=\mathcal{F}((\operatorname{ortho}(\Psi_{\mathrm{k}},D_{\Psi_{\mathrm{k}}},t_\Psi))_{\mathrm{k}\in\mathcal{K}},\eta+t_\eta D_\eta)$;
			\IF{$\zeta(t_\Psi,t_\eta)<\nu$}
			\STATE set
			\[
			\begin{aligned}
				t_\Psi&=\min\left(-\frac{1}{c_{1}}\frac{\partial\bar{\mathcal{F}}}{\partial t_\Psi}(0,0),\frac{\theta_\Psi}{\|D_{\Psi}\|_{\infty}}\right),\\
				t_\eta&=\min\left(-\frac{1}{c_{2}}\frac{\partial\bar{\mathcal{F}}}{\partial t_\eta}(0,0),\frac{\theta_\eta}{\|D_{\eta}\|_{sF,\infty}}\right);
			\end{aligned}
			\]
			\ENDIF
			\IF{$\displaystyle\frac{t_\eta}{t_\Psi}<\underline{c}$}
			\STATE $t_\Psi=\frac{1}{\underline{c}}t_\eta,~t_{\eta}=t_\eta$;
			\ELSIF{$\displaystyle\frac{t_\eta}{t_\Psi}>\bar{c}$}
			\STATE  $t_\Psi=t_\Psi,~t_{\eta}=\bar{c}t_\Psi$;
			\ENDIF
			\STATE Return $(t_\Psi,\,t_\eta)$.
		\end{algorithmic}
	\end{algorithm}

	Note that it is very difficult to calculate the second derivatives of $\bar{\mathcal{F}}_n(t_\Psi,t_\eta)$. Thus we design some strategies to get good approximations $c_{n,1}$ and $c_{n,2}$. We provide three strategies to get $c_{n,1}$ and $c_{n,2}$ by one trial step with step sizes $(t_\Psi^{\text{trial}},t_\eta^{\text{trial}})$. For convenience, we use the short notation
	\[
	\tilde{\mathcal{F}}_n(t_\Psi,t_\eta)=\bar{\mathcal{F}}_n(0,0)+t_\Psi\frac{\partial\bar{\mathcal{F}}_n}{\partial t_\Psi}(0,0)+t_\eta\frac{\partial\bar{\mathcal{F}}_n}{\partial t_\eta}(0,0)+\frac12 c_{n,1} t_\Psi^2+\frac12 c_{n,2}t_\eta^2.
	\]
	We shall also simply denote $\bar{\mathcal{F}}_n(t,t)$ and $\tilde{\mathcal{F}}_n(t,t)$ by $\bar{\mathcal{F}}_n(t)$ and $\tilde{\mathcal{F}}_n(t)$, respectively. In this case, $c_{n,1}+c_{n,2}$ are denoted by $c_n$, trial step sizes $t_\Psi^{\text{trial}}$ and $t_\eta^{\text{trial}}$ are denoted by $t^{\text{trial}}$.
	\begin{enumerate}[label=(S\arabic*)]
		\item\label{step-energy} Applying the same step size $t_\Psi=t_\eta$ for $\Psi$ and $\eta$, we use the energy at $t^{\text{trial}}$ to get the approximation $\tilde{\mathcal{F}}_n$, namely, $\tilde{\mathcal{F}}_n$ satisfies
		\begin{equation*}
			\tilde{\mathcal{F}}_n(t^{\text{trial}}) = \bar{\mathcal{F}}_n(t^{\text{trial}}),
		\end{equation*} 
		where  $$t^{\text{trial}}=\min\left(\max\left(t^{\text{min}},t^{(n-1)}\right),\frac{\theta^{(n)}}{\sqrt{\|D_\Psi^{(n)}\|_{\infty}^2+\|D_\eta^{(n)}\|_{sF,\infty}^2}}\right),$$
		$t^{\text{min}}$ and $\theta^{(n)}\in(0,1)$ are given parameters. Then we have
		\[ 
		c_n = \frac{2(\bar{\mathcal{F}}_n(t^{\text{trial}})-\bar{\mathcal{F}}_n(0)-t^{\text{trial}}\bar{\mathcal{F}}_n'(0))}{(t^{\text{trial}})^2}.
		\]
		We choose 
		\begin{equation*}
			t_\Psi^{(n)}=t_\eta^{(n)}=\begin{cases}
				\min\left(t_m^{(n)},\frac{\theta^{(n)}}{\sqrt{\|D_\Psi^{(n)}\|_{\infty}^2+\|D_\eta^{(n)}\|_{sF,\infty}^2}}\right),\quad&t_m^{(n)} > 0,\\
				t^{\text{trial}},&\text{otherwise},
			\end{cases}
		\end{equation*}
		where
		\[
		t_m^{(n)}=-\frac{\bar{\mathcal{F}}_n'(0)}{c_n}=-\frac{\bar{\mathcal{F}}_n'(0)(t^{\text{trial}})^2}{2(\bar{\mathcal{F}}_n(t^{\text{trial}})-\bar{\mathcal{F}}_n(0)-\bar{\mathcal{F}}_n'(0)t^{\text{trial}})}.
		\]
		
		\item\label{step-d1} Applying the same step size $t_\Psi=t_\eta$ for $\Psi$ and $\eta$, we use the derivative of $\bar{\mathcal{F}}_n(t)$ at $t^{\text{trial}}$ to get the approximation $\tilde{\mathcal{F}}_n$, namely, $\tilde{\mathcal{F}}_n$ satisfies
		\begin{equation*}
			\tilde{\mathcal{F}}_n'(t^{\text{trial}}) =\bar{\mathcal{F}}_n'(t^{\text{trial}}),
		\end{equation*}
		where $$t^{\text{trial}}=\min\left(\max(t^{\text{min}},t^{(n-1)}),\frac{\theta^{(n)}}{\sqrt{\|D_\Psi^{(n)}\|_{\infty}^2+\|D_\eta^{(n)}\|_{sF,\infty}^2}}\right),$$
		$t^{\text{min}}$ and $\theta^{(n)}\in(0,1)$ are given parameters. Then we have
		\[
		c_n = \frac{\bar{\mathcal{F}}_n'(t^{\text{trial}})-\bar{\mathcal{F}}_n'(0)}{t^{\text{trial}}}.
		\]
		We choose
		\begin{equation*}
			t_\Psi^{(n)}=t_\eta^{(n)}=\begin{cases}
				\min\left(t_m^{(n)},\frac{\theta^{(n)}}{\sqrt{\|D_\Psi^{(n)}\|_{\infty}^2+\|D_\eta^{(n)}\|_{sF,\infty}^2}}\right),\quad&t_m^{(n)} > 0,\\
				t^{\text{trial}},&\text{otherwise},
			\end{cases}
		\end{equation*}
		where
		\[
		t_m^{(n)}=-\frac{\bar{\mathcal{F}}_n'(0)}{c_n}=-\frac{\bar{\mathcal{F}}_n'(0)t^{\text{trial}}}{\bar{\mathcal{F}}_n'(t^{\text{trial}})-\mathcal{F}'_n(0)}.
		\]
		
		\item\label{step-d1-2} Applying different step sizes $t_\Psi\ne t_\eta$ for $\Psi$ and $\eta$, we use partial derivatives of $\bar{\mathcal{F}}_n(t_\Psi,t_\eta)$ at $(t^{\text{trial}}_\Psi,t^{\text{trial}}_\eta)$ to get the approximation $\tilde{\mathcal{F}}_n$, namely, $\tilde{\mathcal{F}}_n$ satisfies
		\begin{align*}
			\frac{\partial\tilde{\mathcal{F}}_n}{\partial t_\Psi}{(t^{\text{trial}}_\Psi,t^{\text{trial}}_\eta)}=\frac{\partial\bar{\mathcal{F}}_n}{\partial t_\Psi}{(t^{\text{trial}}_\Psi,t^{\text{trial}}_\eta)},\quad
			\frac{\partial\tilde{\mathcal{F}}_n}{\partial t_\eta}{(t^{\text{trial}}_\Psi,t^{\text{trial}}_\eta)} =\frac{\partial\bar{\mathcal{F}}_n}{\partial t_\eta}{(t^{\text{trial}}_\Psi,t^{\text{trial}}_\eta)},
		\end{align*}
		where
		\begin{gather*}
			t^{\text{trial}}_\Psi=\min\left(\max(t^{\text{min},\Psi},t_{\Psi}^{(n-1)}),\frac{\theta^{(n)}_{\Psi}}{\|D_\Psi^{(n)}\|_{\infty}}\right),\\
			t^{\text{trial}}_\eta=\min\left(\max(t^{\text{min},\eta},t_{\eta}^{(n-1)}),\frac{\theta^{(n)}_{\eta}}{\|D_\eta^{(n)}\|_{sF,\infty}}\right),
		\end{gather*}
		$(t_{\Psi}^{\text{min}},t_{\eta}^{\text{min}})$ and  $\theta_\Psi^{(n)},\theta_\eta^{(n)}\in(0,1)$ are given parameters. Then we have
		\begin{equation*}
			c_{n,1}=\frac{\frac{\partial\bar{\mathcal{F}}_n}{\partial t_\Psi}{(t^{\text{trial}}_\Psi,t^{\text{trial}}_\eta)}-\frac{\partial\bar{\mathcal{F}}_n}{\partial t_\Psi}{(0,0)}}{t_\Psi^{\text{trial}}},\quad c_{n,2}=\frac{\frac{\partial\bar{\mathcal{F}}_n}{\partial t_\eta}{(t^{\text{trial}}_\Psi,t^{\text{trial}}_\eta)}-\frac{\partial\bar{\mathcal{F}}_n}{\partial t_\eta}{(0,0)}}{t_\eta^{\text{trial}}}.
		\end{equation*}
		We choose
		\begin{equation*}
			\begin{cases}
				t_\Psi^{(n)}=\min\left(t_{m,\Psi}^{(n)},\frac{\theta^{(n)}_{\Psi}}{\|D_\Psi^{(n)}\|_{\infty}}\right),~ t_\eta^{(n)}=\min\left(t_{m,\eta}^{(n)},\frac{\theta^{(n)}_{\eta}}{\|D_\eta^{(n)}\|_{sF,\infty}}\right),&t_{m,\Psi}^{(n)}>0 \text{~and~}t_{m,\eta}^{(n)}>0,\\
				t_\Psi^{(n)}=t^{\text{trial}}_\Psi,~ t_\eta^{(n)}=t^{\text{trial}}_\eta,~&\text{otherwise},
			\end{cases}
		\end{equation*}
		where
		\begin{align*}
			t_{m,\Psi}^{(n)}&=-\frac{\frac{\partial\bar{\mathcal{F}}_n}{\partial t_\Psi}(0,0)}{c_{n,1}}=-\frac{\frac{\partial\bar{\mathcal{F}}_n}{\partial t_\Psi}{(0,0)}t^{\text{trial}}_\Psi}{\frac{\partial\bar{\mathcal{F}}_n}{\partial t_\Psi}{(t^{\text{trial}}_\Psi,t^{\text{trial}}_\eta)}-\frac{\partial\bar{\mathcal{F}}_n}{\partial t_\Psi}{(0,0)}},\\
			t_{m,\eta}^{(n)}&=-\frac{\frac{\partial\bar{\mathcal{F}}_n}{\partial t_\eta}(0,0)}{c_{n,2}}=-\frac{\frac{\partial\bar{\mathcal{F}}_n}{\partial t_\eta}(0,0)t^{\text{trial}}_\eta}{\frac{\partial\bar{\mathcal{F}}_n}{\partial t_\eta}{(t^{\text{trial}}_\Psi,t^{\text{trial}}_\eta)}-\frac{\partial\bar{\mathcal{F}}_n}{\partial t_\eta}{(0,0)}}.
		\end{align*}
	\end{enumerate}
	
	For strategies \ref{step-d1} and \ref{step-d1-2}, we need to calculate the following two partial derivatives
	$$\frac{\partial\bar{\mathcal{F}}_n}{\partial t_\Psi}{(t^{\text{trial}}_\Psi,t^{\text{trial}}_\eta)},~\frac{\partial\bar{\mathcal{F}}_n}{\partial t_\eta}{(t^{\text{trial}}_\Psi,t^{\text{trial}}_\eta)}.$$
	A direct calculation shows
	\begin{align*}
		&\mathrel{\phantom{=}}\frac{\partial\bar{\mathcal{F}}_n}{\partial t_\Psi}{(t^{\text{trial}}_\Psi,t^{\text{trial}}_\eta)}\\
		&=\inner{\mathcal{F}_{\Psi}((\operatorname{ortho}(\Psi^{(n)}_{\mathrm{k}},D^{(n)}_{\Psi_{\mathrm{k}}},t^{\text{trial}}_\Psi))_{\mathrm{k}\in\mathcal{K}},\eta^{(n)}+t^{\text{trial}}_\eta D_\eta^{(n)}),\left(\frac{\partial\operatorname{ortho}(\Psi^{(n)}_{\mathrm{k}},D^{(n)}_{{\Psi}_{\mathrm{k}}},t^{\text{trial}}_\Psi)}{\partial t}\right)_{\mathrm{k}\in\mathcal{K}}}.
	\end{align*}
	and
	\[
	\frac{\partial\bar{\mathcal{F}}_n}{\partial t_\eta}{(t^{\text{trial}}_\Psi,t^{\text{trial}}_\eta)}=\inner{\nabla_{\eta}\mathcal{F}((\operatorname{ortho}(\Psi^{(n)}_{\mathrm{k}},D^{(n)}_{\Psi_{\mathrm{k}}},t^{\text{trial}}_\Psi))_{\mathrm{k}\in\mathcal{K}},\eta^{(n)}+t^{\text{trial}}_\eta D_\eta^{(n)}), D_\eta^{(n)}}.
	\]
	We see that $\displaystyle\frac{\partial\operatorname{ortho}(\Psi^{(n)}_{\mathrm{k}},D_{\Psi_{\mathrm{k}}}^{(n)},t^{\text{trial}}_\Psi)}{\partial t}$ is very difficult to calculate. Instead, we apply the third order approximation 
	\begin{align*}
		\frac{\partial\operatorname{ortho}(\Psi^{(n)}_{\mathrm{k}},D_{\Psi_{\mathrm{k}}}^{(n)},t^{\text{trial}}_\Psi)}{\partial t}&\approx\frac{\partial\operatorname{ortho}(\Psi^{(n)}_{\mathrm{k}},D_{\Psi_{\mathrm{k}}}^{(n)},0)}{\partial t}+\frac{\partial^2\operatorname{ortho}(\Psi^{(n)}_{\mathrm{k}},D_{\Psi_{\mathrm{k}}}^{(n)},0)}{\partial t^2}t^{\text{trial}}_\Psi\\
		&\quad+\frac12\frac{\partial^3\operatorname{ortho}(\Psi^{(n)}_{\mathrm{k}},D_{\Psi_{\mathrm{k}}}^{(n)},0)}{\partial t^3}(t^{\text{trial}}_\Psi)^2
	\end{align*}
	in practice.
	
	\subsubsection{The preconditioned conjugate gradient method}\label{subsubsec:pcg}
	Now we introduce the preconditioned conjugate gradient method for solving the minimization problem \eqref{eq:d-opt-eta-sym}. The preconditioned conjugate gradient (PCG) method is a typical line search based optimization method. For the constrained optimization problem \eqref{eq:d-opt-eta-sym}, we usually need to keep each iteration point on the constrained manifold. Thus some unitarity preserving strategies are required. We then introduce the preconditioner, the conjugate gradient parameter and the unitarity preserving strategies one by one.
	
	We first introduce the preconditioner applied to $\nabla_{\Psi}\mathcal{F}$ and $\nabla_\eta\mathcal{F}$. Let $(\Psi,\eta)\in \left(\prod\limits_{\mathrm{k}\in\mathcal{K}}\mathcal{M}_{\mathcal{B},\mathbb{C},\mathrm{k},N_G}^N\right)\times\left(\mathcal{S}_{\mathbb{C}}^{N\times N}\right)^{|\mathcal{K}|}$, where all $\eta_{\mathrm{k}}$ are diagonal matrices. We consider a preconditioner in the form of $M_\Psi^{\eta}(\Phi)=(M_{\Psi_{\mathrm{k}}}^{\eta_{\mathrm{k}}}(\Phi_{\mathrm{k}}))_{\mathrm{k}\in\mathcal{K}}$
	for $\nabla_{\Psi}\mathcal{F}$, where
	\[
	M_{\Psi_{\mathrm{k}}}^{\eta_{\mathrm{k}}}(\Phi_{\mathrm{k}})= M_{\Psi_{\mathrm{k}}}\left(\frac{1}{2w_{\mathrm{k}}}\Phi_{\mathrm{k}}F_{\eta_{\mathrm{k}}}^{-1}\right)
	\]
	and $M_{\Psi_{\mathrm{k}}}:V_{\mathrm{k},N_G}\to V_{\mathrm{k},N_G}$ is a linear operator. In our numerical experiments, we apply the following preconditioner $M_{\Psi_{\mathrm{k}}}$ used in Quantum ESPRESSO \cite{quantum}
	$$[M_{\Psi_{\mathrm{k}}}]_{G,G'}=  \delta_{G,G'} \frac{1}{1+\frac{1}{2}|\mathrm{k}+G|^2+\sqrt{1+\left(\frac{1}{2}|\mathrm{k}+G|^2-1\right)^2}},$$
	which is independent of wavefunctions. We consider a preconditioner in the form of $M_{\eta}(A)=(M_{\eta_{\mathrm{k}}}(A_{\mathrm{k}}))_{\mathrm{k}\in\mathcal{K}}$ for $\nabla_{\eta}\mathcal{F}$, where $M_{\eta_{\mathrm{k}}}:\mathcal{S}_{\mathbb{C}}^{N\times N}\to\mathcal{S}_{\mathbb{C}}^{N\times N}$ is a linear operator defined by
	\begin{equation}\label{eq:Meta}
		\left(M_{\eta_{\mathrm{k}}}(A_{\mathrm{k}})\right)_{ij}=-A_{\mathrm{k}ij}\frac{1}{w_{\mathrm{k}}} \frac{\eta_{\mathrm{k}ii}-\eta_{\mathrm{k}jj}}{f_{\mathrm{k}j}-f_{\mathrm{k}i}},~\forall i,j=1,2,\ldots,N,~\forall \mathrm{k}\in\mathcal{K}.
	\end{equation}
	Here $f_{\mathrm{k}i}=f((\eta_{\mathrm{k}ii}-\mu)/\sigma)$ and
	\[
	\frac{f_{\mathrm{k}j}-f_{\mathrm{k}i}}{\eta_{\mathrm{k}ii}-\eta_{\mathrm{k}jj}}=\frac{1}{\sigma}f'\left(\frac{\eta_{\mathrm{k}ii}-\mu}{\sigma}\right)
	\]
	when $\eta_{\mathrm{k}ii}=\eta_{\mathrm{k}jj}$.
	
	Applying $M_{\Psi_{\mathrm{k}}}^{\eta_{\mathrm{k}}}$ to
	\[
	\nabla_{\Psi_{\mathrm{k}}}\mathcal{F}(\Psi,\eta)=2w_{\mathrm{k}}(H_{\mathrm{k}}(\rho_{\Psi,\eta})\Psi  - \mathcal{B}\Psi_{\mathrm{k}}\Sigma_{\mathrm{k}})F_{\eta_{\mathrm{k}}},
	\]
	we obtain
	\[
	M_{\Psi_{\mathrm{k}}}^{\eta_{\mathrm{k}}}(\nabla_{\Psi_{\mathrm{k}}}\mathcal{F}(\Psi,\eta))=M_{\Psi_{\mathrm{k}}}(H_{\mathrm{k}}(\rho_{\Psi,\eta})\Psi - \mathcal{B}\Psi_{\mathrm{k}}\Sigma_{\mathrm{k}}).
	\]
	Here $\Sigma_{\mathrm{k}}=\langle\Psi_{\mathrm{k}}^* H_{\mathrm{k}}(\rho_{\Psi,\eta})\Psi_{\mathrm{k}}\rangle$. Compared to $\nabla_{\Psi_{\mathrm{k}}}\mathcal{F}(\Psi,\eta)$, $M_{\Psi_{\mathrm{k}}}^{\eta_{\mathrm{k}}}(\nabla_{\Psi_{\mathrm{k}}}\mathcal{F}(\Psi,\eta))$ eliminates the occupation number $F_{\eta_{\mathrm{k}}}$ and $2w_{\mathrm{k}}$. We see that 
	$$(\nabla_{\Psi_{\mathrm{k}}}\mathcal{F} (\Psi,\eta))_i=2w_{\mathrm{k}}(H_{\mathrm{k}}(\rho_{\Psi,\eta})\psi_{\mathrm{k}i}-(\mathcal{B}\Psi_{\mathrm{k}}\Sigma_{\mathrm{k}})_i)(F_{\eta_{\mathrm{k}}})_{ii}$$
	is almost $0$ when the occupation number $(F_{\eta_{\mathrm{k}}})_{ii}$ is close to $0$. Consequently, the preconditioner $M_{\Psi_{\mathrm{k}}}^{\eta_{\mathrm{k}}}$ removes $F_{\eta_{\mathrm{k}}}$ in $\nabla_{\Psi_{\mathrm{k}}}\mathcal{F}(\Psi,\eta)$ to eliminate the impact of small occupation numbers on the convergence rate, which has been mentioned in  \cite{ismail-beigi2000new,marzari1997ensemble}.
	
	Applying $M_{\eta_{\mathrm{k}}}$ to $\nabla_{\eta_{\mathrm{k}}}\mathcal{F}(\Psi,\eta)$, we have
	\[
	M_{\eta_{\mathrm{k}}}(\nabla_{\eta_{\mathrm{k}}}\mathcal{F}(\Psi,\eta))=cI+\eta_{\mathrm{k}}-\Sigma_{\mathrm{k}},
	\]
	where $c$ is defined by \eqref{eq:c-value}. We note that $\kappa(\eta_{\mathrm{k}}-\Sigma_{\mathrm{k}})$ is the preconditioned gradient mentioned in \cite{freysoldt2009direct}, where $\kappa$ is some positive constant.
	
	We then introduce the conjugate gradient parameters. The typical choices of the conjugate gradient parameters include the Hestenes-Stiefel (HS) formula \cite{hestenes1952methods}, the Polak-Ribi\'ere-Polyak (PRP) formula \cite{polak1969note,polyak1969conjugate}, the Fletcher-Reeves (FR) formula \cite{fletcher1964function} and the Dai-Yuan (DY) formula \cite{dai1999nonlinear}. In our numerical experiments, we choose the DY formula, which is expressed as
	\[
	\beta^{(n)}=\frac{\operatorname{Re}\left(\inner{ M_{\Psi^{(n)}}^{\eta^{(n)}}(G_{\Psi}^{(n)}),G_{\Psi}^{(n)} }+\inner{M_{\eta^{(n)}}(G_{\eta}^{(n)}),G_{\eta}^{(n)}}\right)}{\operatorname{Re}\left(\inner{ D_\Psi^{(n-1)},G_{\Psi}^{(n)}-G_{\Psi}^{(n-1)} }+\inner{D_\eta^{(n-1)},G_{\eta}^{(n)}-G_{\eta}^{(n-1)}}\right)}
	\]
	for the PCG algorithm, where $\operatorname{Re}$ gives the real part, $G_{\Psi}^{(n)}=\nabla_{\Psi}\mathcal{F}(\Psi^{(n)},\eta^{(n)})$, $G_{\eta}^{(n)}=\nabla_{\eta}\mathcal{F}(\Psi^{(n)},\eta^{(n)})$. Hereafter, we shall sometimes use the notations $G_{\Psi}^{(n)}$ and $G_{\eta}^{(n)}$ to simplify some formulas.
	
	Now we turn to introduce the unitarity preserving strategy we use. Let $D_{\Psi_{\mathrm{k}}}\in\mathcal{T}_{\Psi_{\mathrm{k}}} \mathcal{M}_{\mathcal{B},\mathbb{C},\mathrm{k},N_G}^N$. We denote by
	$$\operatorname{ortho}(\Psi_{\mathrm{k}},D_{\Psi_{\mathrm{k}}},t_\Psi)$$
	one step from $\Psi_{\mathrm{k}}\in\mathcal{M}_{\mathcal{B},\mathbb{C},\mathrm{k},N_G}^N$ with the search direction $D_{\Psi_{\mathrm{k}}}$ and the step size $t_\Psi$ to the next point in $\mathcal{M}_{\mathcal{B},\mathbb{C},\mathrm{k},N_G}^N$.
	In our numerical experiments, we apply the QR strategy, which is defined by
	\begin{equation}\label{eq:qr-c}
		{\operatorname{ortho}}_{\textup{QR}}(\Psi_{\mathrm{k}}, D_{\Psi_{\mathrm{k}}},  t_\Psi)  =  (\Psi_{\mathrm{k}} +  t_\Psi D_{\Psi_{\mathrm{k}}})L^{-*},
	\end{equation}
	where $L$ is the lower triangular matrix such that
	\begin{equation*}
		LL^*=I_N + t_\Psi^2 \langle D_{\Psi_{\mathrm{k}}}^* \mathcal{B} D_{\Psi_{\mathrm{k}}}\rangle.
	\end{equation*}
	We refer \cite{dai2017conjugate} for some other unitarity preserving strategies such as the PD strategy.
	
	We assume $\operatorname{ortho}(\Psi_{\mathrm{k}},D_{\Psi_{\mathrm{k}}},t_\Psi)$ satisfies the following assumption, which is needed in our analysis and valid for both QR and PD strategy (see, e.g., \cite{dai2017conjugate}).
	\begin{assumption}\label{assump:ortho-prop}
		There exist constants $C_1,C_2>0$ such that
		\begin{gather*}
			\|\operatorname{ortho}(\Phi,D_\Phi,t)-\Phi\|\le C_1 t\|D_\Phi\|,\quad\forall t\ge0,\\
			\left\|\frac{\partial}{\partial t}\operatorname{ortho}(\Phi,D_\Phi,t)-D_\Phi\right\|\le C_2 t\|D_\Phi\|^2,\quad\forall t\ge0
		\end{gather*}
		for any $\Phi\in\mathcal{M}_{\mathcal{B}}^N$ and $D_\Phi\in\mathcal{T}_\Phi\mathcal{M}_{\mathcal{B}}^N$.
	\end{assumption}
	
	We now propose our preconditioned conjugate gradient method as Algorithm \ref{algo:pcg}.
	\begin{algorithm}[!htbp]
		\caption{PCG method}
		\label{algo:pcg}
		\begin{algorithmic}[1]
			\STATE Given $\alpha\in[0,1)$, $\nu\in(0,1/2]$, $t_\Psi^{\text{min}},t_\eta^{\text{min}},E_{\text{cut}}>0$, and choose the initial data $\Psi^{(0)}_{\mathrm{k}}\in \mathcal{M}_{\mathcal{B},\mathbb{C},\mathrm{k},N_G}^N$ and  $\eta_{\mathrm{k}}^{(0)} = \operatorname{Diag}(\epsilon_{\mathrm{k}1}^{(0)},\ldots,\epsilon_{\mathrm{k}N}^{(0)})$ for any $\mathrm{k}\in\mathcal{K}$. Let $D_\Psi^{(-1)}=(D_{\Psi_{\mathrm{k}}}^{(-1)})_{\mathrm{k}\in\mathcal{K}}= 0$, $D_\eta^{(-1)}=(D_{\eta_{\mathrm{k}}}^{(-1)})_{\mathrm{k}\in\mathcal{K}}=0$, $n = 0$;
			\STATE Calculate the gradient $G_\Psi^{(n)}=(G_{\Psi_{\mathrm{k}}}^{(n)})_{\mathrm{k}\in\mathcal{K}}$, $G_\eta^{(n)}=(G_{\eta_{\mathrm{k}}}^{(n)})_{\mathrm{k}\in\mathcal{K}}$ and the preconditioned gradient $\widetilde{G}_{\Psi}^{(n)}=M_{\Psi^{(n)}}^{\eta^{(n)}}(G_{\Psi}^{(n)}),~\widetilde{G}_{\eta}^{(n)}=M_{\eta^{(n)}}(G_{\eta}^{(n)})$, where $G_{\Psi_{\mathrm{k}}}^{(n)}=\nabla_{\Psi_{\mathrm{k}}}\mathcal{F}(\Psi^{(n)},\eta^{(n)})$, $G_{\eta_{\mathrm{k}}}^{(n)}=\nabla_{\eta_{\mathrm{k}}}\mathcal{F}(\Psi^{(n)},\eta^{(n)})$
			\STATE Calculate the conjugate gradient parameter $\beta^{(n)}$;
			\STATE\label{algo-state:pcg-cg} Calculate the search direction
			\begin{equation*}
				D_{\Psi}^{(n)} = -\widetilde{G}_{\Psi}^{(n)}+ \beta^{(n)}D_{\Psi}^{(n-1)},~D_{\eta}^{(n)} = -\widetilde{G}_{\eta}^{(n)}+ \beta^{(n)}D_{\eta}^{(n-1)};
			\end{equation*}
			\STATE\label{algo-state:pcg-proj} Project the search direction $D_{\Psi_{\mathrm{k}}}^{(n)}$ to the tangent space $\mathcal{T}_{\Psi_{\mathrm{k}}} \mathcal{M}_{\mathcal{B},\mathbb{C},\mathrm{k},N_{\mathrm{G}}}^{N}$
			$$D_{\Psi_{\mathrm{k}}}^{(n)} = P_{0,\Psi^{(n)}_{\mathrm{k}}}^* (D_{\Psi_{\mathrm{k}}}^{(n)}),~\forall\mathrm{k}\in\mathcal{K};$$
			
			\STATE\label{algo-state:pcg-sign} Set $D_{\Psi_{\mathrm{k}}}^{(n)} = -D_{\Psi_{\mathrm{k}}}^{(n)} \operatorname{sign}\operatorname{Re}\inner{G_{\Psi}^{(n)}, D_{\Psi}^{(n)}}$, $D_{\eta_{\mathrm{k}}}^{(n)} = -D_{\eta_{\mathrm{k}}}^{(n)}\operatorname{sign}\operatorname{Re}\inner{G_{\eta}^{(n)}, D_{\eta}^{(n)}}$ for any $\mathrm{k}\in\mathcal{K}$;
			
			\STATE Choose the appropriate parameters $(\theta_\Psi^{(n)},\theta_\eta^{(n)})$;
			
			\STATE Calculate $\mathcal{C}_n$ by \eqref{eq:Armijo-C};
			
			\STATE Given the initial guess of the step sizes $(t_\Psi^{n,\text{initial}},t_\eta^{n,\text{initial}})$;
			
			\STATE Give $c_{n,1}$ and $c_{n,2}$ and calculate $t_\Psi^{(n)}$ and $t_\eta^{(n)}$ by
			\begin{align*}
				&\mathrel{\phantom{=}}(t_\Psi^{(n)},t_\eta^{(n)})\\
				&=\text{Adaptve double step size strategy}(\Psi^{(n)},\eta^{(n)},D_\Psi^{n},D_\eta^{(n)},t_\Psi^{n,\text{initial}},t_\eta^{n,\text{initial}},\\
				&~~~~~~~~~~~~~~~~~~~~~~~~~~~~~~~~~~~~~~~~~~~~~~~~~~~~t_\Psi^{\text{min}},t_\eta^{\text{min}},\nu,c_{n,1},c_{n,2},\theta_\Psi^{(n)},\theta_\eta^{(n)},\mathcal{C}_n);
			\end{align*}
			\STATE Set $\Psi^{(n+1)}_{\mathrm{k}} = {\operatorname{ortho}}(\Psi^{(n)}_{\mathrm{k}}, D_{\Psi_{\mathrm{k}}}^{(n)}, t_{\Psi}^{(n)})$, $\eta^{(n+1)}_{\mathrm{k}} = \eta^{(n)}_{\mathrm{k}}+t_\eta^{(n)} D_{\eta_{\mathrm{k}}}^{(n)}$ for any $\mathrm{k}\in\mathcal{K}$;
			\STATE Pick up $P^{(n+1)}=(P_{\mathrm{k}}^{(n+1)})_{\mathrm{k}\in\mathcal{K}}\in(\mathcal{O}_{\mathbb{C}}^{N\times N})^{|\mathcal{K}|}$ such that $(P^{(n+1)}_{\mathrm{k}})^*\eta_{\mathrm{k}}^{(n+1)}P_{\mathrm{k}}^{(n+1)}$ is diagonal for any $\mathrm{k}\in\mathcal{K}$ and then update
			\begin{gather*}
				\Psi^{(n+1)}=\Psi^{(n+1)} P^{(n+1)},\quad
				\eta^{(n+1)}=(P^{(n+1)})^*\eta^{(n+1)}P^{(n+1)},\\
				D_{\Psi}^{(n)}=D_{\Psi}^{(n)}P^{(n+1)},\quad
				D_{\eta}^{(n)}=(P^{(n+1)})^* D_{\eta}^{(n)}P^{(n+1)};
			\end{gather*}
			\STATE Let $n = n+1$. Convergence check: if not converged, go to step 2; else, stop.
		\end{algorithmic}
	\end{algorithm}
	
	We see that $D_{\Psi}^{(n)}$ in the \ref{algo-state:pcg-cg}-th step of Algorithm \ref{algo:pcg} is not in the tangent space $\prod_{\mathrm{k}\in\mathcal{K}}\mathcal{T}_{\Psi_{\mathrm{k}}^{(n)}}\mathcal{M}_{\mathcal{B},\mathbb{C},\mathrm{k},N_G}^N$. Thus we project $D_{\Psi_\mathrm{k}}^{(n)}$ to $\mathcal{T}_{\Psi_\mathrm{k}^{(n)}}\mathcal{M}_{\mathcal{B},\mathbb{C},\mathrm{k},N_G}^N$ in the \ref{algo-state:pcg-proj}-th step. In order to ensure
	\[
	\frac{\partial\bar{\mathcal{F}}_n}{\partial t_\Psi}(0,0)=\operatorname{Re}\langle\nabla_\Psi\mathcal{F}(\Psi^{(n)},\eta^{(n)}),D_\Psi^{(n)}\rangle,
	\]
	we apply the projection $P_{0,\Psi^{(n)}_{\mathrm{k}}}^*$ for each $\mathrm{k}\in\mathcal{K}$.
	
	\subsubsection{The restarted preconditioned conjugate gradient method}
	To get better approximations, we turn to consider the restarted preconditioned conjugate gradient method.
	
	In practice, we expect that there exists a positive constant $a$ such that
	\begin{equation}\label{eq:limsup-ge0}
		\varlimsup_{n\to\infty}\frac{-\operatorname{Re}\left(\inner{G_{\Psi}^{(n)},D_\Psi^{(n)}}+\inner{G_{\eta}^{(n)},D_\eta^{(n)}}\right)}{\left|\inner{G_{\Psi}^{(n)},M_{\Psi^{(n)}}^{\eta^{(n)}}(G_{\Psi}^{(n)})}\right|^a+\left|\inner{G_{\eta}^{(n)},M_{\eta^{(n)}}(G_{\eta}^{(n)})}\right|^a}> 0.
	\end{equation}
	Here $G_\Psi^{(n)}=\nabla_\Psi\mathcal{F}(\Psi^{(n)},\eta^{(n)})$ and $G_\eta^{(n)}=\nabla_\eta\mathcal{F}(\Psi^{(n)},\eta^{(n)})$. Thus we restart the PCG method when
	\begin{equation}\label{eq:restart-cond}
		\frac{-\operatorname{Re}\left(\inner{G_{\Psi}^{(n)},D_\Psi^{(n)}}+\inner{G_{\eta}^{(n)},D_\eta^{(n)}}\right)}{\left|\inner{G_{\Psi}^{(n)},M_{\Psi^{(n)}}^{\eta^{(n)}}(G_{\Psi}^{(n)})}\right|^a+\left|\inner{G_{\eta}^{(n)},M_{\eta^{(n)}}(G_{\eta}^{(n)})}\right|^a}<\gamma,
	\end{equation}
	for some given parameter $\gamma\in(0,1)$. Applying this strategy, we propose a restarted preconditioned conjugate gradient method shown as Algorithm \ref{algo:rpcg1}.

	\begin{algorithm}[!htbp]
		\caption{Restarted PCG method I}
		\label{algo:rpcg1}
		\begin{algorithmic}[1]
			\STATE Given $\alpha\in[0,1)$, $\nu\in(0,1/2]$, $a,t_\Psi^{\text{min}},t_\eta^{\text{min}},E_{\text{cut}}>0$, and choose the initial data $\Psi^{(0)}_{\mathrm{k}}\in \mathcal{M}_{\mathcal{B},\mathbb{C},\mathrm{k},N_G}^N$ and  $\eta_{\mathrm{k}}^{(0)} = \operatorname{Diag}(\epsilon_{\mathrm{k}1}^{(0)},\ldots,\epsilon_{\mathrm{k}N}^{(0)})$ for any $\mathrm{k}\in\mathcal{K}$. Let $D_\Psi^{(-1)}=(D_{\Psi_{\mathrm{k}}}^{(-1)})_{\mathrm{k}\in\mathcal{K}}= 0$, $D_\eta^{(-1)}=(D_{\eta_{\mathrm{k}}}^{(-1)})_{\mathrm{k}\in\mathcal{K}}=0$, $n = 0$;
			\STATE Calculate the gradient $G_\Psi^{(n)}=(G_{\Psi_{\mathrm{k}}}^{(n)})_{\mathrm{k}\in\mathcal{K}}$, $G_\eta^{(n)}=(G_{\eta_{\mathrm{k}}}^{(n)})_{\mathrm{k}\in\mathcal{K}}$ and the preconditioned gradient $\widetilde{G}_{\Psi}^{(n)}=M_{\Psi^{(n)}}^{\eta^{(n)}}(G_{\Psi}^{(n)}),~\widetilde{G}_{\eta}^{(n)}=M_{\eta^{(n)}}(G_{\eta}^{(n)})$, where $G_{\Psi_{\mathrm{k}}}^{(n)}=\nabla_{\Psi_{\mathrm{k}}}\mathcal{F}(\Psi^{(n)},\eta^{(n)})$, $G_{\eta_{\mathrm{k}}}^{(n)}=\nabla_{\eta_{\mathrm{k}}}\mathcal{F}(\Psi^{(n)},\eta^{(n)})$
			\STATE Calculate the conjugate gradient parameter $\beta^{(n)}$;
			\STATE Calculate the search direction
			\begin{equation*}
				D_{\Psi}^{(n)} = -\widetilde{G}_{\Psi}^{(n)}+ \beta^{(n)}D_{\Psi}^{(n-1)},~D_{\eta}^{(n)} = -\widetilde{G}_{\eta}^{(n)}+ \beta^{(n)}D_{\eta}^{(n-1)};
			\end{equation*}
			\STATE Project the search direction $D_{\Psi_{\mathrm{k}}}^{(n)}$ to the tangent space $\mathcal{T}_{\Psi_{\mathrm{k}}} \mathcal{M}_{\mathcal{B},\mathbb{C},\mathrm{k},N_{\mathrm{G}}}^{N}$
			$$D_{\Psi_{\mathrm{k}}}^{(n)} = P_{0,\Psi^{(n)}_{\mathrm{k}}}^* (D_{\Psi_{\mathrm{k}}}^{(n)}),~\forall\mathrm{k}\in\mathcal{K};$$
			
			\STATE\label{algo-state:rpcg1-sign} Set $D_{\Psi_{\mathrm{k}}}^{(n)} = -D_{\Psi_{\mathrm{k}}}^{(n)} \operatorname{sign}\operatorname{Re}\inner{G_{\Psi}^{(n)}, D_{\Psi}^{(n)}}$, $D_{\eta_{\mathrm{k}}}^{(n)} = -D_{\eta_{\mathrm{k}}}^{(n)}\operatorname{sign}\operatorname{Re}\inner{G_{\eta}^{(n)}, D_{\eta}^{(n)}}$ for any $\mathrm{k}\in\mathcal{K}$;
			
			\IF{\eqref{eq:restart-cond} holds}
			\STATE $D_{\Psi_{\mathrm{k}}}^{(n)} = -P_{0,\Psi^{(n)}_{\mathrm{k}}}^* (\widetilde{G}_{\Psi_{\mathrm{k}}}^{(n)}),~D_{\eta_{\mathrm{k}}}^{(n)}=-\widetilde{G}_{\eta_{\mathrm{k}}}^{(n)},~\forall \mathrm{k}\in\mathcal{K}$;
			\ENDIF
			
			\STATE Choose the appropriate parameters $(\theta_\Psi^{(n)},\theta_\eta^{(n)})$;
			
			\STATE Calculate $\mathcal{C}_n$ by \eqref{eq:Armijo-C};
			
			\STATE Given the initial guess of the step sizes $(t_\Psi^{n,\text{initial}},t_\eta^{n,\text{initial}})$;
			
			\STATE Give $c_{n,1}$ and $c_{n,2}$ and calculate $t_\Psi^{(n)}$ and $t_\eta^{(n)}$ by
			\begin{align*}
				&\mathrel{\phantom{=}}(t_\Psi^{(n)},t_\eta^{(n)})\\
				&=\text{Adaptve double step size strategy}(\Psi^{(n)},\eta^{(n)},D_\Psi^{n},D_\eta^{(n)},t_\Psi^{n,\text{initial}},t_\eta^{n,\text{initial}},\\
				&~~~~~~~~~~~~~~~~~~~~~~~~~~~~~~~~~~~~~~~~~~~~~~~~~~~~t_\Psi^{\text{min}},t_\eta^{\text{min}},\nu,c_{n,1},c_{n,2},\theta_\Psi^{(n)},\theta_\eta^{(n)},\mathcal{C}_n);
			\end{align*}
			\STATE Set $\Psi^{(n+1)}_{\mathrm{k}} = {\operatorname{ortho}}(\Psi^{(n)}_{\mathrm{k}}, D_{\Psi_{\mathrm{k}}}^{(n)}, t_{\Psi}^{(n)})$, $\eta^{(n+1)}_{\mathrm{k}} = \eta^{(n)}_{\mathrm{k}}+t_\eta^{(n)} D_{\eta_{\mathrm{k}}}^{(n)}$ for any $\mathrm{k}\in\mathcal{K}$;
			\STATE Pick up $P^{(n+1)}=(P_{\mathrm{k}}^{(n+1)})_{\mathrm{k}\in\mathcal{K}}\in(\mathcal{O}_{\mathbb{C}}^{N\times N})^{|\mathcal{K}|}$ such that $(P^{(n+1)}_{\mathrm{k}})^*\eta_{\mathrm{k}}^{(n+1)}P_{\mathrm{k}}^{(n+1)}$ is diagonal for any $\mathrm{k}\in\mathcal{K}$ and then update
			\begin{gather*}
				\Psi^{(n+1)}=\Psi^{(n+1)} P^{(n+1)},\quad
				\eta^{(n+1)}=(P^{(n+1)})^*\eta^{(n+1)}P^{(n+1)},\\
				D_{\Psi}^{(n)}=D_{\Psi}^{(n)}P^{(n+1)},\quad
				D_{\eta}^{(n)}=(P^{(n+1)})^* D_{\eta}^{(n)}P^{(n+1)};
			\end{gather*}
			\STATE Let $n = n+1$. Convergence check: if not converged, go to step 2; else, stop.
		\end{algorithmic}
	\end{algorithm}

	In the numerical experiments, we observe that retarting directly is sometimes better than changing the sign of the search direction when the preconditioned conjugate gradient direction is not a descent direction. Thus we propose a new restarted preconditioned conjugate gradient method shown as Algorithm \ref{algo:rpcg2}.
	
	\begin{algorithm}[!htbp]
		\caption{Restarted PCG method II}
		\label{algo:rpcg2}
		\begin{algorithmic}[1]
			\STATE Given $\alpha\in[0,1)$, $\nu\in(0,1/2]$, $a,t_\Psi^{\text{min}},t_\eta^{\text{min}},E_{\text{cut}}>0$, and choose the initial data $\Psi^{(0)}_{\mathrm{k}}\in \mathcal{M}_{\mathcal{B},\mathbb{C},\mathrm{k},N_G}^N$ and  $\eta_{\mathrm{k}}^{(0)} = \operatorname{Diag}(\epsilon_{\mathrm{k}1}^{(0)},\ldots,\epsilon_{\mathrm{k}N}^{(0)})$ for any $\mathrm{k}\in\mathcal{K}$. Let $D_\Psi^{(-1)}=(D_{\Psi_{\mathrm{k}}}^{(-1)})_{\mathrm{k}\in\mathcal{K}}= 0$, $D_\eta^{(-1)}=(D_{\eta_{\mathrm{k}}}^{(-1)})_{\mathrm{k}\in\mathcal{K}}=0$, $n = 0$;
			\STATE Calculate the gradient $G_\Psi^{(n)}=(G_{\Psi_{\mathrm{k}}}^{(n)})_{\mathrm{k}\in\mathcal{K}}$, $G_\eta^{(n)}=(G_{\eta_{\mathrm{k}}}^{(n)})_{\mathrm{k}\in\mathcal{K}}$ and the preconditioned gradient $\widetilde{G}_{\Psi}^{(n)}=M_{\Psi^{(n)}}^{\eta^{(n)}}(G_{\Psi}^{(n)}),~\widetilde{G}_{\eta}^{(n)}=M_{\eta^{(n)}}(G_{\eta}^{(n)})$, where $G_{\Psi_{\mathrm{k}}}^{(n)}=\nabla_{\Psi_{\mathrm{k}}}\mathcal{F}(\Psi^{(n)},\eta^{(n)})$, $G_{\eta_{\mathrm{k}}}^{(n)}=\nabla_{\eta_{\mathrm{k}}}\mathcal{F}(\Psi^{(n)},\eta^{(n)})$
			\STATE Calculate the conjugate gradient parameter $\beta^{(n)}$;
			\STATE Calculate the search direction
			\begin{equation*}
				D_{\Psi}^{(n)} = -\widetilde{G}_{\Psi}^{(n)}+ \beta^{(n)}D_{\Psi}^{(n-1)},~D_{\eta}^{(n)} = -\widetilde{G}_{\eta}^{(n)}+ \beta^{(n)}D_{\eta}^{(n-1)};
			\end{equation*}
			\STATE Project the search direction $D_{\Psi_{\mathrm{k}}}^{(n)}$ to the tangent space $\mathcal{T}_{\Psi_{\mathrm{k}}} \mathcal{M}_{\mathcal{B},\mathbb{C},\mathrm{k},N_{\mathrm{G}}}^{N}$
			$$D_{\Psi_{\mathrm{k}}}^{(n)} = P_{0,\Psi^{(n)}_{\mathrm{k}}}^* (D_{\Psi_{\mathrm{k}}}^{(n)}),~\forall\mathrm{k}\in\mathcal{K};$$
			
			\IF{$\operatorname{sign}\operatorname{Re}\inner{G_{\Psi}^{(n)},D_{\Psi}^{(n)}}\ge 0$ or $\operatorname{sign}\operatorname{Re}\inner{G_{\eta}^{(n)}, D_{\eta}^{(n)}}\ge 0$ or \eqref{eq:restart-cond} holds}
			\STATE $D_{\Psi_{\mathrm{k}}}^{(n)} = -P_{0,\Psi^{(n)}_{\mathrm{k}}}^* (\widetilde{G}_{\Psi_{\mathrm{k}}}^{(n)}),~D_{\eta_{\mathrm{k}}}^{(n)}=-\widetilde{G}_{\eta_{\mathrm{k}}}^{(n)},~\forall\mathrm{k}\in\mathcal{K}$;
			\ENDIF
			
			\STATE Choose the appropriate parameters $(\theta_\Psi^{(n)},\theta_\eta^{(n)})$;
			
			\STATE Calculate $\mathcal{C}_n$ by \eqref{eq:Armijo-C};
			
			\STATE Given the initial guess of the step sizes $(t_\Psi^{n,\text{initial}},t_\eta^{n,\text{initial}})$;
			
			\STATE Give $c_{n,1}$ and $c_{n,2}$ and calculate $t_\Psi^{(n)}$ and $t_\eta^{(n)}$ by
			\begin{align*}
				&\mathrel{\phantom{=}}(t_\Psi^{(n)},t_\eta^{(n)})\\
				&=\text{Adaptve double step size strategy}(\Psi^{(n)},\eta^{(n)},D_\Psi^{n},D_\eta^{(n)},t_\Psi^{n,\text{initial}},t_\eta^{n,\text{initial}},\\
				&~~~~~~~~~~~~~~~~~~~~~~~~~~~~~~~~~~~~~~~~~~~~~~~~~~~~t_\Psi^{\text{min}},t_\eta^{\text{min}},\nu,c_{n,1},c_{n,2},\theta_\Psi^{(n)},\theta_\eta^{(n)},\mathcal{C}_n);
			\end{align*}
			\STATE Set $\Psi^{(n+1)}_{\mathrm{k}} = {\operatorname{ortho}}(\Psi^{(n)}_{\mathrm{k}}, D_{\Psi_{\mathrm{k}}}^{(n)}, t_{\Psi}^{(n)})$, $\eta^{(n+1)}_{\mathrm{k}} = \eta^{(n)}_{\mathrm{k}}+t_\eta^{(n)} D_{\eta_{\mathrm{k}}}^{(n)}$ for any $\mathrm{k}\in\mathcal{K}$;
			\STATE Pick up $P^{(n+1)}=(P_{\mathrm{k}}^{(n+1)})_{\mathrm{k}\in\mathcal{K}}\in(\mathcal{O}_{\mathbb{C}}^{N\times N})^{|\mathcal{K}|}$ such that $(P^{(n+1)}_{\mathrm{k}})^*\eta_{\mathrm{k}}^{(n+1)}P_{\mathrm{k}}^{(n+1)}$ is diagonal for any $\mathrm{k}\in\mathcal{K}$ and then update
			\begin{gather*}
				\Psi^{(n+1)}=\Psi^{(n+1)} P^{(n+1)},\quad
				\eta^{(n+1)}=(P^{(n+1)})^*\eta^{(n+1)}P^{(n+1)},\\
				D_{\Psi}^{(n)}=D_{\Psi}^{(n)}P^{(n+1)},\quad
				D_{\eta}^{(n)}=(P^{(n+1)})^* D_{\eta}^{(n)}P^{(n+1)};
			\end{gather*}
			\STATE Let $n = n+1$. Convergence check: if not converged, go to step 2; else, stop.
		\end{algorithmic}
	\end{algorithm}

	\subsection{Convergence analysis}
	In this subsection, we analyze the convergence of the restarted PCG methods (Algorithms \ref{algo:rpcg1} and \ref{algo:rpcg2}). For convenience, we show the detailed proofs for the case that the sampling of k-points is at $\Gamma$ point only. For the general sampling $\mathcal{K}$, the convergence of the restarted PCG method can be obtained by the similar arguments.  We shall sometimes use the notations $G_{\Psi}^{(n)}=\nabla_{\Psi}\mathcal{F}(\Psi^{(n)},\eta^{(n)})$ and $G_{\eta}^{(n)}=\nabla_{\eta}\mathcal{F}(\Psi^{(n)},\eta^{(n)})$ to simplify some formulas.
	
	We first give some assumptions which is needed in our analysis.
	\begin{assumption}\label{assump:M-positive}
		There exist $\alpha_\Psi,\alpha_\eta>0$ such that 
		\begin{equation}\label{ineq:M-positive}
			\begin{aligned}
				\inner{\nabla_\Psi\mathcal{F}(\Psi,\eta),M_{\Psi}^{\eta}(\nabla_\Psi\mathcal{F}(\Psi,\eta))}&\ge\alpha_\Psi\|\nabla_\Psi\mathcal{F}(\Psi,\eta)\|^2,\\
				\inner{\nabla_\eta\mathcal{F}(\Psi,\eta),M_{\eta}(\nabla_\eta\mathcal{F}(\Psi,\eta))}&\ge\alpha_\eta\|\nabla_\eta\mathcal{F}(\Psi,\eta)\|_{sF}^2
			\end{aligned}
		\end{equation}
		for $(\Psi,\eta)\in\mathcal{M}_{\mathcal{B},N_G}^N\times\mathcal{S}^{N\times N}$.
	\end{assumption}

	We obtain from the assumption above that the preconditioner is bounded from below uniformly. We see that $M_{\Psi}^{\eta}$ we applied always satisfies \eqref{ineq:M-positive} and $M_{\eta}$ we applied satisfies \eqref{ineq:M-positive} when $f$ is strictly monotonically decreasing.
	
	\begin{assumption}\label{assump:lipschitzF}
		The gradient of $\mathcal{F}$ is Lipschitz continuous. That is, there exists $L_0>0$ such that
		\begin{align*}
			&\mathrel{\phantom{=}}\|\mathcal{F}_\Psi(\Psi_1,\eta_1)-\mathcal{F}_\Psi(\Psi_2,\eta_2)\|+\|\mathcal{F}_\eta(\Psi_1,\eta_1)-\mathcal{F}_\eta(\Psi_2,\eta_2)\|_{sF}\\
			&\le L_0(\|\Psi_1-\Psi_2\|+\|\eta_1-\eta_2\|_{sF})
		\end{align*}
		for any $(\Psi_1,\eta_1),(\Psi_2,\eta_2)\in\mathcal{M}_{\mathcal{B},N_G}^N\times \mathcal{S}^{N\times N}$.
	\end{assumption}
	
	\begin{assumption}\label{assump:cn-bounded}
		There exists a constant $\bar{C}>0$ such that
		\begin{equation}\label{eq:cn-bounded}
			c_{n,1}+c_{n,2}\le\bar{C}(\|D_\Psi^{(n)}\|^2+\|D_\eta^{(n)}\|_{sF}^2),\quad n=0,1,2,\ldots
		\end{equation}
	\end{assumption}
	\begin{assumption}\label{assump:cn-lowerupper}
		There holds
		\begin{equation}\label{eq:cn-lowerupper}
			\underline{c}\le\frac{-\frac{1}{c_{n,2}}\frac{\partial\bar{\mathcal{F}}_n}{\partial t_\eta}(0,0)}{-\frac{1}{c_{n,1}}\frac{\partial\bar{\mathcal{F}}_n}{\partial t_\Psi}(0,0)}\le\bar{c}, \quad n=0,1,2,\ldots
		\end{equation}
	\end{assumption}
	
	We observe that the assumption \ref{assump:cn-bounded} is similar to that the Hessian of $\mathcal{F}$ is bounded. If the same step sizes for $\Psi$ and $\eta$ are applied, then we see from Remark \ref{rmk:same-step} that Assumption \ref{assump:cn-lowerupper} is satisfied. And we can always choose some $c_{n,1}$ and $c_{n,2}$ such that Assumptions \ref{assump:cn-bounded} and \ref{assump:cn-lowerupper} hold.
	
	\begin{assumption}\label{assump:limsup-Dbounded}
		For the subsequence $\{n_j\}_{j\in\mathbb{N}}$ satisfying
		\begin{equation*}
			\lim_{j\to\infty}\frac{-\left(\inner{G_{\Psi}^{(n_j)},D_\Psi^{(n_j)}}+\inner{G_{\eta}^{(n_j)},D_\eta^{(n_j)}}\right)}{\left|\inner{G_{\Psi}^{(n_j)},M_{\Psi^{(n_j)}}^{\eta^{(n_j)}}(G_{\Psi}^{(n_j)})}\right|^a+\left|\inner{G_{\eta}^{(n_j)},M_{\eta^{(n_j)}}(G_{\eta}^{(n_j)})}\right|^a}\ne 0,
		\end{equation*}
		there exists a constant $C>0$ such that
		\begin{equation}\label{ineq:Dbounded}
			\|D_\Psi^{(n_j)}\|+\|D_\eta^{(n_j)}\|\le C,~\forall j\in\mathbb{N}.
		\end{equation}
	\end{assumption}
	
	We see that the above assumption can be satisfied by many strategies in practice. For example, if the preconditioned gradients in the iterations are bounded uniformly, we can restart the algorithm when the conjugate gradient parameter is very large. Then we obtain uniformly bounded search directions.

	In the following lemma, we need the following assumption for the step sizes.
	\begin{equation}\label{eq:liminf-step-ge0}
		\varliminf_{n\to\infty} t_\Psi^{(n)}>0,\quad\varliminf_{n\to\infty} t_\eta^{(n)}>0.
	\end{equation}

	\begin{lemma}\label{lem:converge}
		Suppose Assumption \ref{assump:M-positive} holds and the sequence $\{(\Psi^{(n)},\eta^{(n)})\}_{n\in\mathbb{N}}$ is generated by Algorithm \ref{algo:pcg}. If $D_\Psi^{(n)}$ and $D_\eta^{(n)}$ satisfy \eqref{eq:descent} and \eqref{eq:limsup-ge0}, $t_\Psi^{(n)}$ and $t_\eta^{(n)}$ satisfy \eqref{eq:armijo} and \eqref{eq:liminf-step-ge0}, then either 
		\[
		\|\nabla_\Psi\mathcal{F}(\Psi^{n},\eta^{(n)})\|=0,\,\|\nabla_\eta\mathcal{F}(\Psi^{n},\eta^{(n)})\|_{sF}=0
		\]
		for some positive $n$ or
		\[
		\varliminf_{n\to\infty}(\|\nabla_\Psi\mathcal{F}(\Psi^{n},\eta^{(n)})\|+\|\nabla_\eta\mathcal{F}(\Psi^{n},\eta^{(n)})\|_{sF})=0.
		\]
	\end{lemma}
	\begin{proof}
		Suppose
		\[
		\|\nabla_\Psi\mathcal{F}(\Psi^{n},\eta^{(n)})\|+\|\nabla_\eta\mathcal{F}(\Psi^{n},\eta^{(n)})\|_{sF}\ne 0,\forall n\in\mathbb{N},
		\]
		otherwise the conclusion is true. It follows from the definition of $\mathcal{C}_n$ that for any $n\ge 1$, there holds
		\begin{equation*}
			\mathcal{F}(\Psi^{(n+1)},\eta^{(n+1)})-\mathcal{F}(\Psi^{(n)},\eta^{(n)})=\mathcal{F}(\Psi^{(n+1)},\eta^{(n+1)})-\mathcal{C}_n-\frac{\alpha Q_{n-1}}{Q_n}(\mathcal{F}(\Psi^{(n)},\eta^{(n)})-\mathcal{C}_{n-1}).
		\end{equation*}
		Since 
		\[
		\mathcal{F}(\Psi^{(1)},\eta^{(1)})-\mathcal{F}(\Psi^{(0)},\eta^{(0)})=\mathcal{F}(\Psi^{(1)},\eta^{(1)})-\mathcal{C}_0,
		\]
		summing up all $n\in\mathbb{N}$ gives that
		\[
		\begin{aligned}
			&\mathrel{\phantom{=}}\sum_{n=0}^{\infty}(\mathcal{F}(\Psi^{n},\eta^{(n)})-\mathcal{F}(\Psi^{(n+1)},\eta^{(n+1)}))\\
			&=-\sum_{n=0}^{\infty}(\mathcal{F}(\Psi^{n+1},\eta^{(n+1)})-\mathcal{C}_n)+\sum_{n=0}^{\infty}\frac{\alpha Q_{n}}{Q_{n+1}}(\mathcal{F}(\Psi^{(n+1)},\eta^{(n+1)})-\mathcal{C}_n)\\
			&=-\sum_{n=0}^{\infty}\frac{1}{Q_{n+1}}(\mathcal{F}(\Psi^{n+1},\eta^{(n+1)})-\mathcal{C}_n)\\
			&\ge -\nu\sum_{n=0}^{\infty}\frac{1}{Q_{n+1}}\left(t_\Psi^{(n)}\frac{\partial\bar{\mathcal{F}}_n}{\partial t_\Psi}(0,0)+t_\eta^{(n)}\frac{\partial\bar{\mathcal{F}}_n}{\partial t_\eta}(0,0)\right).
		\end{aligned}
		\]
		Note that $Q_n=1+\sum\limits_{i=1}^n\alpha^i\in[1,\frac{1}{1-\alpha}]$, which together with \eqref{eq:descent} leads to
		\[
		-\sum_{n=0}^{\infty}t_\Psi^{(n)}\frac{\partial\bar{\mathcal{F}}_n}{\partial t_\Psi}(0,0)<+\infty,\quad	-\sum_{n=0}^{\infty}t_\eta^{(n)}\frac{\partial\bar{\mathcal{F}}_n}{\partial t_\eta}(0,0)<+\infty.
		\]
		Hence
		\[
		\lim_{n\to\infty}t_\Psi^{(n)}\frac{\partial\bar{\mathcal{F}}_n}{\partial t_\Psi}(0,0)=0,\quad\lim_{n\to\infty}t_\eta^{(n)}\frac{\partial\bar{\mathcal{F}}_n}{\partial t_\eta}(0,0)=0.
		\]
		Then by \eqref{eq:liminf-step-ge0}, we have
		\[
		\lim_{n\to\infty}\frac{\partial\bar{\mathcal{F}}_n}{\partial t_\Psi}(0,0)=0,\quad\lim_{n\to\infty}\frac{\partial\bar{\mathcal{F}}_n}{\partial t_\eta}(0,0)=0,
		\]
		which arrive at
		\[
		\lim_{n\to\infty}\left(-\frac{\partial\bar{\mathcal{F}}_n}{\partial t_\Psi}(0,0)-\frac{\partial\bar{\mathcal{F}}_n}{\partial t_\eta}(0,0)\right)=0.
		\]
		Since $-\frac{\partial\bar{\mathcal{F}}_n}{\partial t_\Psi}(0,0)-\frac{\partial\bar{\mathcal{F}}_n}{\partial t_\eta}(0,0)$ is a product of 
		\[
		\left|\inner{G_{\Psi}^{(n)},M_{\Psi^{(n)}}^{\eta^{(n)}}(G_{\Psi}^{(n)})}\right|^a+\left|\inner{G_{\eta}^{(n)},M_{\eta^{(n)}}(G_{\eta}^{(n)})}\right|^a
		\]
		and
		\[
		\frac{-\left(\inner{G_{\Psi}^{(n)},D_\Psi^{(n)}}+\inner{G_{\Psi}^{(n)},D_\eta^{(n)}}\right)}{\left|\inner{G_{\Psi}^{(n)},M_{\Psi^{(n)}}^{\eta^{(n)}}(G_{\Psi}^{(n)})}\right|^a+\left|\inner{G_{\eta}^{(n)},M_{\eta^{(n)}}(G_{\eta}^{(n)})}\right|^a},
		\]
		we obtain from \eqref{eq:limsup-ge0} that
		\[
		\varliminf_{n\to\infty}\left|\inner{G_{\Psi}^{(n)},M_{\Psi^{(n)}}^{\eta^{(n)}}(G_{\Psi}^{(n)})}\right|^a+\left|\inner{G_{\eta}^{(n)},M_{\eta^{(n)}}(G_{\eta}^{(n)})}\right|^a=0.
		\]
		Consequently, we get from \eqref{ineq:M-positive} that
		\[
		\varliminf_{n\to\infty}(\|\nabla_\Psi\mathcal{F}(\Psi^{n},\eta^{(n)})\|+\|\nabla_\eta\mathcal{F}(\Psi^{n},\eta^{(n)})\|_{sF})=0,
		\]
		which completes the proof.
	\end{proof}

	\begin{remark}\label{rmk:D-step}
		We see from the above proof that we may only need to consider the subsequence $\{n_j\}_{j\in\mathbb{N}}$ satisfying
		\begin{equation*}
			\lim_{j\to\infty}\frac{-\left(\inner{G_{\Psi}^{(n_j)},D_\Psi^{(n_j)}}+\inner{G_{\eta}^{(n_j)},D_\eta^{(n_j)}}\right)}{\left|\inner{G_{\Psi}^{(n_j)},M_{\Psi^{(n_j)}}^{\eta^{(n_j)}}(G_{\Psi}^{(n_j)})}\right|^a+\left|\inner{G_{\eta}^{(n_j)},M_{\eta^{(n_j)}}(G_{\eta}^{(n_j)})}\right|^a}\ne 0.
		\end{equation*}
		In addition, \eqref{eq:liminf-step-ge0} can be replaced by that \eqref{eq:step-lowup-bounded} holds for the above $\{n_j\}_{j\in\mathbb{N}}$ and
		\begin{equation}\label{eq:stepsum-infty}
			\sum_{j=0}^{\infty}t_\Psi^{(n_j)}=+\infty.
		\end{equation}
		We mention that \eqref{eq:stepsum-infty} is weaker than \eqref{eq:liminf-step-ge0} under the premise of \eqref{eq:step-lowup-bounded}.
	\end{remark}
	
	\begin{theorem}
		Suppose $\mathcal{F}$ is continuously differentiable in $H_\#^1(\Omega)\times \mathcal{S}^{N\times N}$ and $\mathcal{F}_\Psi$ is bounded, i.e., there exists $C_0>0$ such that
		\begin{equation}\label{eq:gradpsi-bounded}
			\|\mathcal{F}_\Psi(\Psi,\eta)\|\le C_0,~\forall(\Psi,\eta)\in\mathcal{M}_{\mathcal{B},N_G}^N\times\mathcal{S}^{N\times N},
		\end{equation}
		and Assumptions \ref{assump:ortho-prop} - \ref{assump:cn-lowerupper} hold true. Let $\{(D_\Psi^{(n)},D_\eta^{(n)})\}_{n\in\mathbb{N}}$ and $\{(\Psi^{(n)},\eta^{(n)})\}_{n\in\mathbb{N}}$ are generated by Algorithm \ref{algo:rpcg1} or Algorithm \ref{algo:rpcg2}. If $\{(D_\Psi^{(n)},D_\eta^{(n)})\}_{n\in\mathbb{N}}$ satisfies Assumption \ref{assump:limsup-Dbounded}, then there exists a positive sequence $\{(\theta_\Psi^{(n)},\theta_\eta^{(n)})\}_{n\in\mathbb{N}}$ such that either
		\[
		\|\nabla_\Psi\mathcal{F}(\Psi^{n},\eta^{(n)})\|=0,\,\|\nabla_\eta\mathcal{F}(\Psi^{n},\eta^{(n)})\|_{sF}=0
		\]
		for some $n>0$ or
		\[
		\varliminf_{n\to\infty}(\|\nabla_\Psi\mathcal{F}(\Psi^{n},\eta^{(n)})\|+\|\nabla_\eta\mathcal{F}(\Psi^{n},\eta^{(n)})\|_{sF})=0.
		\]
	\end{theorem}
	\begin{proof}
		Let
		\[
		\begin{aligned}
			(\theta_\Psi^{(n)},\theta_\eta^{(n)})=\sup\bigg\{&(\tilde{\theta}_\Psi^{(n)},\tilde{\theta}_\eta^{(n)}):\bar{\mathcal{F}}_n(t_\Psi,t_\eta)-\bar{\mathcal{F}}_n(0,0)-t_\Psi\frac{\partial\bar{\mathcal{F}}_n}{\partial t_\Psi}(0,0)-t_\eta\frac{\partial\bar{\mathcal{F}}_n}{\partial t_\eta}(0,0)\\
			&-\frac12 c_{n,1} t_\Psi^2-\frac12 c_{n,2}t_\eta^2
			\le-\frac{\nu}{2}\left(t_\Psi^{(n)}\frac{\partial\bar{\mathcal{F}}_n}{\partial t_\Psi}(0,0)+t_\eta^{(n)}\frac{\partial\bar{\mathcal{F}}_n}{\partial t_\eta}(0,0)\right)\\
			&\text{for any~}(t_\Psi,t_\eta)\in T_{\tilde{\theta}_{\Psi}^{(n)},\tilde{\theta}_{\eta}^{(n)}},~\underline{c}\le\frac{\tilde{\theta}_\eta^{(n)}}{\|D_\eta^{(n)}\|_{sF}}\bigg/\frac{\tilde{\theta}_\Psi^{(n)}}{\|D_\Psi^{(n)}\|}\le\bar{c}\\
			&\text{when~}\|D_\Psi^{(n)}\|\ne0\text{~and~}\|D_\eta^{(n)}\|_{sF}\ne0,\\
			&\tilde{\theta}_\Psi^{(n)}=1\text{~when~}\|D_\Psi^{(n)}\|\ne0,~\text{and~}\tilde{\theta}_\eta^{(n)}=1\text{~when~}\|D_\eta^{(n)}\|_{sF}=0\bigg\},
		\end{aligned}
		\]
		where $\sup$ is in the sense of lexicographical order and
		\[
		T_{\tilde{\theta}_{\Psi}^{(n)},\tilde{\theta}_{\eta}^{(n)}}=\left\{(t_{\Psi},t_{\eta}):0\le t_\Psi\le\frac{\tilde{\theta}_\Psi^{(n)}}{\|D_\Psi^{(n)}\|},~0\le t_\eta\le\frac{\tilde{\theta}_\eta^{(n)}}{\|D_\eta^{(n)}\|_{sF}},~\text{and}~\underline{c}\le\frac{t_\eta}{t_\Psi}\le\bar{c} \right\}.
		\]
		Then we prove that the conclusion is valid when above $(\theta_\Psi^{(n)},\theta_\eta^{(n)})$ are taken.
		
		Suppose
		\[
		\|\nabla_\Psi\mathcal{F}(\Psi^{n},\eta^{(n)})\|+\|\nabla_\eta\mathcal{F}(\Psi^{n},\eta^{(n)})\|_{sF}\ne 0,~\forall n\in\mathbb{N},
		\]
		otherwise the conclusion is true. In Algorithm \ref{algo:rpcg1} or Algorithm \ref{algo:rpcg2}, it follows from Assumption \ref{assump:cn-lowerupper} that every $t_\Psi^{(n)}$ and $t_\eta^{(n)}$ satisfies
		\[
		\begin{aligned}
			\zeta_n(t_\Psi^{(n)},t_\eta^{(n)})&\ge \nu,\\
			t_\Psi^{(n)}\|D_\Psi^{(n)}\|\le \theta_\Psi^{(n)},~&t_\eta^{(n)}\|D_\eta^{(n)}\|_{sF}\le\theta_\eta^{(n)},
		\end{aligned}
		\]
		which implies
		\begin{align*}
			&\mathrel{\phantom{=}}\bar{\mathcal{F}}_n(0,0)+t_\Psi^{(n)}\frac{\partial\bar{\mathcal{F}}_n}{\partial t_\Psi}(0,0)+t_\eta^{(n)}\frac{\partial\bar{\mathcal{F}}_n}{\partial t_\eta}(0,0)+\frac12 c_{n,1} (t_\Psi^{(n)})^2+\frac12 c_{n,2}(t_\eta^{(n)})^2-\mathcal{C}_n\\
			&\le\nu \left(t_\Psi^{(n)}\frac{\partial\bar{\mathcal{F}}_n}{\partial t_\Psi}(0,0)+t_\eta^{(n)}\frac{\partial\bar{\mathcal{F}}_n}{\partial t_\eta}(0,0)\right).
		\end{align*}
		Then we obtain from the definition of $(\theta_\Psi^{(n)},\theta_\eta^{(n)})$ that
		\[
		\bar{\mathcal{F}}_n(t_\Psi^{(n)},t_\eta^{(n)})-\mathcal{C}_n\le\frac{\nu}{2}\left(t_\Psi^{(n)}\frac{\partial\bar{\mathcal{F}}_n}{\partial t_\Psi}(0,0)+t_\eta^{(n)}\frac{\partial\bar{\mathcal{F}}_n}{\partial t_\eta}(0,0)\right),
		\]
		i.e., \eqref{eq:armijo} holds.
		
		As shown in Remark \ref{rmk:D-step}, we only need to take subsequence $\{n_j\}_{j\in\mathbb{N}}$ satisfying
		\begin{equation*}
			\lim_{j\to\infty}\frac{-\left(\inner{G_{\Psi}^{(n_j)},D_\Psi^{(n_j)}}+\inner{G_{\eta}^{(n_j)},D_\eta^{(n_j)}}\right)}{\left|\inner{G_{\Psi}^{(n_j)},M_{\Psi^{(n_j)}}^{\eta^{(n_j)}}(G_{\Psi}^{(n_j)})}\right|^a+\left|\inner{G_{\eta}^{(n_j)},M_{\eta^{(n_j)}}(G_{\eta}^{(n_j)})}\right|^a}=\delta> 0
		\end{equation*}
		into account.
		
		We observe that the corresponding $t_\Psi^{(n_j)}$ has only four options:
		\[
		t_\Psi^{(n_j)}=\max(t_\Psi^{\text{initial}},t_\Psi^{\text{min}}),~
		t_\Psi^{(n_j)}=\frac{\theta_\Psi^{(n_j)}}{\|D_\Psi^{(n_j)}\|},~
		t_\Psi^{(n_j)}=-\frac{1}{c_{n_j,1}}\frac{\partial\mathcal{F}_{n_j}}{\partial t_\Psi}(0,0),~
		t_\Psi^{(n_j)}=\frac{1}{\underline{c}}t_\eta^{(n_j)}.
		\]
		Consequently, there exists a subsequence of $\{n_j\}_{j\in\mathbb{N}}$, which is also denoted by $\{n_j\}_{j\in\mathbb{N}}$ for convenience, such that one of the following four cases holds.
		
		\textbf{Case 1.} $t_\Psi^{(n_j)}=\max(t_\Psi^{\text{initial}},t_\Psi^{\text{min}})$. Obviously
		\[
		\sum_{j=0}^{\infty}t_\Psi^{(n_j)}\ge\sum_{j=0}^{\infty}t_\Psi^{\text{min}}=+\infty,
		\]
		which together with Remark \ref{rmk:D-step} yields the conclusion.
		
		\textbf{Case 2.} $t_\Psi^{(n_j)}=\frac{\theta_\Psi^{(n_j)}}{\|D_\Psi^{(n_j)}\|}$. If
		\[
		\varliminf_{j\to\infty} t_{\Psi}^{(n_j)}>0,
		\]
		then Lemma \ref{lem:converge} leads to the conclusion. Otherwise, there exists a subsequence of $\{n_j\}_{j\in\infty}$ also denoted by $\{n_j\}_{j\in\mathbb{N}}$ such that $\lim\limits_{j\to\infty}\frac{\theta_\Psi^{(n_j)}}{\|D_\Psi^{(n_j)}\|}=\lim\limits_{j\to\infty} t_\Psi^{(n_j)}=0$.
		
		We first prove that there holds
		\begin{equation}\label{eq:2order-approx}
			\begin{aligned}
				&\mathrel{\phantom{=}}\bar{\mathcal{F}}_n(t_\Psi,t_\eta)-\bar{\mathcal{F}}_n(0,0)-t_\Psi\frac{\partial\bar{\mathcal{F}}_n}{\partial t_\Psi}(0,0)-t_\eta\frac{\partial\bar{\mathcal{F}}_n}{\partial t_\eta}(0,0)-\frac12 c_{n,1} t_\Psi^2-\frac12 c_{n,2}t_\eta^2\\
				&=O(t_\Psi^2\|D_\Psi^{(n)}\|^2+t_\eta^2\|D_\eta^{(n)}\|_{sF}^2)
			\end{aligned}
		\end{equation}
		when $(t_\Psi,t_\eta)$ satisfies \eqref{eq:step-lowup-bounded}.
		
		For convenience, we denote by $\bar{\mathcal{F}}_{n,s}(t)=\bar{\mathcal{F}}_n(t,s t)$, $\Psi^{(n)}(t)=\operatorname{ortho}(\Psi^{(n)},D_\Psi^{(n)},t)$, $\eta^{(n)}(t)=\eta^{(n)}+t D_\eta^{(n)}$. By Assumptions \ref{assump:ortho-prop} and \ref{assump:lipschitzF}, \eqref{eq:gradpsi-bounded} and $\dot{\Psi}^{(n)}(0)=D_\Psi^{(n)}$, we have
		\begin{align*}
			&\mathrel{\phantom{=}}|\bar{\mathcal{F}}_{n,s}'(t)-\bar{\mathcal{F}}_{n,s}'(0)|\\
			&=\bigg|2\operatorname{Re}\inner{\mathcal{F}_\Psi(\Psi^{(n)}(t),\eta^{(n)}(s t)),\dot{\Psi}^{(n)}(t)}+s \inner{\nabla_\eta\mathcal{F}(\Psi^{(n)}(t),\eta^{(n)}(s t)),D_\eta^{(n)}}\\
			&\quad-2\operatorname{Re}\inner{\mathcal{F}_\Psi(\Psi^{(n)}(0),\eta^{(n)}(0)),\dot{\Psi}^{(n)}(0)}
			-s \inner{\nabla_\eta\mathcal{F}(\Psi^{(n)}(0),\eta^{(n)}(0)),D_\eta^{(n)}}\bigg|\\
			&\le\bigg|2\inner{\mathcal{F}_\Psi(\Psi^{(n)}(t),\eta^{(n)}(s t)),\dot{\Psi}^{(n)}(t)-\dot{\Psi}^{(n)}(0)}\bigg|\\
			&\quad+\bigg|2\inner{\mathcal{F}_\Psi(\Psi^{(n)}(t),\eta^{(n)}(s t))-\nabla_\Psi\mathcal{F}(\Psi^{(n)}(0),\eta^{(n)}(0)),\dot{\Psi}^{(n)}(0)}\bigg|\\
			&\quad
			+s \bigg|\inner{\nabla_\eta\mathcal{F}(\Psi^{(n)}(t),\eta^{(n)}(s t))-\nabla_\eta\mathcal{F}(\Psi^{(n)}(0),\eta^{(n)}(0)),D_\eta^{(n)}}\bigg|\\
			&\le 2C_0 C_2 t\|D_\Psi^{(n)}\|^2+2L_0(C_1 t\|D_\Psi^{(n)}\|+s t\|D_\eta^{(n)}\|_{sF})\|D_\Psi^{(n)}\|\\
			&\quad+L_0(C_1 s t\|D_\Psi^{(n)}\|+s^2 t\|D_\eta^{(n)}\|_{sF})\|D_\eta^{(n)}\|_{sF}
		\end{align*}
		for any $s,t>0$, where Proposition \ref{prop:sF} and Theorem \ref{prop:grad} are used in the last inequality. Applying Young inequality, we obtain that
		\begin{align*}
			&\mathrel{\phantom{=}}|\bar{\mathcal{F}}_{n,s}'(t)-\bar{\mathcal{F}}_{n,s}'(0)|\\
			&\le 2C_0 C_2 t\|D_\Psi^{(n)}\|^2+2L_0\left(C_1 t\|D_\Psi^{(n)}\|^2+\frac{t\|D_\Psi^{(n)}\|^2+s^2 t\|D_\eta^{(n)}\|_{sF}^2}{2} \right)\\
			&\quad+L_0\left(C_1\frac{t\|D_\Psi^{(n)}\|^2+s^2 t\|D_\eta^{(n)}\|_{sF}^2}{2}+s^2 t\|D_\eta^{(n)}\|_{sF}^2 \right)\\
			&\le \tilde{C} t(\|D_\Psi^{(n)}\|^2+ s^2 \|D_\eta^{(n)}\|_{sF}^2),
		\end{align*}
		where
		\[
		\tilde{C}=\max\left(2C_0 C_2+L_0+\frac{5}{2}C_1 L_0,2L_0+\frac{1}{2}C_1 L_0\right).
		\]
		Hence we have 
		\begin{align*}
			&\mathrel{\phantom{=}}\left|\bar{\mathcal{F}}_n(t,s t)-\bar{\mathcal{F}}_n(0,0)-t\frac{\partial\bar{\mathcal{F}}_n}{\partial t_\Psi}(0,0)-s t\frac{\partial\bar{\mathcal{F}}_n}{\partial t_\eta}(0,0)\right|\\
			&\le\int_0^{t}\left|\bar{\mathcal{F}}_{n,s}'(\tau)-\bar{\mathcal{F}}_{n,s}'(0)\right|\operatorname{d\!}\tau\\
			&\le \tilde{C} t^2(\|D_\Psi^{(n)}\|^2+ s^2 \|D_\eta^{(n)}\|_{sF}^2)
		\end{align*}
		for any $s,t>0$. Therefore, if $(t_\Psi,t_\eta)$ satisfies \eqref{eq:step-lowup-bounded}, then we arrive at \eqref{eq:2order-approx} by \eqref{eq:cn-bounded}.
		
		By the definition of $(\theta_\Psi^{(n_j)},\theta_\eta^{(n_j)})$, Assumption \ref{assump:M-positive}, \eqref{eq:step-lowup-bounded}, \eqref{ineq:M-positive} and \eqref{ineq:Dbounded}, for any $n_j$ large enough, there exists
		\begin{equation}\label{eq:t-star}
		t_\Psi^{*,n_j}\in\left(0,\frac{\theta_\Psi^{(n_j)}}{\|D_\Psi^{(n_j)}\|}+\frac{1}{n_j}\right),~t_\eta^{*,n_j}=s^*_{n_j} t_\Psi^{*,n_j},~\underline{c}\le s^*_{n_j}\le\bar{c}
		\end{equation}
		such that
			\begin{equation*}
				\begin{aligned}
					&\mathrel{\phantom{=}}O((t_\Psi^{*,n_j})^2\|D_\Psi^{(n_j)}\|^2+(t_\eta^{*,n_j})^2\|D_\eta^{(n_j)}\|_{sF}^2)\\
					&=\mathcal{F}_{n_j}(t_\Psi^{*,n_j},t_\eta^{*,n_j})-\mathcal{F}_{n_j}(0,0)-t_\Psi^{*,n_j}\frac{\partial\mathcal{F}_{n_j}}{\partial t_\Psi}(0,0)-t_\eta^{*,n_j}\frac{\partial\mathcal{F}_{n_j}}{\partial t_\eta}(0,0)\\
					&\quad-\frac12 c_{n,1} (t_\Psi^{*,n_j})^2-\frac12 c_{n,2}(t_\eta^{*,n_j})^2\\
					&>-\frac{\nu}{2}\left(t_\Psi^{*,n_j}\frac{\partial\mathcal{F}_{n_j}}{\partial t_\Psi}(0,0)+t_\eta^{*,n_j}\frac{\partial\mathcal{F}_{n_j}}{\partial t_\eta}(0,0)\right)\\
					&\ge\frac{-\frac{\nu}{2}\min(1,\underline{c})\min(1,{1}/{\bar{c}})\left(\frac{\partial\mathcal{F}_{n_j}}{\partial t_\Psi}(0,0)+\frac{\partial\mathcal{F}_{n_j}}{\partial t_\eta}(0,0)\right)}{\left|\inner{G_{\Psi}^{(n_j)},M_{\Psi^{(n_j)}}^{\eta^{(n_j)}}(G_{\Psi}^{(n_j)})}\right|^a+\left|\inner{G_{\eta}^{(n_j)},M_{\eta^{(n_j)}}(G_{\eta}^{(n_j)})}\right|^a}\\
					&\quad\cdot \frac{1}{\left(\|D_\Psi^{(n_j)}\|^2+\|D_\eta^{(n_j)}\|_{sF}^2\right)^{1/2}}\left((t_\Psi^{*,n_j})^2\|D_\Psi^{(n_j)}\|^2+(t_\eta^{*,n_j})^2\|D_\eta^{(n_j)}\|_{sF}^2\right)^{1/2}\\
					&\quad\cdot \left(\left|\inner{G_{\Psi}^{(n_j)},M_{\Psi^{(n_j)}}^{\eta^{(n_j)}}(G_{\Psi}^{(n_j)})}\right|^a+\left|\inner{G_{\eta}^{(n_j)},M_{\eta^{(n_j)}}(G_{\eta}^{(n_j)})}\right|^a\right)\\
					&\ge\frac{1}{4C}\nu\delta\min(1,\underline{c})\min(1,1/\bar{c})\min(a_\Psi^a,a_\eta^a)\left((t_\Psi^{*,n_j})^2\|D_\Psi^{(n_j)}\|^2+(t_\eta^{*,n_j})^2\|D_\eta^{(n_j)}\|_{sF}^2\right)^{1/2}\\
					&\quad\cdot \left(\|G_{\Psi}^{(n_j)}\|^{2a}+\|G_{\eta}^{(n_j)}\|_{sF}^{2a}\right),
				\end{aligned}
			\end{equation*}
		i.e.,
		\begin{equation}\label{eq:Otp0}
			\begin{aligned}
				&\mathrel{\phantom{=}}O(((t_\Psi^{*,n_j})^2\|D_\Psi^{(n_j)}\|^2+(t_\eta^{*,n_j})^2\|D_\eta^{(n_j)}\|_{sF}^2)^{1/2})\\
				&\ge\frac{\nu\delta\min(1,\underline{c})\min(1,1/\bar{c})\min(a_\Psi^a,a_\eta^a)}{4C}\cdot\left(\|G_{\Psi}^{(n_j)}\|^{2a}+\|G_{\eta}^{(n_j)}\|_{sF}^{2a}\right).
			\end{aligned}
		\end{equation}
		We see from $\lim\limits_{j\to\infty}\frac{\theta_\Psi^{(n_j)}}{\|D_\Psi^{(n_j)}\|}=0$, \eqref{ineq:Dbounded} and \eqref{eq:t-star}, that
		\[
		\lim_{j\to\infty} \left((t_\Psi^{*,n_j})^2\|D_\Psi^{(n_j)}\|^2+(t_\eta^{*,n_j})^2\|D_\eta^{(n_j)}\|_{sF}^2\right)^{1/2}=0.
		\]
		Let $j\to\infty$ in \eqref{eq:Otp0}, we get
		\[
		0\ge\frac{\nu\delta\min(1,\underline{c})\min(1,1/\bar{c})\min(a_\Psi^a,a_\eta^a)}{4C}\lim_{j\to\infty}\left(\left(\|G_{\Psi}^{(n_j)}\|^{2a}+\|G_{\eta}^{(n_j)}\|_{sF}^{2a}\right)\right),
		\]
		which produces the conclusion.
		
		\textbf{Case 3.} $t_\Psi^{(n_j)}=-\frac{1}{c_{n_j,1}}\frac{\partial\mathcal{F}_{n_j}}{\partial t_\Psi}(0,0)=-\frac{\inner{G_\Psi^{(n_j)},D_\Psi^{(n_j)}}}{c_{n_j,1}}$. We get from Assumption \ref{assump:cn-lowerupper} that $t_\eta^{(n_j)}$ has only three options:
		\[
		t_\eta^{(n_j)}=\max(t_\eta^{\text{initial}},t_\eta^{\text{min}}),~
		t_\eta^{(n_j)}=\frac{\theta_\eta^{(n_j)}}{\|D_\eta^{(n_j)}\|},~
		t_\eta^{(n_j)}=-\frac{1}{c_{n_j,2}}\frac{\partial\mathcal{F}_{n_j}}{\partial t_\eta}(0,0).
		\]
		If $t_\eta^{(n_j)}$ is one of the first two options, the similar arguments in Cases 1 and 2 can be applied to $t_\eta^{(n_j)}$. Thus, let $	t_\eta^{(n_j)}=-\frac{1}{c_{n_j,2}}\frac{\partial\mathcal{F}_{n_j}}{\partial t_\eta}(0,0)=-\frac{\inner{G_\eta^{(n_j)},D_\eta^{(n_j)}}}{c_{n_j,2}}$. Then we obtain from Assumptions \ref{assump:cn-bounded} and \ref{assump:limsup-Dbounded}, \eqref{eq:step-lowup-bounded}, and \eqref{ineq:M-positive} that
		\[
		\begin{aligned}
			&\mathrel{\phantom{=}}(1+\bar{c})t_\Psi^{(n_j)}\\
			&\ge t_\Psi^{(n_j)}+t_\eta^{(n_j)}\\
			&\ge-\frac{\inner{G_\Psi^{(n_j)},D_\Psi^{(n_j)}}+\inner{G_\eta^{(n_j)},D_\eta^{(n_j)}}}{\bar{C}(\|D_\Psi^{(n_j)}\|^2+\|D_\eta^{(n_j)}\|_{sF}^2)}\\
			&=\frac{-\left(\inner{G_\Psi^{(n_j)},D_\Psi^{(n_j)}}+\inner{G_\eta^{(n_j)},D_\eta^{(n_j)}}\right)}{\left|\inner{G_\Psi^{(n_j)},M_{\Psi^{(n_j)}}^{\eta^{(n_j)}}(G_\Psi^{(n_j)})}\right|^a+\left|\inner{G_\eta^{(n_j)},M_{\eta^{(n_j)}}(G_\eta^{(n_j)})}\right|^a}\\
			&\quad\cdot\frac{\left|\inner{G_\Psi^{(n_j)},M_{\Psi^{(n_j)}}^{\eta^{(n_j)}}(G_\Psi^{(n_j)})}\right|^a+\left|\inner{G_\eta^{(n_j)},M_{\eta^{(n_j)}}(G_\eta^{(n_j)})}\right|^a}{\bar{C}(\|D_\Psi^{(n_j)}\|^2+\|D_\eta^{(n_j)}\|_{sF}^2)}\\
			&\ge\frac{\delta\min(a_\Psi^a,a_\eta^a)}{2\bar{C}C^2}\left(\|G_\Psi^{(n_j)}\|^{2a}+\|G_\eta^{(n_j)}\|_{sF}^{2a}\right)
		\end{aligned}
		\]
		provided $j\gg1$. Consequently, either
		\[
		\sum_{j=0}^{\infty}t_\Psi^{(n_j)}=\infty
		\]
		or
		\[
		\varliminf_{j\to+\infty}(\|\nabla_\Psi\mathcal{F}(\Psi^{(n_j)},\eta^{(n_j)})\|^{2a}+\|\nabla_\eta\mathcal{F}(\Psi^{(n_j)},\eta^{(n_j)})\|_{sF}^{2a})=0,
		\]
		which leads to the conclusion.
		
		\textbf{Case 4.} $\displaystyle t_\eta^{(n_j)}=\frac{1}{\underline{c}}t_\eta^{(n_j)}$. We observe that the corresponding $t_\Psi^{(n_j)}$ has only two options:
		\[
		t_\eta^{(n_j)}=\max(t_\eta^{\text{initial}},t_\eta^{\text{min}}),~
		t_\eta^{(n_j)}=\frac{\theta_\eta^{(n_j)}}{\|D_\eta^{(n_j)}\|_{sF}}.
		\]
		Thus applying similar arguments in Cases 1 and 2 for $t_\Psi^{(n_j)}$ to $t_\eta^{(n_j)}$, we complete the proof.
	\end{proof}
	
	\section{Numerical experiments}\label{sec:numeraical}
	In this section, we apply the PCG method and its restarted versions to simulate several gold clusters (see Figure \ref{fig:au-structure1} for their configurations) and two complicated multicomponent periodic systems (see Figure \ref{fig:structure} for their configurations). We implement the PCG method and its restarted versions in the software package Quantum ESPRESSO \cite{quantum}. All calculations are carried out on LSSC-IV in the State Key Laboratory of Scientific and Engineering Computing of the Chinese Academy of Sciences.
	
	\begin{figure}[!htbp]
		\centering
		\caption{The configurations of the gold clusters}
		\label{fig:au-structure1}
		\subfloat[\ce{Au14}]{\includegraphics[width=0.3\linewidth]{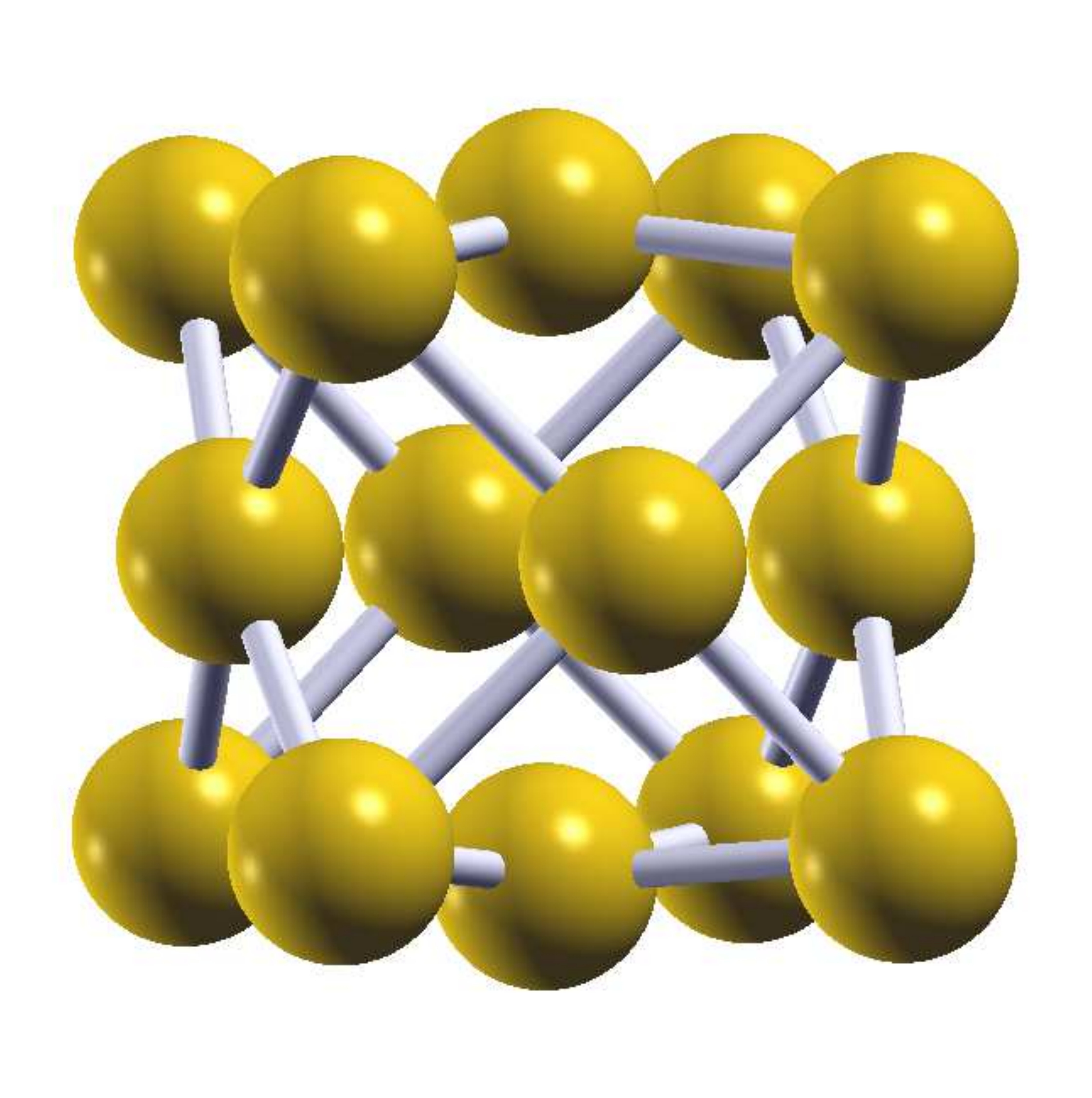}}~~~~~~~
		\subfloat[\ce{Au18}]{\includegraphics[width=0.3\linewidth]{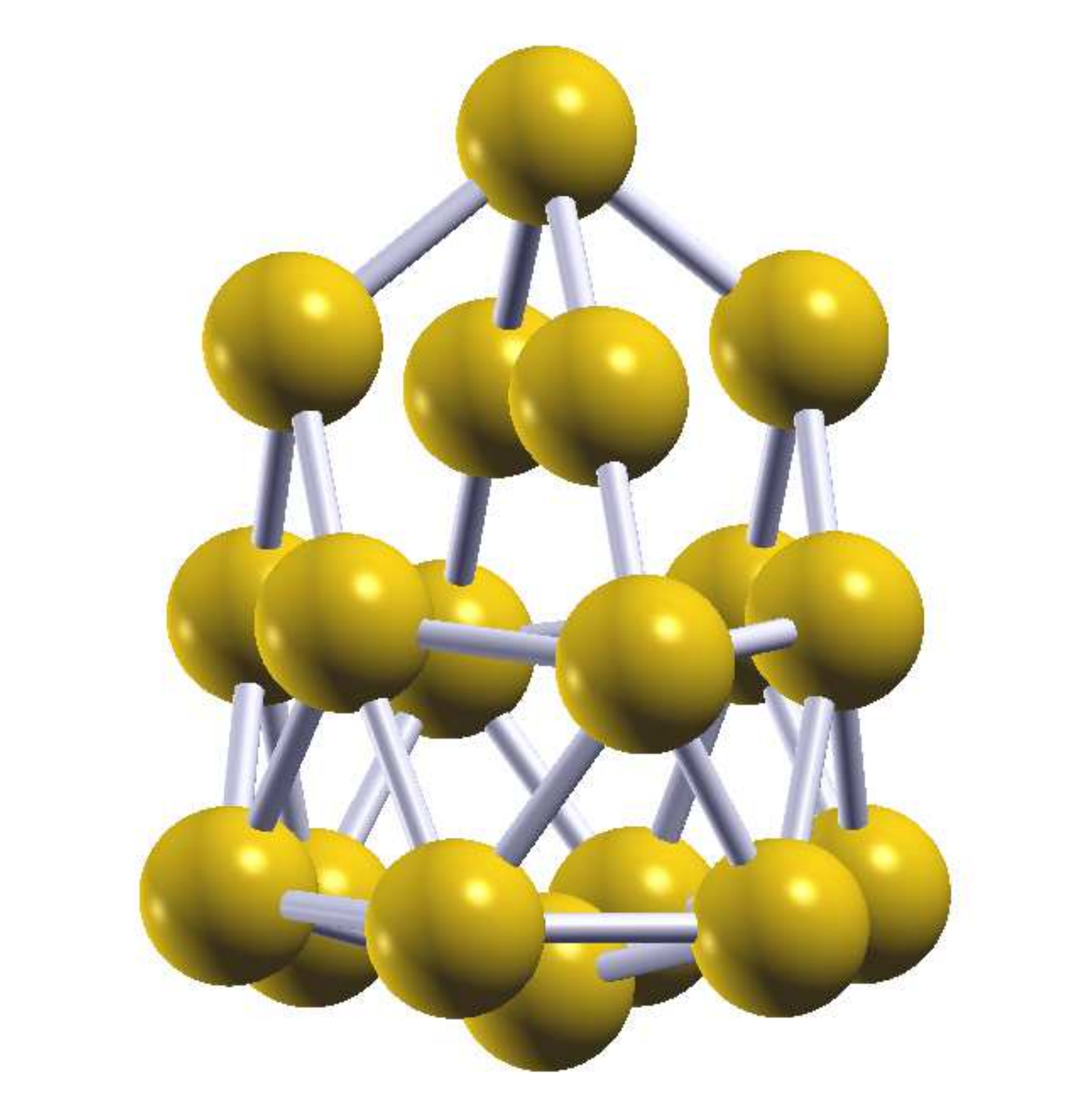}}~~~~~~~
		\subfloat[\ce{Au20}]{\includegraphics[width=0.3\linewidth]{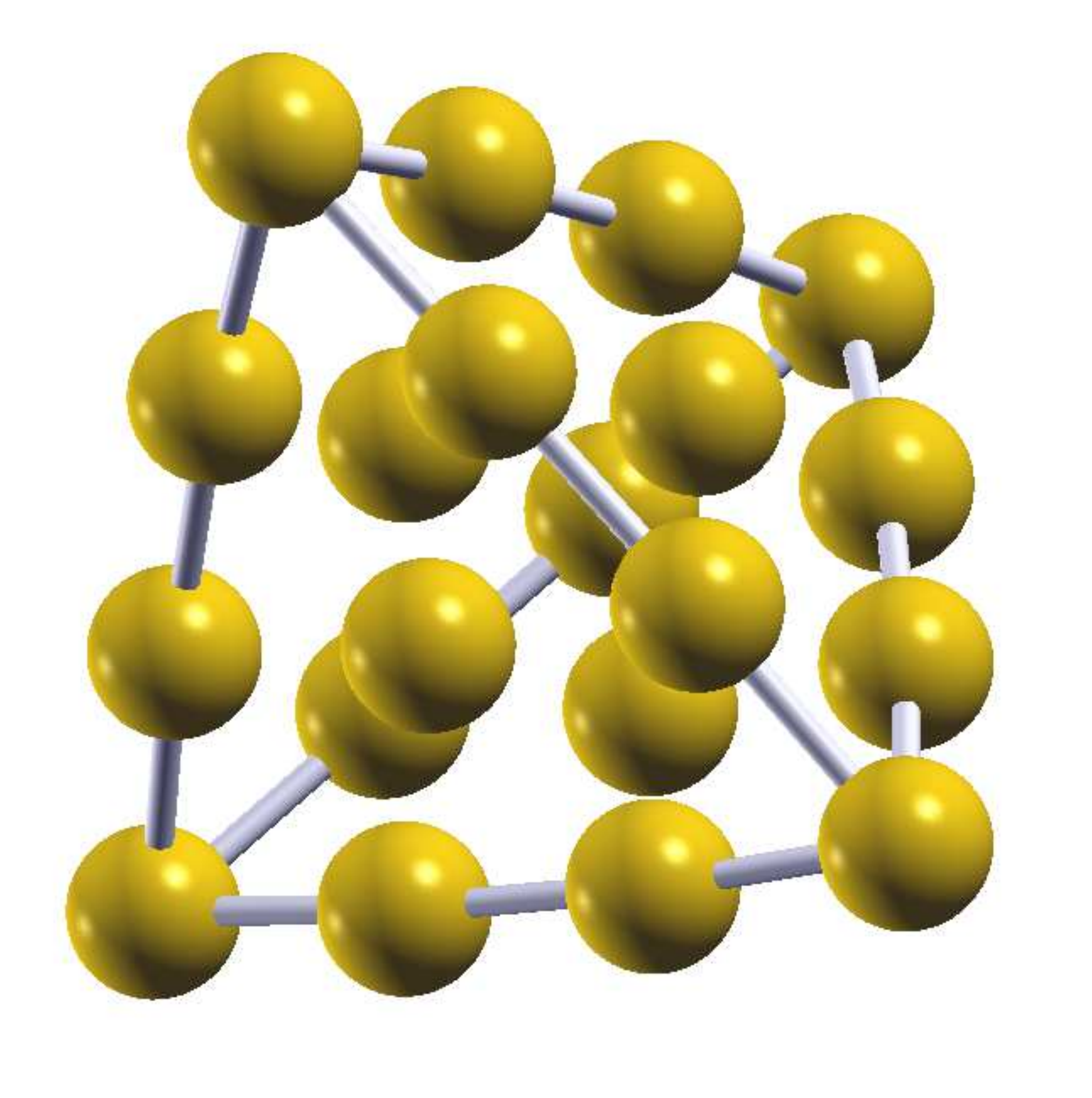}}
		\\
		\subfloat[\ce{Au32}]{\includegraphics[width=0.3\linewidth]{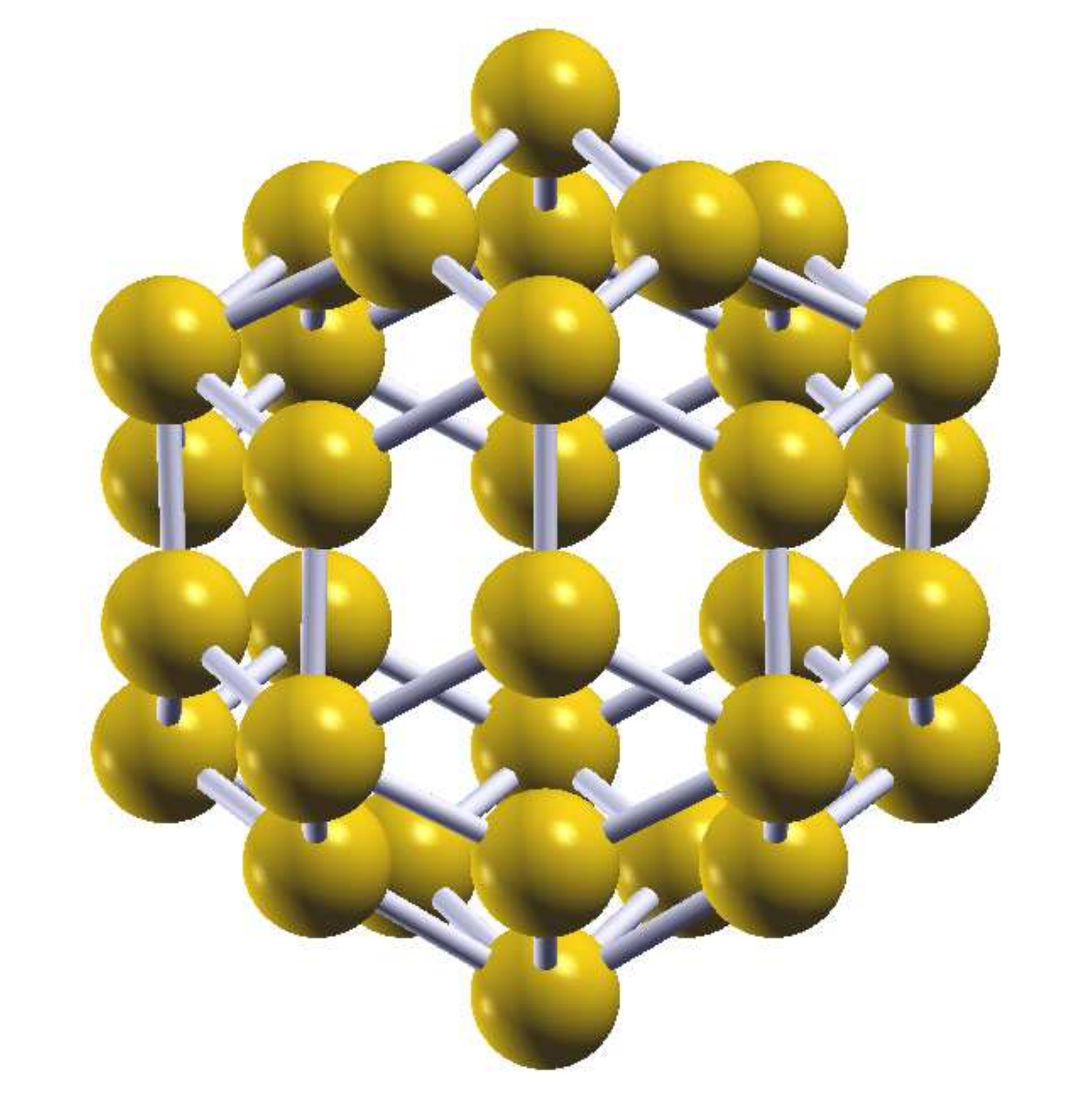}}~~~~~~~
		\subfloat[\ce{Au42}]{\includegraphics[width=0.3\linewidth]{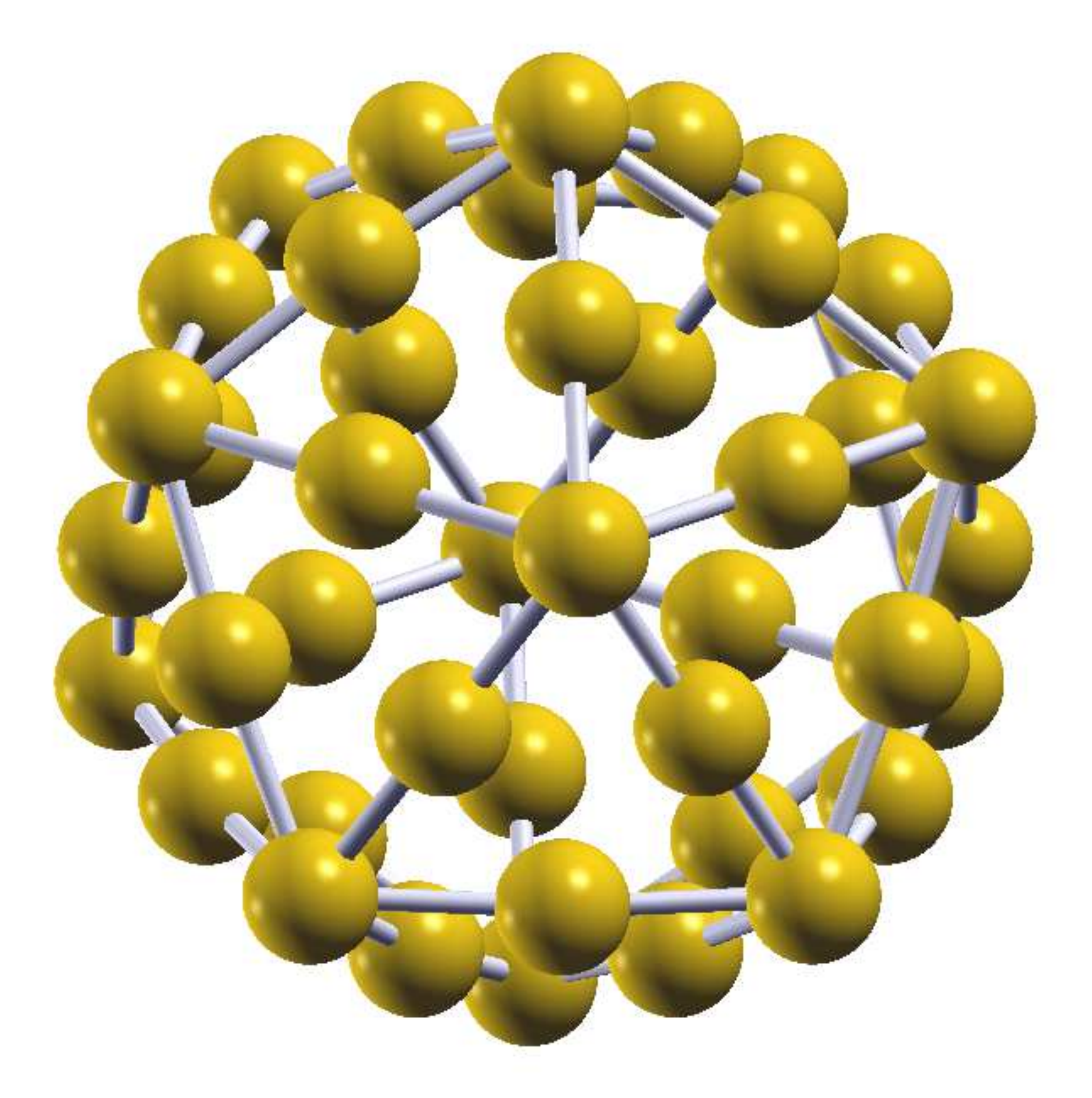}}~~~~~~~
		\subfloat[\ce{Au50}]{\includegraphics[width=0.3\linewidth]{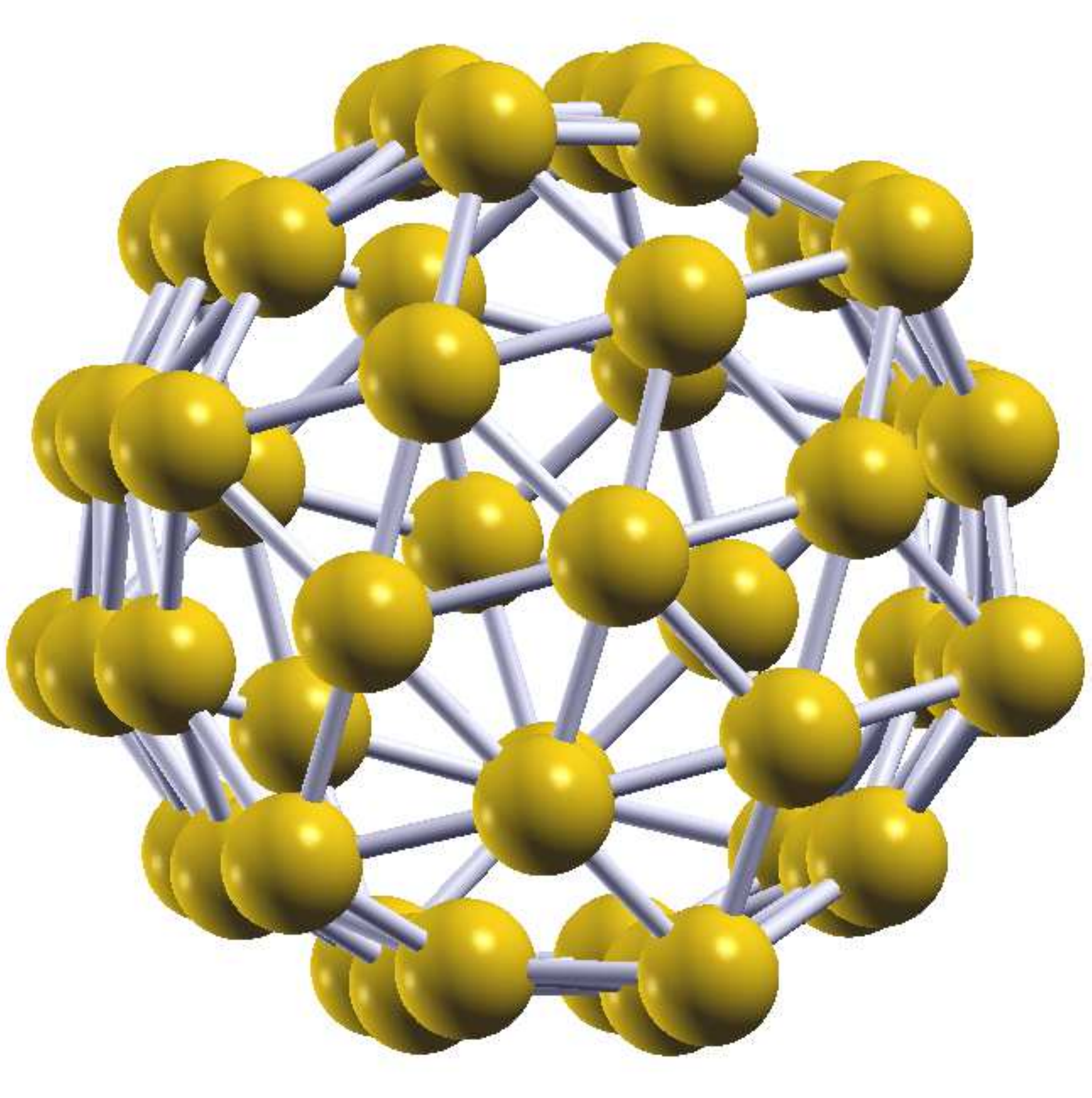}}
		\\
		\subfloat[\ce{Au72}]{\includegraphics[width=0.3\linewidth]{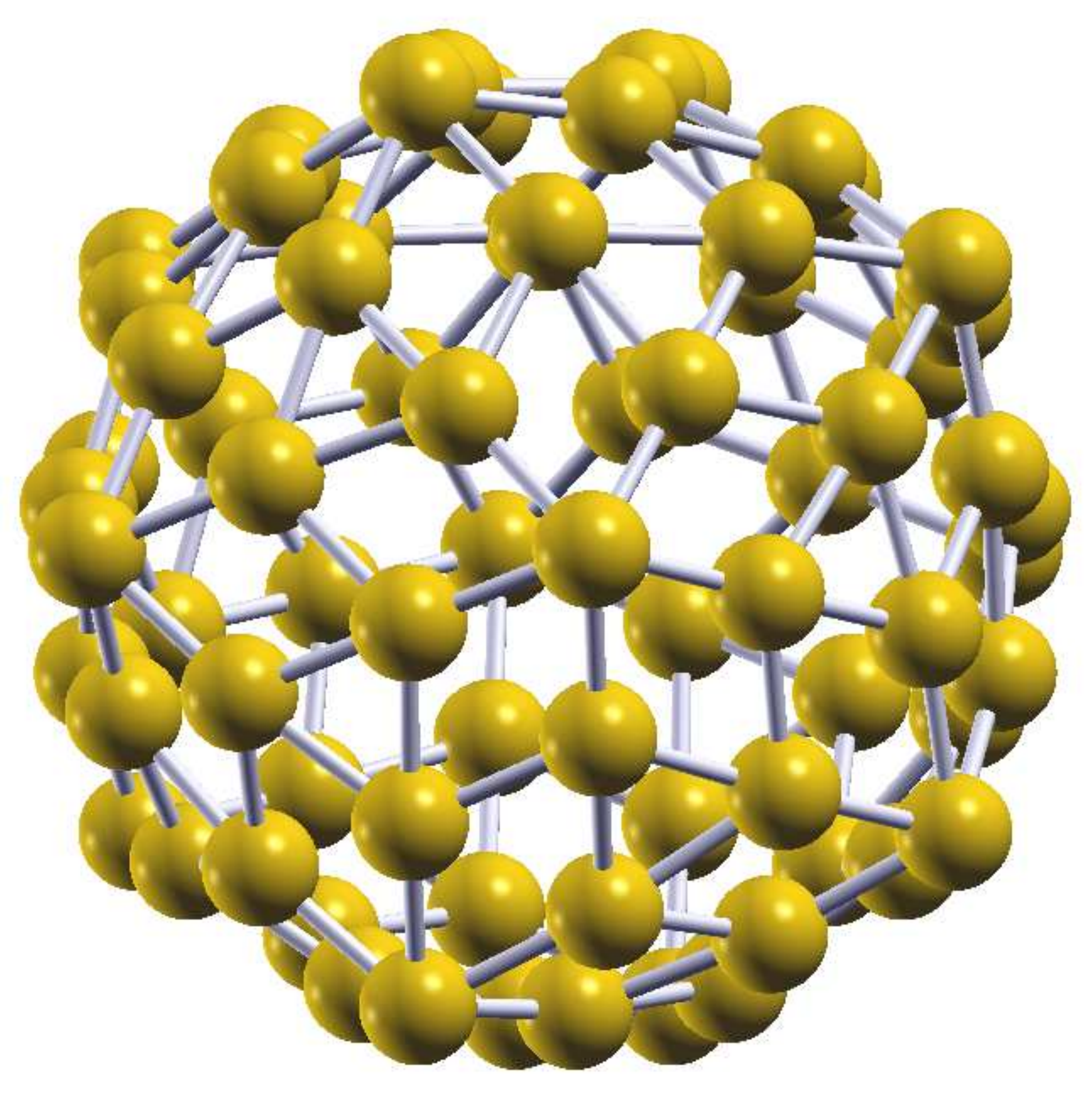}}~~~~~~~
		\subfloat[\ce{Au92}]{\includegraphics[width=0.3\linewidth]{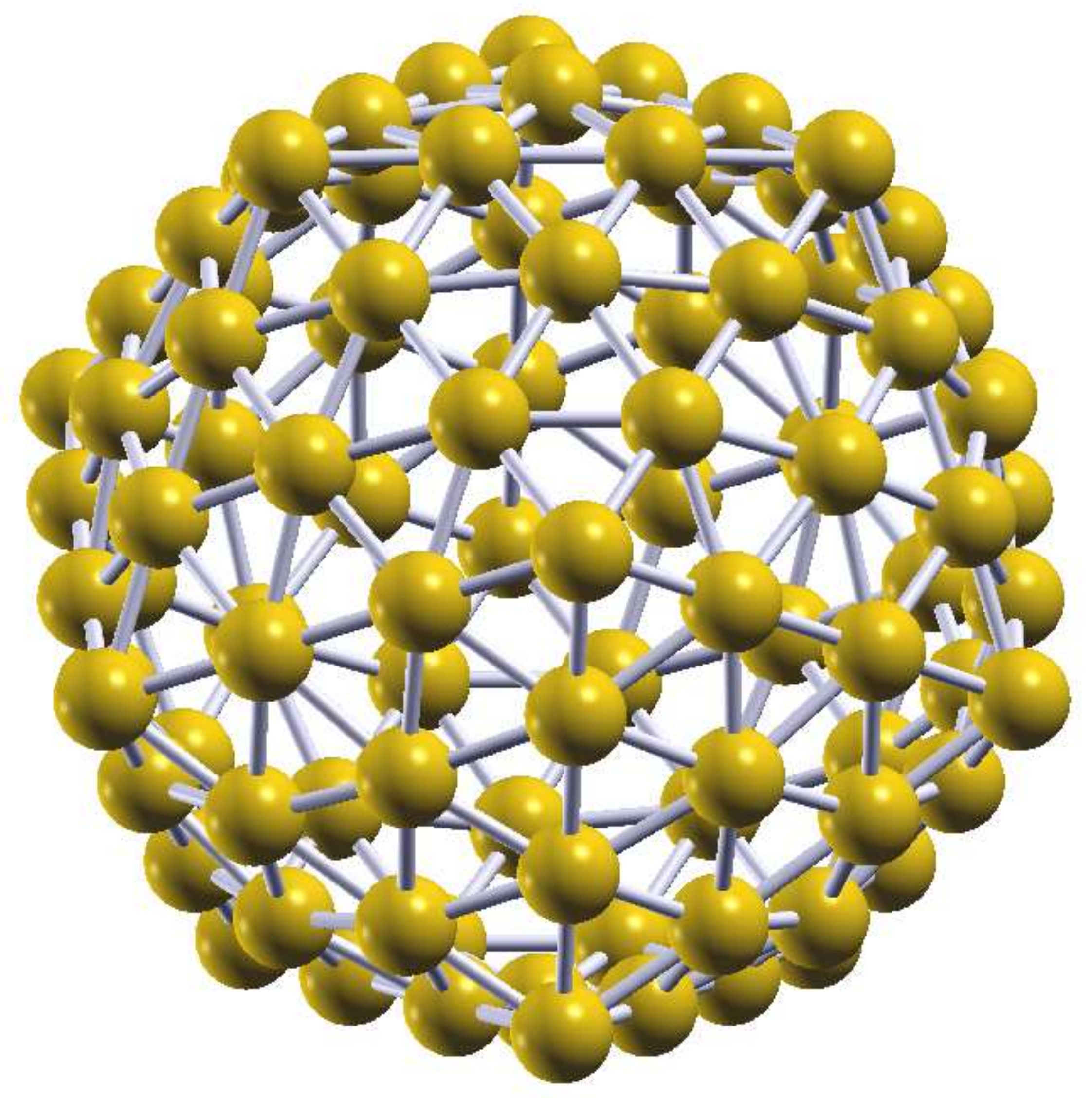}}~~~~~~~
		\subfloat[\ce{Au147}]{\includegraphics[width=0.3\linewidth]{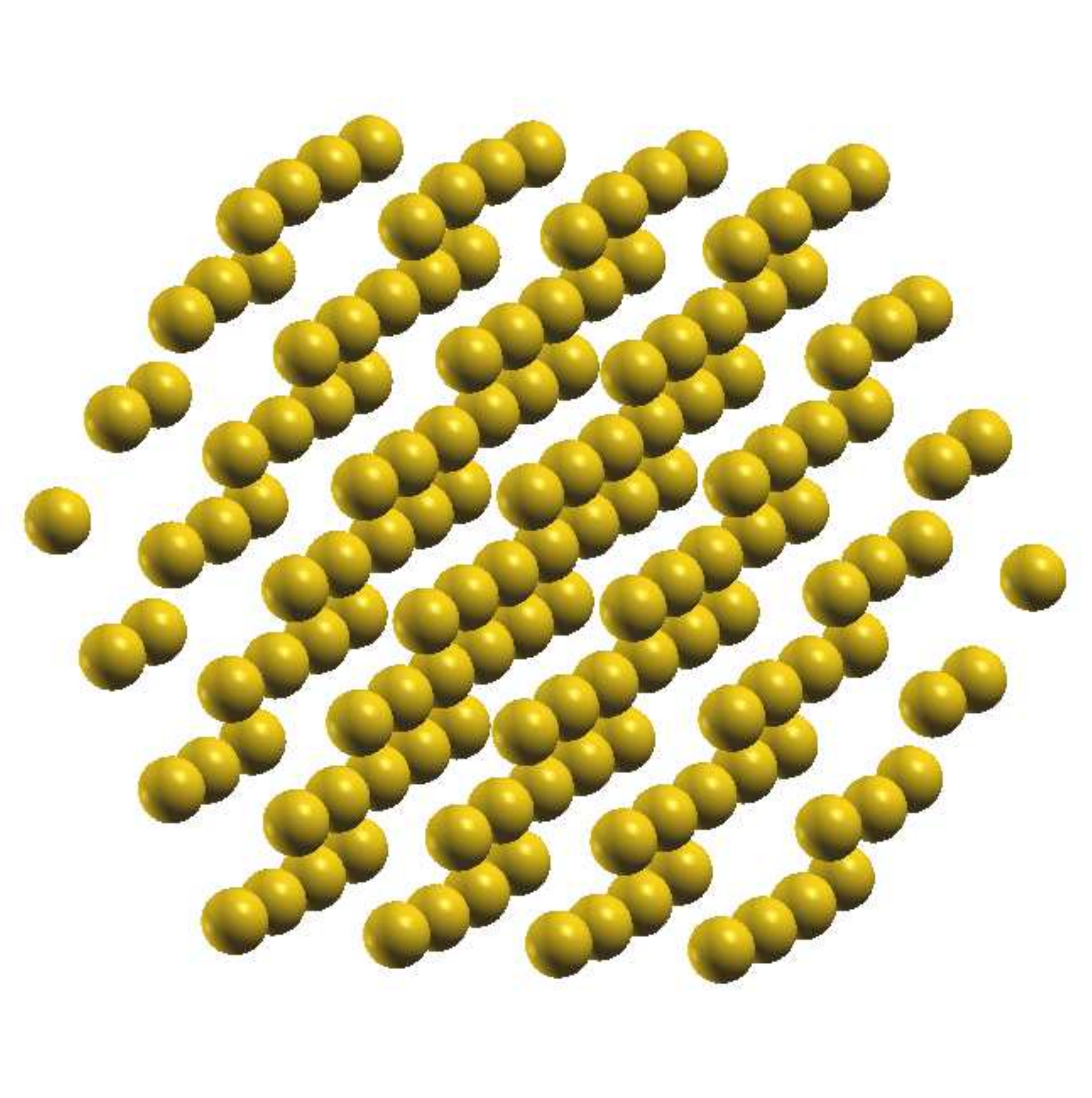}}~~~~~~~
		\\
		\subfloat[\ce{Au309}]{\includegraphics[width=0.3\linewidth]{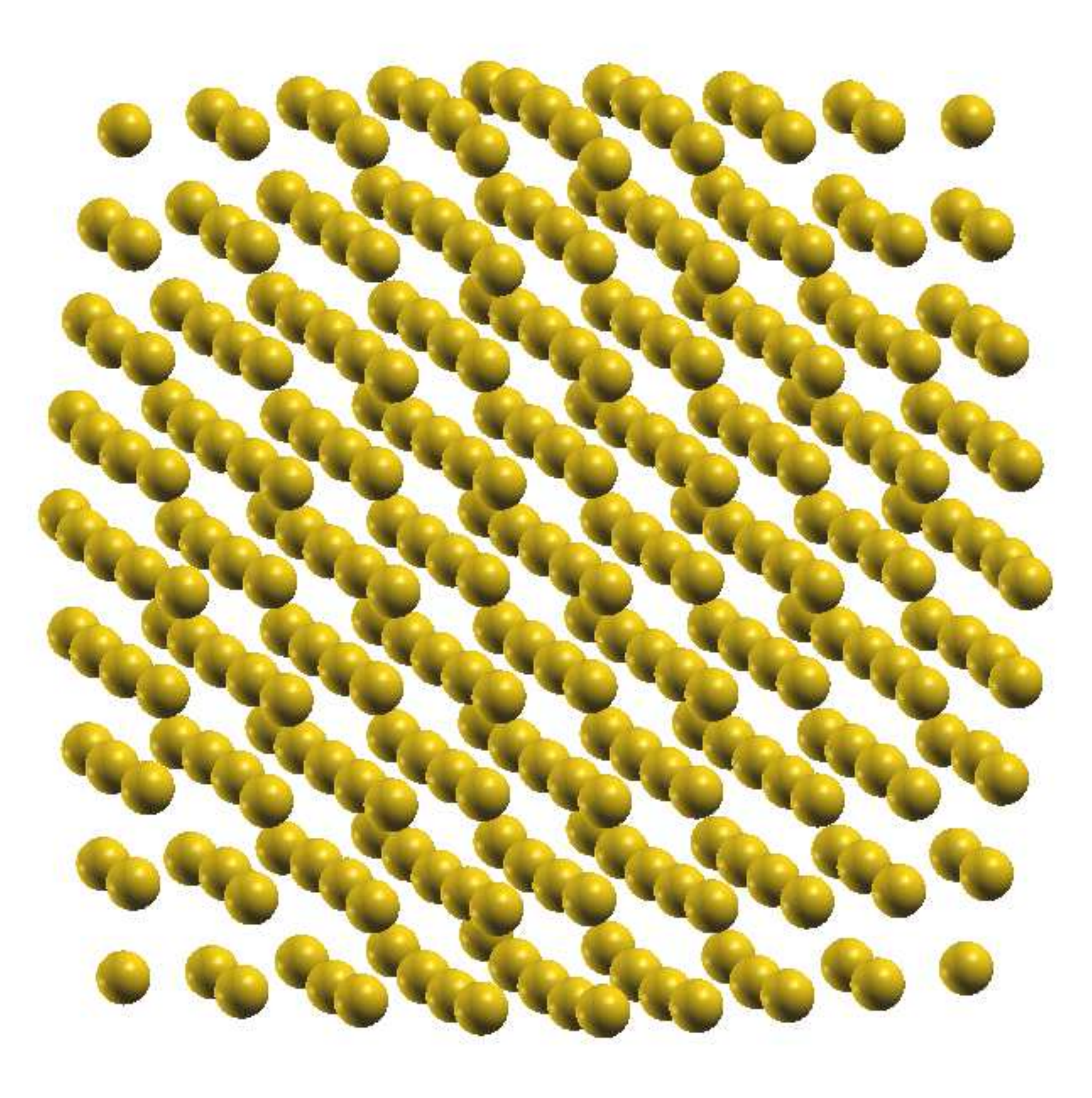}}~~~~~~~
	\end{figure}

	\begin{figure}[!htbp]
		\centering
		\caption{The configurations in the unit cell for the multicomponent periodic systems}
		\label{fig:structure}
		\subfloat[\ce{NdCu2Si2}]{\includegraphics[width=0.3\linewidth]{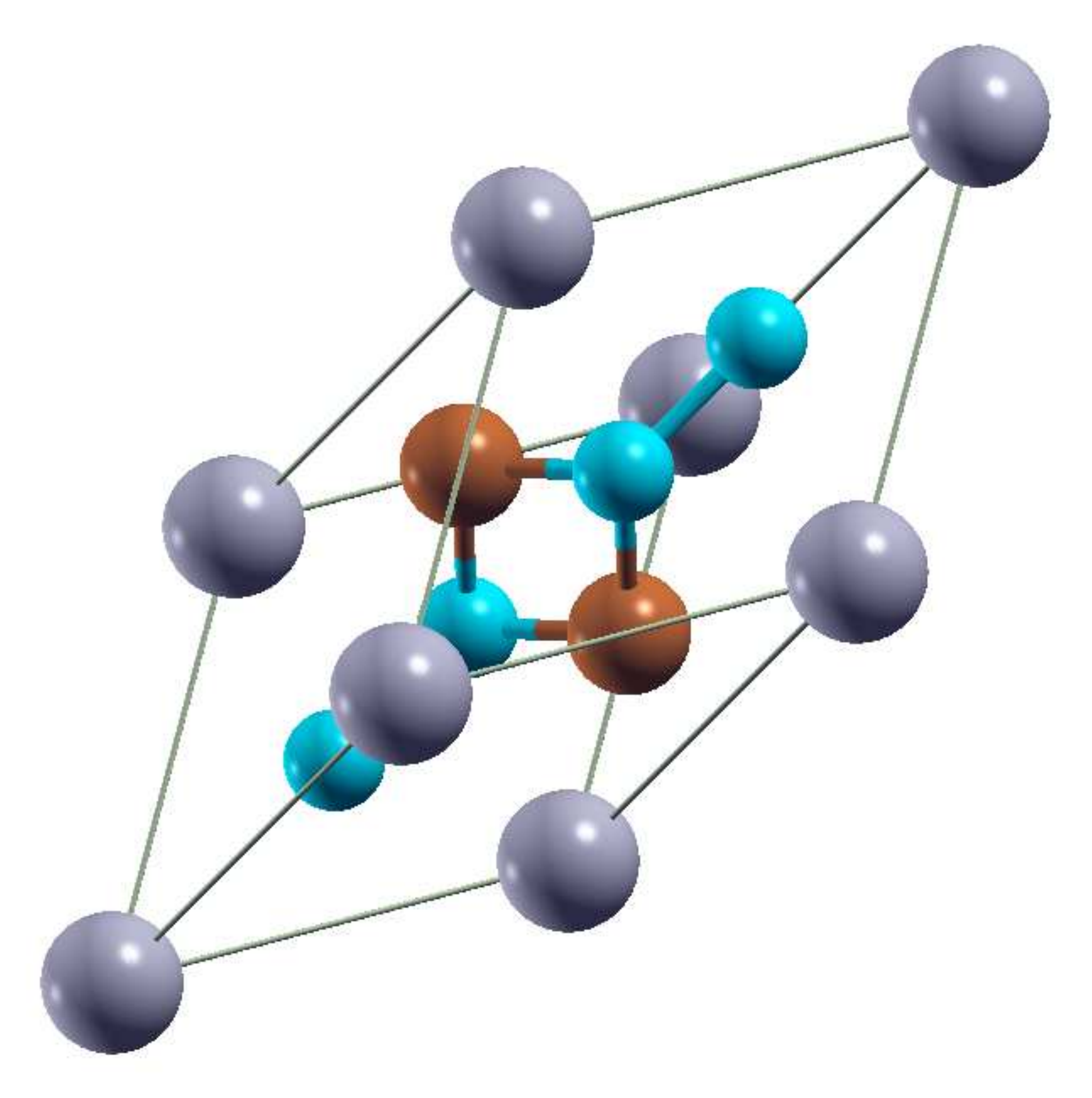}}~~~~~~~
		\subfloat[\ce{AlCrTiV}]{\includegraphics[width=0.3\linewidth]{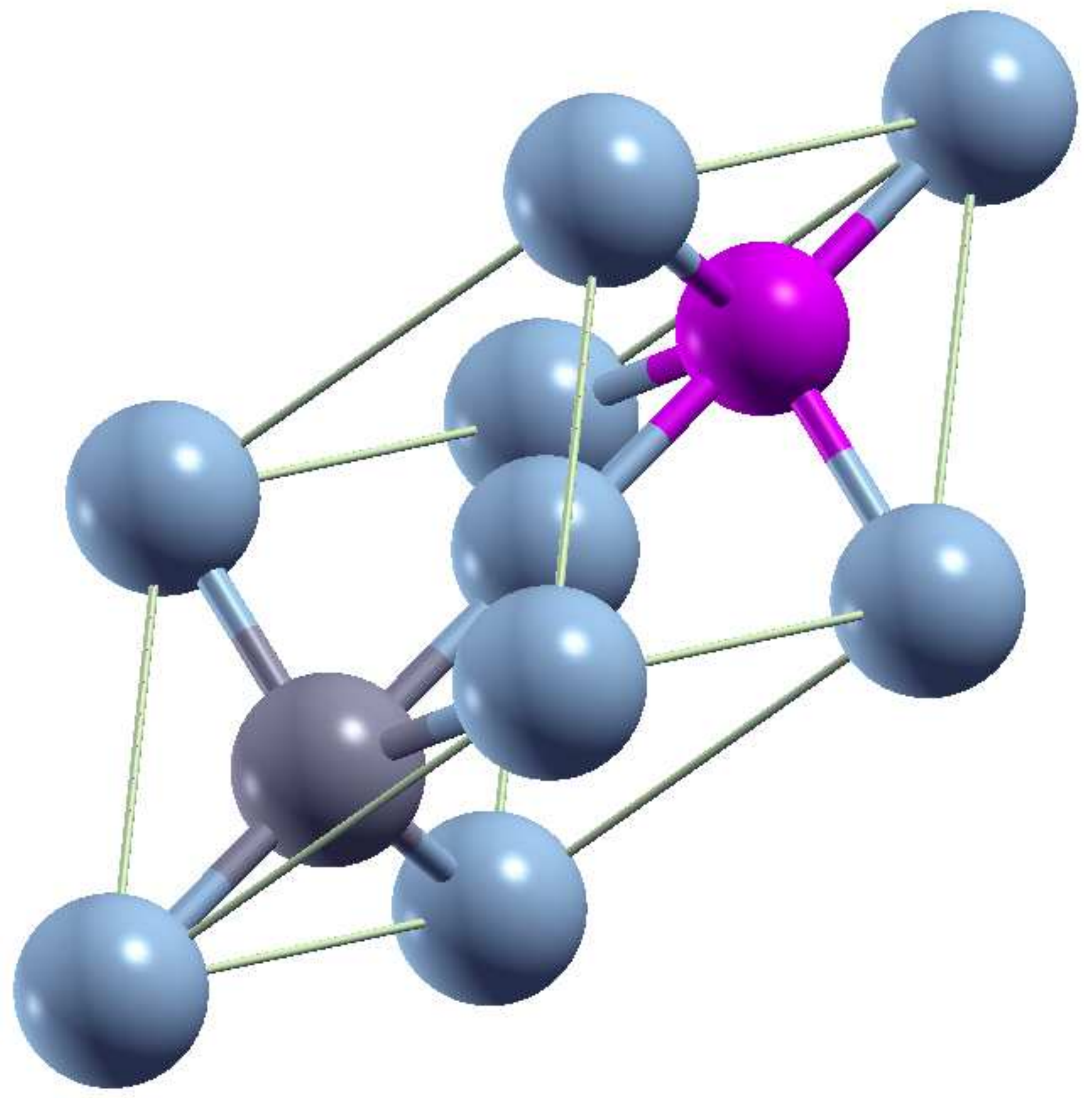}}~~~~~~~
	\end{figure}

	In our numerical experiments, we do not restrict the step sizes to satisfy \eqref{eq:step-lowup-bounded} for some given parameters $\underline{c}$ and $\bar{c}$, which can be viewed as $\underline{c}=0,~\bar{c}=+\infty$. Although  \eqref{eq:step-lowup-bounded} is necessary in our theoretical analysis, numerical results show that the step sizes can be more relaxed. Therefore, we directly apply the step size strategies \ref{step-energy}, \ref{step-d1} or \ref{step-d1-2} to get the step sizes in the numerical simulations. 
	
	In the following tables and figures, PCG-S1, PCG-S2 and PCG-S3 stand for the corresponding PCG method (Algorithm \ref{algo:pcg}) when the step size strategy \ref{step-energy}, strategy \ref{step-d1} and strategy \ref{step-d1-2} are applied, respectively. We denote the restarted versions Algorithm \ref{algo:rpcg1} and Algorithm \ref{algo:rpcg2} by PCG-S$\star$-r1 and PCG-S$\star$-r2 respectively, where $\star$ can be 1, 2 or 3. We mention that ``Error" for the SCF iteraions is the error of density and ``Error" for the PCG methods is $\left(\|\frac12\nabla_\Psi\mathcal{F}\|^2+\|\nabla_\eta\mathcal{F}\|_{sF}^2\right)^{1/2}$.
	
	We will compare our PCG methods with the SCF iterations. It is known that we have to solve a linear eigenvalue problem at each SCF iteration, for which the Davidson iterative diagonalization and the CG diagonalization are commonly used in Quantum ESPRESSO. The Davidson iterative diagonalization is faster, but the CG diagonalization uses less memory and is more robust \cite{quantum}.
	
	We list all the parameters used in our numerical experiments. The Ultrasoft pseudopotentials and the Gaussian smearing with $\sigma=0.05$ Ry are applied for gold clusters. We use the DY approach to get the CG parameter and the QR strategy as \eqref{eq:qr-c} for the orthogonalization operation.  We apply $\theta^{(n)}=\min\{0.8,\sqrt{\|D_\Psi^{(n)}\|_{\infty}^2+\|D_\eta^{(n)}\|_{sF,\infty}^2}\}$ for strategies \ref{step-energy} and \ref{step-d1} and $\theta_\Psi^{(n)}=\min\{0.8,\|D_\Psi^{(n)}\|_{\infty}\},~\theta_\eta^{(n)}=\min\{0.8,\|D_\eta^{(n)}\|_{sF,\infty}\}$ for strategy \ref{step-d1-2}. We set $t^{\text{min}}=t^{\text{min}}_{\Psi}=t^{\text{min}}_{\eta}=0.001$ and initial trial step sizes $t^{\text{trial}}=t_\Psi^{\text{trial}}=t_\eta^{\text{trial}}=0.4$. For the restarted versions, we set $\gamma=0.5$ and $a=1$. The convergence criterion is 
	\[
	\left(\|\frac12\nabla_\Psi\mathcal{F}\|^2+\|\nabla_\eta\mathcal{F}\|^2_{sF}\right)^{1/2}<1.0\times 10^{-5}
	\]
	for the PCG method and its restarted versions, and the convergence threshold for density is $1.0\times 10^{-9}$ for the SCF iterations. For the SCF iterations, We apply the Broyden mixing method. The initial guess for the wavefunctions is generated by the superposition of atomic orbitals \cite{quantum} if not specified.
	
	We see that whether or not to restart has almost no effect for the simulation of gold clusters for the strategies \ref{step-d1} and \ref{step-d1-2}. As a result, we mainly show the numerical results obtained by the PCG method (Algorithm \ref{algo:pcg}) for gold clusters. In addition, we will also mention some improvement of the restarting approach for the strategy \ref{step-energy} in Figure \ref{fig:Au92}.
	
	First, we take a look at the results of all the gold clusters. The results obtained by the PCG method (Algorithm \ref{algo:pcg}) based on different step size strategies are listed in Table \ref{tab:differentStep1}. In Table \ref{tab:differentStep1}, ``Iter." means the number of iterations required to terminate the algorithm and ``A.T.P.I" is the average CPU time required per iteration. As shown in Table \ref{tab:differentStep1}, the strategy \ref{step-d1-2} with different step sizes for $\Psi$ and $\eta$ is indeed the best. More precisely, the strategy \ref{step-d1-2} needs less iteration and the CPU time to achieve similar accuracy than the strategies \ref{step-energy} and \ref{step-d1} with same step sizes for $\Psi$ and $\eta$, especially for large systems. We see that the strategies \ref{step-d1} and \ref{step-d1-2} are more expensive than the strategy \ref{step-energy} per iteration. However, by comparing the strategies \ref{step-energy} and \ref{step-d1}, we see that the strategy \ref{step-energy} need more iterations to achieve the same accuracy than the strategy \ref{step-d1}. Even the iterations for \ce{Au18}, \ce{Au72}, \ce{Au92}, \ce{Au147} and \ce{Au309} do not converge after 200 iterations under the strategy \ref{step-energy}. We point out that whether or not to restart has no effect on the strategies \ref{step-d1} and \ref{step-d1-2} for the simulation of these gold clusters under the convergence criterion discussed in this section. However, it will improve the convergence of the iteration a little for the strategy \ref{step-energy}. If we restart the PCG method as Algorithm \ref{algo:rpcg1}, the calculations for \ce{Au18}, \ce{Au72} and \ce{Au92} can also converge under the strategy \ref{step-energy}. Due to limited space, we only show the results of \ce{Au92} obtained by the restarted PCG method I (Algorithm \ref{algo:rpcg1}) later.
	
	\begin{table}[!htbp]
		\centering
		\caption{The numerical results for gold clusters obtained by the PCG method (Algorithm \ref{algo:pcg}) based on different step size strategies.}
		\label{tab:differentStep1}
			\begin{tabular}{|c|ccccc|}
				\hline
				Algorithm                & Energy (Ry)             & Iter.    & Error        & CPU time (s)    & A.T.P.I (s)    \\ \hline
				\multicolumn{6}{|c|}{\ce{Au14} \quad $ N_G = 322453 $ \quad $ N = 92 $ \quad $ cores = 36 $}  \\ \hline
				PCG-S1      & -1194.49861028     & 90      & 9.3E-6    & 587.0      & 6.52         \\
				PCG-S2      & -1194.49861028     & 50      & 9.7E-6    & 397.7      & 7.95         \\
				PCG-S3      & -1194.49861028     & 37      & 8.5E-6    & 299.2      & 8.09         \\ \hline
				\multicolumn{6}{|c|}{\ce{Au18} \quad $ N_G = 322453 $ \quad $ N = 119 $ \quad $ cores = 36 $} \\ \hline
				PCG-S1      & -1536.01945578     & 200     & 1.4E-5    & 1626.6     & 8.13         \\
				PCG-S2      & -1536.01945578     & 62      & 7.2E-6    & 647.3      & 10.44        \\
				PCG-S3      & -1536.01945578     & 37      & 8.9E-6    & 384.6      & 10.39        \\ \hline
				\multicolumn{6}{|c|}{\ce{Au20} \quad $ N_G = 322453 $ \quad $ N = 132 $ \quad $cores= 36 $}   \\ \hline
				PCG-S1      & -1706.76524000     & 109     & 9.1E-6    & 963.3      & 8.84         \\
				PCG-S2      & -1706.76524000     & 55      & 8.3E-6    & 621.6      & 11.30        \\
				PCG-S3      & -1706.76524000     & 38      & 9.1E-6    & 429.7      & 11.31        \\ \hline
				\multicolumn{6}{|c|}{\ce{Au32} \quad $ N_G = 429409 $ \quad $ N = 211 $ \quad $cores= 36 $}   \\ \hline
				PCG-S1      & -2731.11762824     & 90      & 8.3E-6    & 1808.6     & 20.10        \\
				PCG-S2      & -2731.11762824     & 48      & 8.4E-6    & 1270.7     & 26.47        \\
				PCG-S3      & -2731.11762824     & 38      & 8.3E-6    & 1019.1     & 26.82        \\ \hline
				\multicolumn{6}{|c|}{\ce{Au42} \quad $ N_G = 429409 $ \quad $ N = 277 $ \quad $ cores = 36 $} \\ \hline
				PCG-S1      & -3584.66580292     & 78      & 1.0E-6    & 2133.8     & 27.36        \\
				PCG-S2      & -3584.66580292     & 55      & 5.8E-6    & 2011.6     & 36.57        \\
				PCG-S3      & -3584.66580292     & 39      & 8.5E-6    & 1390.6     & 35.66        \\ \hline
				\multicolumn{6}{|c|}{\ce{Au50}\quad $N_G=429409$\quad $N=330$\quad $cores=36$}                \\ \hline
				PCG-S1      & -4267.69535810     & 114     & 6.8E-6    & 3700.7     & 32.46        \\
				PCG-S2      & -4267.69535810     & 58      & 9.4E-6    & 2626.5     & 45.28        \\
				PCG-S3      & -4267.69535810     & 39      & 9.2E-6    & 1786.2     & 45.80        \\ \hline
				\multicolumn{6}{|c|}{\ce{Au72}\quad $N_G=556667$\quad $N=475$\quad $cores=36$}    \\ \hline
				PCG-S1    & -6145.78233806    & 200   & 1.1E-4  & 13959.2  & 69.80     \\
				PCG-S2    & -6145.78233806    & 89    & 9.8E-6  & 8557.1   & 96.15     \\
				PCG-S3    & -6145.78233806    & 40    & 9.0E-6  & 3760.7   & 94.02     \\ \hline
				\multicolumn{6}{|c|}{\ce{Au92}\quad $N_G=556667$\quad $N=607$\quad $cores=36$}    \\ \hline
				PCG-S1    & -7853.07110320    & 200   & 2.3E-5  & 19697.0  & 98.49     \\
				PCG-S2    & -7853.07110320    & 91    & 7.9E-6  & 12229.9  & 134.39    \\
				PCG-S3    & -7853.07110320    & 40    & 9.8E-6  & 5535.4   & 138.39    \\ \hline
				\multicolumn{6}{|c|}{\ce{Au147}\quad $N_G=1320073$\quad $N=971$\quad $cores=72$}  \\ \hline
				PCG-S1    & -12547.62980551   & 200   & 4.4E-5  & 37056.7  & 185.28    \\
				PCG-S2    & -12547.62980551   & 88    & 9.7E-6  & 23166.6  & 263.26    \\
				PCG-S3    & -12547.62980551   & 42    & 9.0E-6  & 11193.5  & 266.51    \\ \hline
				\multicolumn{6}{|c|}{\ce{Au309}\quad $N_G=1320073$\quad $N=2040$\quad $cores=72$} \\ \hline
				PCG-S1    & -26379.41930504   & 200   & 1.5E-4  & 134638.7 & 673.19    \\
				PCG-S2    & -26379.41930504   & 124   & 9.1E-6  & 119025.8 & 959.89    \\
				PCG-S3    & -26379.41930507   & 51    & 6.8E-6  & 49831.2  & 977.08    \\ \hline
			\end{tabular}
	\end{table}
	
	To compare the three step size strategies more clearly, we take \ce{Au92} as an example and show the convergence curves for $\mathcal{F}-\mathcal{F}_{\text{min}}$, $\frac12\|\nabla_\Psi\mathcal
	F\|$ and $\|\nabla_{\eta}\mathcal{F}\|_{sF}$ in Figure \ref{fig:Au92}, where $\mathcal{F}_{\text{min}}$ is a high-accuracy approximation of the exact total energy. We also illustrate the benefit of the restarting approach for the strategy \ref{step-energy}. First, the strategy \ref{step-d1-2} is indeed faster than the other two strategies. Secondly, by comparing the convergence curves for the error of the energy, we see that the strategy \ref{step-energy} is not much different from the strategy \ref{step-d1}, and the strategy \ref{step-energy} seems to be better when the energy has not converged. But there may be some fluctuation for the strategy \ref{step-energy} when the energy almost converges. From the convergence curves for $\frac12\|\nabla_\Psi\mathcal
	F\|$ and $\|\nabla_{\eta}\mathcal{F}\|_{sF}$, we see that the descent speed of the gradient obtained by the strategy \ref{step-energy} slows down suddenly when the energy almost converges and then is much smaller than the strategy \ref{step-d1}. Finally, by comparing PCG-S1 and PCG-S1-r1, we find that the restarting approach does improve the convergence of the iteration for the strategy \ref{step-energy}.
	
	\begin{figure}[!htbp]
		\centering
		\includegraphics[width=0.68\textwidth]{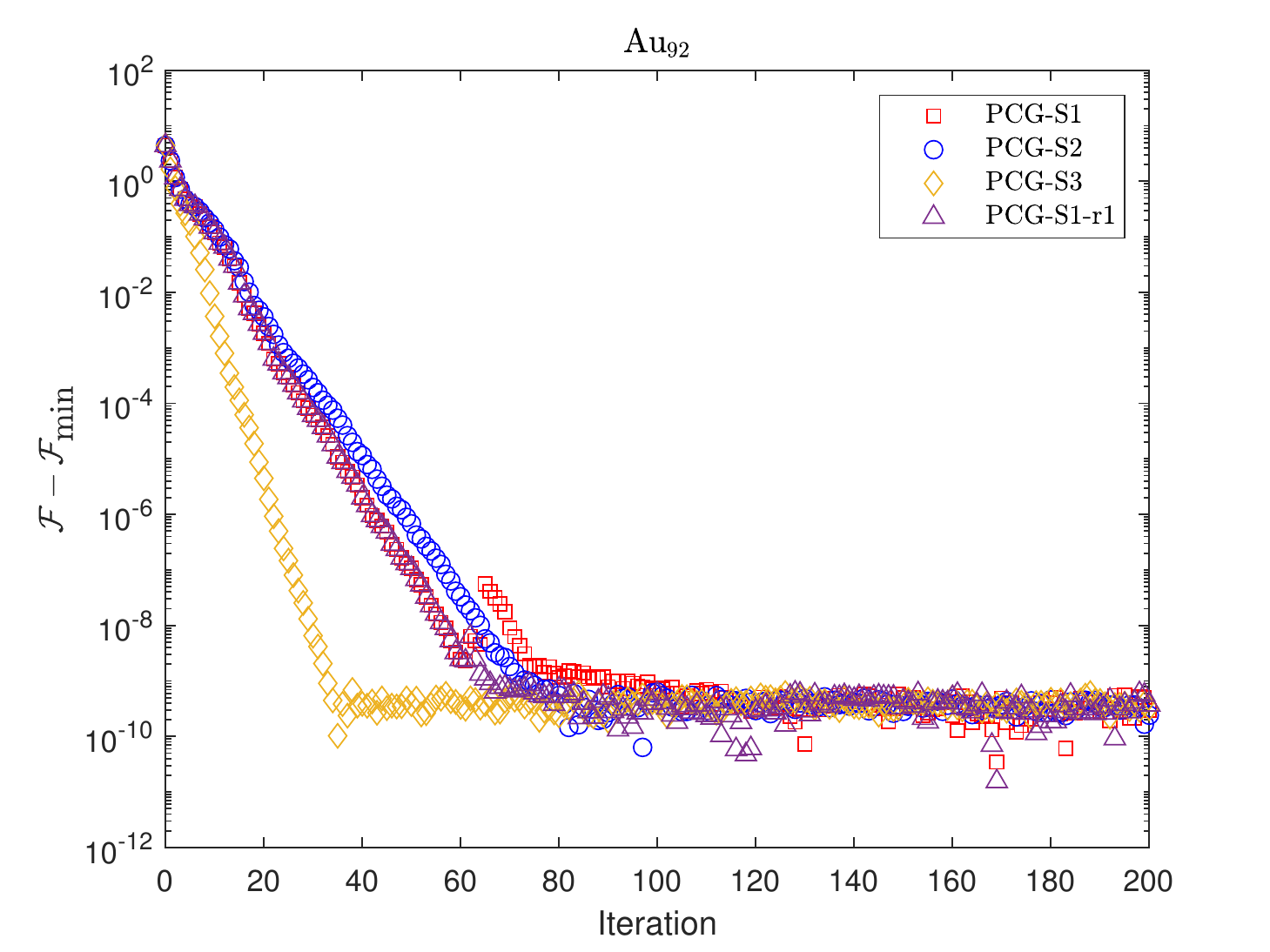}\\
		\includegraphics[width=0.68\textwidth]{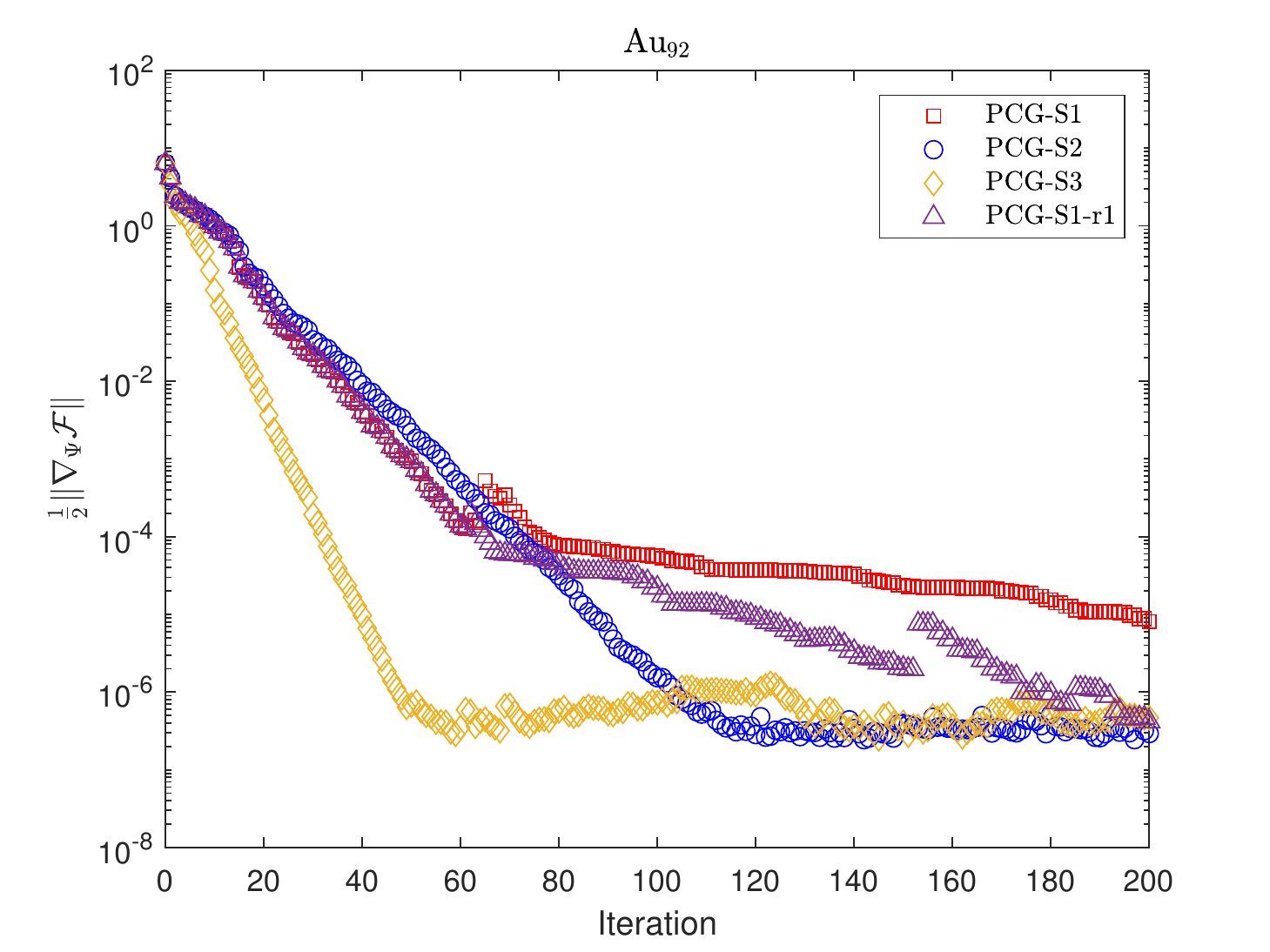}\\
		\includegraphics[width=0.68\textwidth]{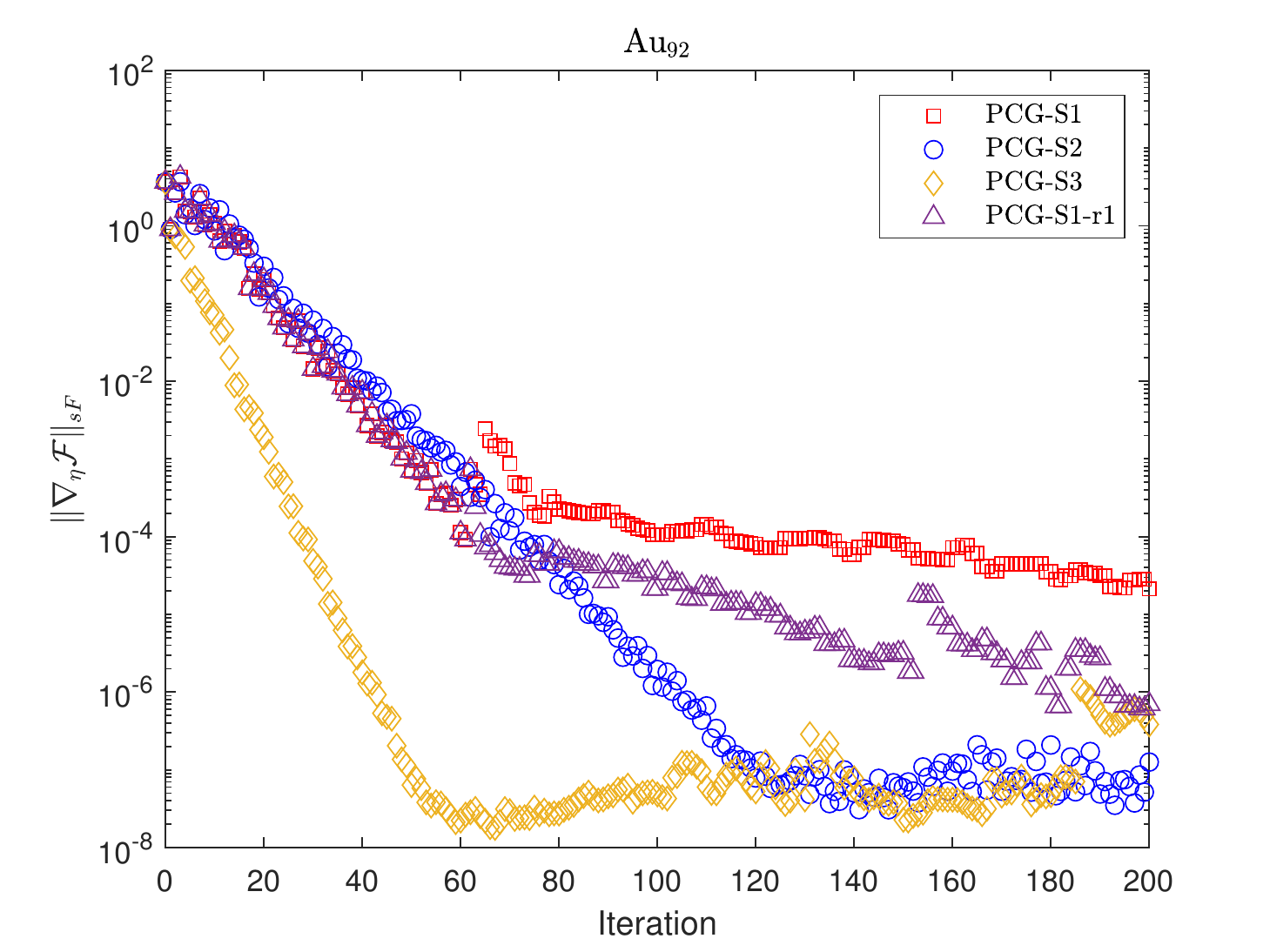}
		\caption{Convergence curves for $\mathcal{F}-\mathcal{F}_{\text{min}}$, $\frac12\|\nabla_\Psi\mathcal
			F\|$ and $\|\nabla_{\eta}\mathcal{F}\|_{sF}$ obtained by different step size strategies for \ce{Au92}.}
		\label{fig:Au92}
	\end{figure}

	We conclude from the above that the strategy \ref{step-d1-2} seems to be the best one among the three strategies. We then choose the PCG method based on the step size strategy \ref{step-d1-2} to be compared with the SCF iterations based on the CG diagonalization. The detailed results are shown in Table \ref{tab:scf-pcg-s3}. We see from Table \ref{tab:scf-pcg-s3} that, apart from \ce{Au14}, the PCG method converges faster than the SCF iterations, especially for large scale systems. For instance, the PCG method converges in half the CPU time of SCF iterations for \ce{Au42}, and the PCG method converges in less than 1/3 the CPU time of the SCF iterations for \ce{Au147}. We also mention that the energy obtained by the PCG method is slightly smaller than that obtained by SCF iterations for \ce{Au20}, \ce{Au42}, \ce{Au50}, \ce{Au92} and \ce{Au309}, which means that SCF iterations may require a smaller convergence threshold to obtain the same energy obtained by the PCG method. However, SCF iterations has already cost more CPU time even with the accuracy in the table.
	\begin{table}[!htbp]
		\centering
		\caption{Comparison of the SCF iterations based on the CG diagonalization and the PCG method based on the step size strategy \ref{step-d1-2}. The density mixing factor for the SCF iterations is 0.3.}
		\label{tab:scf-pcg-s3}
			\begin{tabular}{|c|cccc|}
				\hline
				Algorithm                  & Energy (Ry)                & Iter.       & Error            & CPU time (s)       \\ \hline
				\multicolumn{5}{|c|}{\ce{Au14} \quad $ N_G = 322453 $ \quad $ N = 92 $ \quad $ cores = 36 $}  \\ \hline
				SCF                 & -1194.49861028        & 16         & 9.5E-10       & 271.9         \\
				PCG-S3        & -1194.49861028        & 37         & 8.5E-6        & 299.2         \\ \hline
				\multicolumn{5}{|c|}{\ce{Au18} \quad $ N_G = 322453 $ \quad $ N = 119 $ \quad $ cores = 36 $} \\ \hline
				SCF                 & -1536.01945578        & 18         & 5.5E-10       & 452.7         \\
				PCG-S3        & -1536.01945578        & 37         & 8.9E-6        & 384.6         \\ \hline
				\multicolumn{5}{|c|}{\ce{Au20} \quad $ N_G = 322453 $ \quad $ N = 132 $ \quad $cores= 36 $}   \\ \hline
				SCF                 & -1706.76523999        & 15         & 8.3E-10       & 470.2         \\
				PCG-S3        & -1706.76524000        & 38         & 9.1E-6        & 429.7         \\ \hline
				\multicolumn{5}{|c|}{\ce{Au32} \quad $ N_G = 429409 $ \quad $ N = 211 $ \quad $cores= 36 $}   \\ \hline
				SCF                 & -2731.11762824        & 16         & 3.5E-10       & 1476.1        \\
				PCG-S3        & -2731.11762824        & 38         & 8.3E-6        & 1019.1        \\ \hline
				\multicolumn{5}{|c|}{\ce{Au42} \quad $ N_G = 429409 $ \quad $ N = 277 $ \quad $ cores = 36 $} \\ \hline
				SCF                 & -3584.66580291        & 20         & 2.2E-11       & 2870.1        \\
				PCG-S3        & -3584.66580292        & 39         & 8.5E-6        & 1390.6        \\ \hline
				\multicolumn{5}{|c|}{\ce{Au50}\quad $N_G=429409$\quad $N=330$\quad $cores=36$}                \\ \hline
				SCF                 & -4267.69535809        & 17         & 7.8E-10       & 3629.5        \\
				PCG-S3        & -4267.69535810        & 39         & 9.2E-6        & 1786.2        \\ \hline
				\multicolumn{5}{|c|}{\ce{Au72}\quad $N_G=556667$\quad $N=475$\quad $cores=36$}                \\ \hline
				SCF                 & -6145.78233806        & 24         & 1.9E-10       & 10766.7       \\
				PCG-S3        & -6145.78233806        & 40         & 9.0E-6        & 3760.7        \\ \hline
				\multicolumn{5}{|c|}{\ce{Au92}\quad $N_G=556667$\quad $N=607$\quad $cores=36$}                \\ \hline
				SCF                 & -7853.07110315        & 21         & 4.8E-10       & 16142.4       \\ 
				PCG-S3        & -7853.07110320        & 40         & 9.8E-6        & 5535.4        \\ \hline
				\multicolumn{5}{|c|}{\ce{Au147}\quad $N_G=1320073$\quad $N=971$\quad $cores=72$}              \\ \hline
				SCF                 & -12547.62980551       & 30         & 3.8E-10       & 39669.2       \\
				PCG-S3        & -12547.62980551       & 42         & 9.0E-6        & 11193.5       \\ \hline
				\multicolumn{5}{|c|}{\ce{Au309}\quad $N_G=1320073$\quad $N=2040$\quad $cores=72$}             \\ \hline
				SCF                 & -26379.41930501       & 23         & 3.5E-10       & 154451.0      \\
				PCG-S3        & -26379.41930507       & 51         & 6.8E-6        & 49831.2       \\ \hline
			\end{tabular}
	\end{table}

	Now, we show the numerical results for the two complicated periodic systems shown in Figure \ref{fig:structure}. Different from the gold clusters, for these two systems, the spin polarization is taken into account and the cases using different initial guesses of wavefunctions are tested. Since these two systems show more obvious metallicity, more smearing strategies may be used. Here, we consider the Gaussian smearing and the Marzari–Vanderbilt smearing, which are some typical smearing functions used in the simulation of metallic systems. The detailed results are reported in Tables \ref{tab:periodic-gs} and \ref{tab:periodic-mv}. Here, $N_{\mathrm{k}}=2|\mathcal{K}|$, ``atomic" means that the initial guess of wavefunctions is generated by the superposition of atomic orbitals, and ``atomic+random" means that the initial guess of wavefunctions is generated by the superposition of atomic orbitals plus a superimposed ``randomization" of atomic orbitals \cite{quantum}. We observe from Tables \ref{tab:periodic-gs} and \ref{tab:periodic-mv} that, for both the two smearing methods, except for the system \ce{NdCu2Si2} with the initial guess of wavefunctions being given by the superposition of atomic orbitals, the SCF iterations fail to converge after 500 iterations. We also see that both the PCG method and the restarded PCG methods can obtain convergent approximations for both the two systems, no matter what kind of initial guesses and smearing methods are used. Comparing the results for PCG-S3 with the results for PCG-S3-r1 and PCG-S3-r2, we observe that the restarting strategy does accelerate the convergence of the PCG method except for the system \ce{NdCu2Si2} calculated by PCG-S3-r2 with the initial guesses of wavefunctions being given by ``atomic+random'' and the Gaussian smearing. Comparing the results for PCG-S3-r1 with the results for PCG-S3-r2, we also see that the second restarting approach (Algorithm \ref{algo:rpcg2}) is better than the first restarting approach (Algorithm \ref{algo:rpcg1}) for the system \ce{AlCrTiV}, but the first restarting approach is better than the second restarting approach for the system \ce{NdCu2Si2}. We conclude that the PCG method and the restarted PCG methods are more stable when different initial orbitals are used and our methods are suitable for different smearing functions.

	\begin{table}
		\caption{Comparison of the SCF iterations based on the Davidson iterative diagonalization, the PCG method and the restarted PCG methods. The density mixing factor for the SCF iterations is 0.4, and the Gaussian smearing with $\sigma=0.01$ Ry is applied.}
		\label{tab:periodic-gs}
		\centering
		\begin{tabular}{|ccccc|}
			\hline
			\multicolumn{1}{|c|}{Algorithm}                  & \multicolumn{1}{c|}{Initial orbitals} & Energy (Ry)    & Iter. & Error   \\ \hline
			\multicolumn{5}{|c|}{\ce{NdCu2Si2} \quad $ N_G = 3837$ \quad $ N = 36 $ \quad $N_k=576$ \quad $ cores = 36 $}                    \\ \hline
			\multicolumn{1}{|c|}{\multirow{2}{*}{SCF}}       & \multicolumn{1}{c|}{atomic}           & -1368.00296219 & 24    & 1.9E-10 \\ \cline{2-5} 
			\multicolumn{1}{|c|}{}                           & \multicolumn{1}{c|}{atomic+random}    & -1367.99467076 & 500   & 6.8E-6  \\ \hline
			\multicolumn{1}{|c|}{\multirow{2}{*}{PCG-S3}}    & \multicolumn{1}{c|}{atomic}           & -1368.00296213 & 440   & 9.9E-6  \\ \cline{2-5} 
			\multicolumn{1}{|c|}{}                           & \multicolumn{1}{c|}{atomic+random}    & -1367.99713429 & 255   & 9.8E-6  \\ \hline
			\multicolumn{1}{|c|}{\multirow{2}{*}{PCG-S3-r1}} & \multicolumn{1}{c|}{atomic}           & -1368.00296213 & 313   & 9.9E-6  \\ \cline{2-5} 
			\multicolumn{1}{|c|}{}                           & \multicolumn{1}{c|}{atomic+random}    & -1367.99713429 & 239   & 9.5E-6  \\ \hline
			\multicolumn{1}{|c|}{\multirow{2}{*}{PCG-S3-r2}} & \multicolumn{1}{c|}{atomic}           & -1368.00296205 & 308   & 8.0E-6  \\ \cline{2-5} 
			\multicolumn{1}{|c|}{}                           & \multicolumn{1}{c|}{atomic+random}    & -1367.99713429 & 290   & 8.8E-6  \\ \hline
			\multicolumn{5}{|c|}{\ce{AlCrTiV} \quad $ N_G = 1759 $ \quad $ N = 25 $ \quad $N_k=144$ \quad $cores= 36 $}                      \\ \hline
			\multicolumn{1}{|c|}{\multirow{2}{*}{SCF}}       & \multicolumn{1}{c|}{atomic}           & -479.31372455  & 500   & 2.1E-4  \\ \cline{2-5} 
			\multicolumn{1}{|c|}{}                           & \multicolumn{1}{c|}{atomic+random}    & -479.31491981  & 500   & 7.0E-4  \\ \hline
			\multicolumn{1}{|c|}{\multirow{2}{*}{PCG-S3}}    & \multicolumn{1}{c|}{atomic}           & -479.36755754  & 200   & 9.9E-6  \\ \cline{2-5} 
			\multicolumn{1}{|c|}{}                           & \multicolumn{1}{c|}{atomic+random}    & -479.36755753  & 283   & 9.3E-6  \\ \hline
			\multicolumn{1}{|c|}{\multirow{2}{*}{PCG-S3-r1}} & \multicolumn{1}{c|}{atomic}           & -479.36755753  & 136   & 9.2E-6  \\ \cline{2-5} 
			\multicolumn{1}{|c|}{}                           & \multicolumn{1}{c|}{atomic+random}    & -479.36755753  & 234   & 9.6E-6  \\ \hline
			\multicolumn{1}{|c|}{\multirow{2}{*}{PCG-S3-r2}} & \multicolumn{1}{c|}{atomic}           & -479.36755753  & 109   & 8.5E-6  \\ \cline{2-5} 
			\multicolumn{1}{|c|}{}                           & \multicolumn{1}{c|}{atomic+random}    & -479.36755754  & 115   & 8.0E-6  \\ \hline
		\end{tabular}
	\end{table}
	
	\begin{table}[!htbp]
		\centering
		\caption{Comparison of the SCF iterations based on the Davidson iterative diagonalization, the PCG method and the restarted PCG methods. The density mixing factor for the SCF iterations is 0.4, and the Marzari–Vanderbilt smearing with $\sigma=0.01$ Ry is applied.}
		\label{tab:periodic-mv}
		\begin{tabular}{|ccccc|}
			\hline
			\multicolumn{1}{|c|}{Algorithm}                  & \multicolumn{1}{c|}{Initial orbitals} & Energy (Ry)    & Iter. & Error   \\ \hline
			\multicolumn{5}{|c|}{NdCu2Si2 \quad $ N_G = 3837$ \quad $ N = 36 $ \quad $N_k=576$ \quad $ cores = 36 $}                    \\ \hline
			\multicolumn{1}{|c|}{\multirow{2}{*}{SCF}}       & \multicolumn{1}{c|}{atomic}           & -1368.00214304 & 24    & 4.0E-10 \\ \cline{2-5} 
			\multicolumn{1}{|c|}{}                           & \multicolumn{1}{c|}{atomic+random}    & -1367.99221653 & 500   & 5.6E-6  \\ \hline
			\multicolumn{1}{|c|}{\multirow{2}{*}{PCG-S3}}    & \multicolumn{1}{c|}{atomic}           & -1368.00214302 & 313   & 9.7E-6  \\ \cline{2-5} 
			\multicolumn{1}{|c|}{}                           & \multicolumn{1}{c|}{atomic+random}    & -1368.00214304 & 541   & 9.9E-6  \\ \hline
			\multicolumn{1}{|c|}{\multirow{2}{*}{PCG-S3-r1}} & \multicolumn{1}{c|}{atomic}           & -1367.99610068 & 309   & 8.0E-6  \\ \cline{2-5} 
			\multicolumn{1}{|c|}{}                           & \multicolumn{1}{c|}{atomic+random}    & -1368.00214304 & 226   & 9.4E-6  \\ \hline
			\multicolumn{1}{|c|}{\multirow{2}{*}{PCG-S3-r2}} & \multicolumn{1}{c|}{atomic}           & -1368.00214301 & 310   & 9.5E-6  \\ \cline{2-5} 
			\multicolumn{1}{|c|}{}                           & \multicolumn{1}{c|}{atomic+random}    & -1368.00214303 & 315   & 9.1E-6  \\ \hline
			\multicolumn{5}{|c|}{AlCrTiV \quad $ N_G = 1759 $ \quad $ N = 25 $ \quad $N_k=144$ \quad $cores= 36 $}                      \\ \hline
			\multicolumn{1}{|c|}{\multirow{2}{*}{SCF}}       & \multicolumn{1}{c|}{atomic}           & -479.31161107  & 500   & 9.6E-4  \\ \cline{2-5} 
			\multicolumn{1}{|c|}{}                           & \multicolumn{1}{c|}{atomic+random}    & -479.30711971  & 500   & 2.1E-3  \\ \hline
			\multicolumn{1}{|c|}{\multirow{2}{*}{PCG-S3}}    & \multicolumn{1}{c|}{atomic}           & -479.36717223  & 204   & 9.8E-6  \\ \cline{2-5} 
			\multicolumn{1}{|c|}{}                           & \multicolumn{1}{c|}{atomic+random}    & -479.36717223  & 154   & 9.2E-6  \\ \hline
			\multicolumn{1}{|c|}{\multirow{2}{*}{PCG-S3-r1}} & \multicolumn{1}{c|}{atomic}           & -479.36717223  & 104   & 9.8E-6  \\ \cline{2-5} 
			\multicolumn{1}{|c|}{}                           & \multicolumn{1}{c|}{atomic+random}    & -479.36717223  & 128   & 9.2E-6  \\ \hline
			\multicolumn{1}{|c|}{\multirow{2}{*}{PCG-S3-r2}} & \multicolumn{1}{c|}{atomic}           & -479.36717223  & 90    & 8.5E-6  \\ \cline{2-5} 
			\multicolumn{1}{|c|}{}                           & \multicolumn{1}{c|}{atomic+random}    & -479.36717223  & 109   & 9.5E-6  \\ \hline
		\end{tabular}
	\end{table}

	\section{Concluding remarks}\label{sec:concluding}
	In this paper, we have first investigated the energy minimization model of the ensemble Kohn-Sham density functional theory from a mathematical aspect, in which the pseudo-eigenvalue matrix and the general smearing approach are involved. We have shown the invariance and the existence of the minimizer of the energy functional and proposed a preconditioned conjugate gradient method to solve the numerical approximations of the energy minimization problem. In particular, we have presented an adaptive double step size strategy since the iterative behavior for $\Psi$ and $\eta$ may be different. Under some mild and reasonable assumptions, we have obtained the global convergence of the PCG algorithm based on the adaptive double step size strategy. We have reported a large number of numerical experiments which can not only verify our theory, but also show the superiority over the traditional SCF iterations. In particular, our numerical experiments have demonstrated that our algorithm can produce convergent numerical approximations for some metallic systems, for which the traditional self-consistent field iterations fails to converge.
	
	\appendix
	\section{Gradient of the energy functional}\label{appx:gradient}
	In this appendix, we introduce the gradient of $\mathcal{F}$ with respect to $\Psi$ and $\eta$. Assume that the exchange-correction functional $\mathcal{E}_{\textup{xc}}$ is differentiable.
	
	Since $\Psi_{\mathrm{k}}$ is complex valued and $\mathcal{F}$ is real valued, $\mathcal{F}$ is not differentiable with respect to $\Psi_{\mathrm{k}}$. Let $\Psi_{\mathrm{k}}=\Psi_{\mathrm{k},\textup{Re}}+\mathrm{i}\Psi_{\mathrm{k},\textup{Im}}$, where $\Psi_{\mathrm{k},\textup{Re}}$ and $\Psi_{\mathrm{k},\textup{Im}}$ are real valued. We see that $\mathcal{F}$ is differentiable with respect to $\Psi_{\mathrm{k},\textup{Re}}$ and $\Psi_{\mathrm{k},\textup{Im}}$. Thus we apply the Wirtinger derivatives. More precisely, we view $\Psi_{\mathrm{k}}$ and $\bar{\Psi}_{\mathrm{k}}$ as two independent variables for all $\mathrm{k}\in\mathcal{K}$, then the energy functional \eqref{eq:free-energy} is a differentiable functional of $\Psi$, $\bar{\Psi}$ and $\eta$, which is still denoted by $\mathcal{F}$ for convenience, namely, $\mathcal{F}(\Psi,\bar{\Psi},\eta)$. A direct calculation shows
	\[
	\mathcal{F}_{\Psi_{\mathrm{k}}}=\frac{1}{2}(\mathcal{F}_{\Psi_{\mathrm{k},\textup{Re}}}-\mathrm{i}\mathcal{F}_{\Psi_{\mathrm{k},\textup{Im}}}).
	\]
	We refer to \cite{kreutz-delgado2009complex} for more details. We use the convenient notation $\mathcal{F}(\Psi,\eta)=\mathcal{F}(\Psi,\bar{\Psi},\eta)$ and $\mathcal{L}(\Psi,\eta,\Lambda)=\mathcal{L}(\Psi,\bar{\Psi},\eta,\Lambda)$. Then there holds
	\[
	\mathcal{F}_{\Psi_{\mathrm{k}}}(\Psi,\eta)=w_{\mathrm{k}}H_{\mathrm{k}}(\rho_{\Psi,\eta})\Psi_{\mathrm{k}} F_{\eta_{\mathrm{k}}}
	\]
	and
	\[
	\mathcal{L}_{\Psi_{\mathrm{k}}}(\Psi,\eta,\Lambda)=w_{\mathrm{k}}(H_{\mathrm{k}}(\rho_{\Psi,\eta})\Psi_{\mathrm{k}} F_{\eta_{\mathrm{k}}} - \mathcal{B}\Psi_{\mathrm{k}}\Lambda_{\mathrm{k}}),
	\]
	where
	\[
	H_{\mathrm{k}}(\rho) = -\frac12(\mathrm{i}\mathrm{k}+\nabla)^2 + \tilde{V}_\text{loc}(\rho)+ \tilde{V}_{\text{nl}}(\rho)
	\]
	with $\displaystyle\tilde{V}_{\text{loc}}(\rho)=V_\text{loc} + \int_{\Omega}\frac{\rho(r)}{|\cdot-r|}\textup{d}r+V_{\text{xc}}(\rho)$, $\tilde{V}_{\text{nl}}(\rho):\Psi_{\mathrm{k}}\mapsto V_{\text{nl}}(\Psi_{\mathrm{k}})+M\tilde{D}\inner{M^*\Psi_{\mathrm{k}}}$, $\displaystyle V_{\text{xc}}(\rho)=\frac{\delta \mathcal{E}_\text{xc}}{\delta\rho}$, and
	\[
	\tilde{D}=\int_{\Omega}\tilde{V}_{\text{loc}}(\rho)(r)\mathcal{Q}(r)\operatorname{d\!}r.
	\]
	It is clear that at any minimizer $(\Psi,\eta)$, we have 
	\[
	\Lambda_{\mathrm{k}}=\langle\Psi_{\mathrm{k}}^* H(\rho_{\Psi,\eta})\Psi_{\mathrm{k}}\rangle F_{\eta_{\mathrm{k}}}.
	\]
	Hence we set
	\[
	\nabla_{\Psi_{\mathrm{k}}}\mathcal{F}(\Psi,\eta)=2\mathcal{L}_{\Psi_{\mathrm{k}}}(\Psi,\eta,(\langle\Psi_{\mathrm{k}}^* H(\rho_{\Psi,\eta})\Psi_{\mathrm{k}}\rangle F_{\eta_{\mathrm{k}}})_{\mathrm{k}\in\mathcal{K}})
	\]
	and $\nabla_{\Psi}\mathcal{F}=(\nabla_{\Psi_{\mathrm{k}}}\mathcal{F})_{\mathrm{k}\in\mathcal{K}}$. Obviously, $\mathcal{L}_{\Psi_{\mathrm{k},\textup{Re}}}=\mathcal{L}_{\Psi_{\mathrm{k},\textup{Im}}}=0$ if and only if $\mathcal{L}_{\Psi_{\mathrm{k}}}=0$.
	
	Then we calculate $\displaystyle\mathcal{F}_{\eta_{\mathrm{k}}}=\left(\frac{\partial\mathcal{F}}{\partial\eta_{\mathrm{k}ij}}\right)_{i,j=1}^N$ by
	referring to Appendix E in \cite{ismail-beigi2000new}. We see that
	\begin{equation}\label{eq:depsilon}
		\textup{d}\epsilon_{\mathrm{k}i} = (P_{\mathrm{k}}^*\textup{d}\eta_{\mathrm{k}} P_{\mathrm{k}})_{ii},\quad i=1,\ldots,N
	\end{equation}
	and
	\begin{equation}\label{eq:dF}
		\begin{aligned}
			(\textup{d}F_{\eta_{\mathrm{k}}})_{ij} &=\sum_{i',j'=1}^N P_{\mathrm{k}ii'}\left(P^*_{\mathrm{k}}f\left(\frac{\eta_{\mathrm{k}}-\mu I}{\sigma}\right)P_{\mathrm{k}}\right)_{i'j'}P^*_{\mathrm{k}j'j}\\
			&= \sum_{i'=1}^N P_{\mathrm{k}ii'} P^*_{\mathrm{k}i'j}\frac{1}{\sigma}f'\left(\frac{\epsilon_{\mathrm{k}i'}-\mu}{\sigma}\right)(\textup{d}\epsilon_{\mathrm{k}i'}-\textup{d}\mu)\\
			&\quad +\sum_{i'\ne j'}P_{\mathrm{k}ii'}P^*_{\mathrm{k}j'j}\frac{f_{\mathrm{k}j'}-f_{\mathrm{k}i'}}{\epsilon_{\mathrm{k}j'}-\epsilon_{\mathrm{k}i'}}(P_{\mathrm{k}}^*\textup{d}\eta_{\mathrm{k}} P_{\mathrm{k}})_{i'j'},
		\end{aligned}
	\end{equation}
	where $P=(P_{\mathrm{k}})_{\mathrm{k}\in\mathcal{K}}\in\left(\mathcal{O}_{\mathbb{C}}^{N\times N}\right)^{|\mathcal{K}|}$, $P_{\mathrm{k}}^*\eta_{\mathrm{k}} P_{\mathrm{k}}=\operatorname{Diag}(\epsilon_{\mathrm{k}1},\ldots,\epsilon_{\mathrm{k}N})$,  $f_{\mathrm{k}i}=f((\epsilon_{\mathrm{k}i}-\mu)/\sigma)$. We get from $\sum\limits_{\mathrm{k}\in\mathcal{K}}w_{\mathrm{k}}\operatorname{tr} F_{\eta_{k}}=N_e$ that
	\begin{equation}\label{eq:dmu}
		\textup{d}\mu = \frac{\sum_{\mathrm{k}\in\mathcal{K}}w_{\mathrm{k}}\sum_{i=1}^N f'\left(\frac{\epsilon_{\mathrm{k}i}-\mu}{\sigma}\right)\textup{d}\epsilon_{\mathrm{k}i}}{\sum_{\mathrm{k}\in\mathcal{K}}w_{\mathrm{k}}\sum_{i=1}^N f'\left(\frac{\epsilon_{\mathrm{k}i}-\mu}{\sigma}\right)}.
	\end{equation}
	Moreover, we have 
	\begin{equation}\label{eq:dS}
		\begin{aligned}
			\textup{d}\left( \sigma\operatorname{tr} S\left(\frac{1}{\sigma}(\eta_{\mathrm{k}}-\mu I)\right)\right)&=\sigma\sum_{i'=1}^N\textup{d} S\left(\frac{\epsilon_{\mathrm{k}i'}-\mu}{\sigma}\right)\\
			&=\sum_{i'=1}^N S'\left(\frac{\epsilon_{\mathrm{k}i'}-\mu}{\sigma}\right)\left(\textup{d} \epsilon_{\mathrm{k}i'}-\textup{d}\mu\right)\\
			&=\sum_{i'=1}^N \frac{1}{\sigma}(\epsilon_{\mathrm{k}i'}-\mu)f'\left(\frac{\epsilon_{\mathrm{k}i'}-\mu}{\sigma}\right)\left(\textup{d} \epsilon_{\mathrm{k}i'}-\textup{d}\mu\right).
		\end{aligned}
	\end{equation}
	It follows from \eqref{eq:dF} and \eqref{eq:dS} that
	\[\small
	\begin{aligned}
		&\mathrel{\phantom{=}}\frac{\partial \mathcal{F}}{\partial\eta_{\mathrm{k}ij}}\\
		&=\frac{\partial \mathcal{E}}{\partial\eta_{\mathrm{k}ij}}-\sum_{\mathrm{k}'\in\mathcal{K}}w_{\mathrm{k}'}\sigma\frac{\partial \operatorname{tr}S\left(\frac{1}{\sigma}(\eta_{\mathrm{k}'}-\mu I)\right)}{\partial\eta_{\mathrm{k}ij}}\\
		&=\sum_{\mathrm{k}'\in\mathcal{K}}\sum_{i',j'=1}^N\frac{\partial\mathcal{E}}{\partial(F_{\eta_{\mathrm{k}'}})_{i'j'}}\frac{\partial(F_{\eta_{\mathrm{k}'}})_{i'j'}}{\partial\eta_{\mathrm{k}ij}}-\sum_{\mathrm{k}'\in\mathcal{K}}w_{\mathrm{k}'}\sigma\frac{\partial \operatorname{tr}S\left(\frac{1}{\sigma}(\eta_{\mathrm{k}'}-\mu I)\right)}{\partial\eta_{\mathrm{k}ij}}\\
		&=\sum_{\mathrm{k}'\in\mathcal{K}}w_{\mathrm{k}'}\sum_{i''=1}^N\left(\sum_{i',j'=1}^N\langle\psi_{\mathrm{k}'j'},H_{\mathrm{k}'}(\rho_{\Psi,\eta})\psi_{\mathrm{k}'i'}\rangle P_{\mathrm{k}'i'i''} P^*_{\mathrm{k}'i''j'}-\epsilon_{\mathrm{k}'i''}+\mu\right)\frac{1}{\sigma}f'\left(\frac{\epsilon_{\mathrm{k}'i''}-\mu}{\sigma}\right)\frac{\partial\epsilon_{\mathrm{k}'i''}}{\partial\eta_{\mathrm{k}ij}}\\
		&\quad-\frac{\partial\mu}{\partial\eta_{\mathrm{k}ij}}\sum_{\mathrm{k}\in\mathcal{K}}w_{\mathrm{k}'}\sum_{i''=1}^N\left(\sum_{i',j'=1}^N\langle\psi_{\mathrm{k}'j'},H_{\mathrm{k}'}(\rho_{\Psi,\eta})\psi_{\mathrm{k}i'}\rangle P_{\mathrm{k}'i'i''} P^*_{\mathrm{k}'i''j'}-\epsilon_{\mathrm{k}'i''}+\mu\right)\frac{1}{\sigma}f'\left(\frac{\epsilon_{\mathrm{k}i''}-\mu}{\sigma}\right)\\
		&\quad +w_{\mathrm{k}}\sum_{i''\ne j''}\left(\sum_{i',j'=1}^N\langle\psi_{\mathrm{k}j'},H_{\mathrm{k}}(\rho_{\Psi,\eta})\psi_{\mathrm{k}i'}\rangle P_{\mathrm{k}i'i''}P^*_{\mathrm{k}j''j'}\right)\frac{f_{\mathrm{k}j''}-f_{\mathrm{k}i''}}{\epsilon_{\mathrm{k}j''}-\epsilon_{\mathrm{k}i''}}P^*_{\mathrm{k}i''i}P_{\mathrm{k}jj''},
	\end{aligned}
	\]
	which together with \eqref{eq:depsilon} and \eqref{eq:dmu} leads to
	\[
	\small
	\begin{aligned}
		&\mathrel{\phantom{=}}\frac{\partial \mathcal{F}}{\partial\eta_{\mathrm{k}ij}}\\
		&=w_{\mathrm{k}}\sum_{i'=1}^N (\langle\tilde{\psi}_{\mathrm{k}i'},H_{\mathrm{k}}(\rho_{\tilde{\Psi},\eta_{\text{D}}})\tilde{\psi}_{\mathrm{k}i'}\rangle-\epsilon_{\mathrm{k}i'}+\mu)\frac{1}{\sigma}f'\left(\frac{\epsilon_{\mathrm{k}i'}-\mu}{\sigma}\right)P^*_{\mathrm{k}i'i}P_{\mathrm{k}ji'}\\
		&\quad-\frac{w_{\mathrm{k}}\sum_{i'=1}^N f'\left(\frac{\epsilon_{\mathrm{k}i'}-\mu}{\sigma}\right)P^*_{\mathrm{k}i'i}P_{\mathrm{k}ji'}}{\sum_{\mathrm{k}'\in\mathcal{K}}w_{\mathrm{k}'}\sum_{i'=1}^N f'\left(\frac{\epsilon_{\mathrm{k}'i'}-\mu}{\sigma}\right)}\sum_{\mathrm{k}'\in\mathcal{K}}w_{\mathrm{k}'}\sum_{i'=1}^N (\langle\tilde{\psi}_{\mathrm{k}'i'},H_{\mathrm{k}'}(\rho_{\tilde{\Psi},\eta_{\text{D}}})\tilde{\psi}_{\mathrm{k}'i'}\rangle-\epsilon_{\mathrm{k}'i'}+\mu)\frac{1}{\sigma}f'\left(\frac{\epsilon_{\mathrm{k}'i'}-\mu}{\sigma}\right)\\
		&\quad +w_{\mathrm{k}}\sum_{i'\ne j'}\langle\tilde{\psi}_{\mathrm{k}j'},H_{\mathrm{k}}(\rho_{\tilde{\Psi},\eta_\text{D}})\tilde{\psi}_{\mathrm{k}i'}\rangle\frac{f_{\mathrm{k}j'}-f_{\mathrm{k}i'}}{\epsilon_{\mathrm{k}j'}-\epsilon_{\mathrm{k}i'}}P^*_{\mathrm{k}i'i}P_{\mathrm{k}jj'}\\
		&=w_{\mathrm{k}}\bigg(\sum_{i'=1}^N (\langle\tilde{\psi}_{\mathrm{k}i'},H_{\mathrm{k}}(\rho_{\tilde{\Psi},\eta_\text{D}})\tilde{\psi}_{\mathrm{k}i'}\rangle-\epsilon_{\mathrm{k}i'})\frac{1}{\sigma}f'\left(\frac{\epsilon_{\mathrm{k}i'}-\mu}{\sigma}\right)P^*_{\mathrm{k}i'i}P_{\mathrm{k}ji'}\\
		&\quad-\frac{\sum_{i'=1}^N f'\left(\frac{\epsilon_{\mathrm{k}i'}-\mu}{\sigma}\right)P^*_{\mathrm{k}i'i}P_{\mathrm{k}ji'}}{\sum_{\mathrm{k}'\in\mathcal{K}}w_{\mathrm{k}'}\sum_{i'=1}^N f'\left(\frac{\epsilon_{\mathrm{k}'i'}-\mu}{\sigma}\right)}\sum_{\mathrm{k}'\in\mathcal{K}}w_{\mathrm{k}'}\sum_{i'=1}^N (\langle\tilde{\psi}_{\mathrm{k}'i'},H_{\mathrm{k}'}(\rho_{\tilde{\Psi},\eta_\text{D}})\tilde{\psi}_{\mathrm{k}'i'}\rangle-\epsilon_{\mathrm{k}'i'})\frac{1}{\sigma}f'\left(\frac{\epsilon_{\mathrm{k}'i'}-\mu}{\sigma}\right)\\
		&\quad +\sum_{i'\ne j'}\langle\tilde{\psi}_{\mathrm{k}j'},H_{\mathrm{k}}(\rho_{\tilde{\Psi},\eta_{\text{D}}})\tilde{\psi}_{\mathrm{k}i'}\rangle\frac{f_{\mathrm{k}j'}-f_{\mathrm{k}i'}}{\epsilon_{\mathrm{k}j'}-\epsilon_{\mathrm{k}i'}}P^*_{\mathrm{k}i'i}P_{\mathrm{k}jj'}\bigg).
	\end{aligned}
	\]
	Here $\tilde{\Psi}=(\tilde{\Psi}_{\mathrm{k}})_{\mathrm{k\in\mathcal{K}}},~\eta_{\text{D}}=(\eta_{\mathrm{k},\text{D}})_{\mathrm{k}\in\mathcal{K}}$, $\tilde{\Psi}_{\mathrm{k}}=(\tilde{\psi}_{\mathrm{k}1},\ldots,\tilde{\psi}_{\mathrm{k}N})=\Psi_{\mathrm{k}} P_{\mathrm{k}}$, $\eta_{\mathrm{k},\text{D}}\coloneqq\operatorname{Diag}(\epsilon_{\mathrm{k}1},\ldots,\epsilon_{\mathrm{k}N})$, and  
	\[
	\frac{f_{\mathrm{k}j'}-f_{\mathrm{k}i'}}{\epsilon_{\mathrm{k}j'}-\epsilon_{\mathrm{k}i'}}=\frac{1}{\sigma}f'\left(\frac{\epsilon_{\mathrm{k}i'}-\mu}{\sigma}\right)
	\]
	provided $\epsilon_{\mathrm{k}j'}=\epsilon_{\mathrm{k}i'}$.
	
	When all $\eta_{\mathrm{k}}$ are diagonal matrix, we see from $P_{\mathrm{k}}=I_N$ for all $\mathrm{k}\in\mathcal{K}$ that
	\begin{align*}
		\frac{\partial \mathcal{F}}{\partial\eta_{\mathrm{k}ij}}&=w_{\mathrm{k}}\bigg( (\langle\psi_{\mathrm{k}i},H_{\mathrm{k}}(\rho_{\Psi,\eta})\psi_{\mathrm{k}i}\rangle-\epsilon_{\mathrm{k}i})\frac{1}{\sigma}f'\left(\frac{\epsilon_{\mathrm{k}i}-\mu}{\sigma}\right)\delta_{ij}\\
		&\quad-\frac{ f'\left(\frac{\epsilon_{\mathrm{k}'i}-\mu}{\sigma}\right)\delta_{ij}}{\sum_{\mathrm{k}'}w_{\mathrm{k}'}\sum_{i'=1}^N f'\left(\frac{\epsilon_{\mathrm{k}'i'}-\mu}{\sigma}\right)} d_{\mu}\\
		&\quad +\langle\psi_{\mathrm{k}j},H(\rho_{\Psi,\eta})\psi_{\mathrm{k}i}\rangle\frac{f_{\mathrm{k}j}-f_{\mathrm{k}i}}{\epsilon_{\mathrm{k}j}-\epsilon_{\mathrm{k}i}}(1-\delta_{ij})\bigg)
	\end{align*}
	for any $\mathrm{k}\in\mathcal{K}$, where
	\begin{equation}\label{eq:dFdmu}
		d_{\mu}=\sum_{\mathrm{k}'\in\mathcal{K}}w_{\mathrm{k}'}\sum_{i'=1}^N (\langle\psi_{\mathrm{k}'i'},H_{\mathrm{k}'}(\rho_{\Psi,\eta})\psi_{\mathrm{k}'i'}\rangle-\epsilon_{\mathrm{k}'i'})\frac{1}{\sigma}f'\left(\frac{\epsilon_{\mathrm{k}'i'}-\mu}{\sigma}\right).
	\end{equation}
	We denote by $\nabla_{\eta_{\mathrm{k}}}\mathcal{F}=\mathcal{F}_{\eta_{\mathrm{k}}}^T=\left(\left(\frac{\partial \mathcal{F}}{\partial\eta_{\mathrm{k}ij}}\right)_{i,j=1}^N\right)^T$, $\nabla_\eta\mathcal{F}=(\nabla_{\eta_{\mathrm{k}}}\mathcal{F})_{\mathrm{k}\in\mathcal{K}}$.
	
	\section{Kohn-Sham equation}\label{appx:Kohn-Sham}
	In this appendix, we show the associated standard Kohn-Sham equation for the ensemble Kohn-Sham DFT.
	
	Let $\mathcal{L}_{\Psi}(\Phi, \eta,\Lambda) = 0$, i.e.,
	\begin{equation}
		H_{\mathrm{k}}(\rho_{\Phi,\eta})\Phi_{\mathrm{k}} F_{\eta_{\mathrm{k}}} = \mathcal{B}\Phi_{\mathrm{k}} \Lambda_{\mathrm{k}},\quad\forall\mathrm{k}\in\mathcal{K}.
	\end{equation}
	Thus we have
	\begin{equation}\label{eq:Lambdaeq}
		\Sigma_{\Phi_{\mathrm{k}},\eta_{\mathrm{k}}}F_{\eta_{\mathrm{k}}} = \Lambda_{\mathrm{k}},
	\end{equation}
	where $\Sigma_{\Phi_{\mathrm{k}},\eta_{\mathrm{k}}}=\langle\Phi_{\mathrm{k}}^* H(\rho_{\Phi,\eta})\Phi_{\mathrm{k}}\rangle$. Let $\mathcal{L}_{\eta}(\Phi,\eta,\Lambda)=0$. Without loss of generality, let all $\eta_{\mathrm{k}}$ be diagonal. If not, by \eqref{eq:grad-invariant}, we still have $\mathcal{L}_{\Psi_{\mathrm{k}}}=0$ and $\mathcal{L}_{\eta_{\mathrm{k}}}=0$ after diagonalizing $\eta_{\mathrm{k}}$ and then rotating the $\Phi_{\mathrm{k}}$ and performing a similarity transformation on $\Lambda_{\mathrm{k}}$ accordingly.
	
	Denote $\eta_{\mathrm{k}}=\operatorname{Diag}(\epsilon_{\mathrm{k}1},\ldots,\epsilon_{\mathrm{k}N})$. Since $f$ is strictly monotonic decreasing, the derivatives of $f$ are always less than $0$. We obtain from $\eta_{\mathrm{k}}$ being diagonal and $\mathcal{L}_{\eta_{\mathrm{k}}}(\Phi,\eta,\Lambda)=0$ that $\Sigma_{\Phi_{\mathrm{k}},\eta_{\mathrm{k}}}=\eta_{\mathrm{k}}+cI$ is diagonal, where
	\begin{equation}\label{eq:c-value}
		c=\frac{d_{\mu}}{\frac{1}{\sigma}\sum_{\mathrm{k}}w_{\mathrm{k}}\sum_{i'=1}^N f'\left(\frac{\epsilon_{\mathrm{k}i'}-\mu}{\sigma}\right)}.
	\end{equation}
	Here $d_{\mu}$ is defined as \eqref{eq:dFdmu}. Denote $\varepsilon_{\mathrm{k}i}=\epsilon_{\mathrm{k}i}+c$, then $\Sigma_{\Phi_{\mathrm{k}},\eta_{\mathrm{k}}}=\operatorname{Diag}(\varepsilon_{\mathrm{k}1},\ldots,\varepsilon_{\mathrm{k}N})$. Consequently, we arrive at the standard Kohn-Sham equation
	\begin{equation}\label{eq:KS}
		H(\rho)\phi_{\mathrm{k}i} = \varepsilon_{\mathrm{k}i}\mathcal{B}\phi_{\mathrm{k}i},\quad i=1,2,\ldots,N.
	\end{equation}
	where $\displaystyle\rho=\sum_{\mathrm{k}\in\mathcal{K}}w_{\mathrm{k}}\operatorname{tr}((\Phi^*_{\mathrm{k}}\Psi_{\mathrm{k}}+\langle\Phi^*_{\mathrm{k}}M\rangle\mathcal{Q}\langle M^*\Phi_{\mathrm{k}}\rangle) F_{\eta_{\mathrm{k}}})$, $\eta_{\mathrm{k}}=\operatorname{Diag}(\varepsilon_{\mathrm{k}1},\varepsilon_{\mathrm{k}2},\ldots,\varepsilon_{\mathrm{k}N})$.
	
	If $\Lambda_{\mathrm{k}}$ are forced to be Hermitian, then we can derive the Kohn-Sham equation without the condition $\mathcal{L}_{\eta}(\Phi,\eta,\Lambda)=0$. Indeed, it is clear that $\Sigma_{\Phi_{\mathrm{k}},\eta_{\mathrm{k}}}=\langle\Phi^*_{\mathrm{k}} H(\rho_{\Phi,\eta})\Phi_{\mathrm{k}}\rangle$ are Hermitian since Hamiltonian operator $H(\rho_{\Phi,\eta})$ is self-adjoint. It follows from  $\Lambda^*_{\mathrm{k}}=\Lambda_{\mathrm{k}}$ and $F_{\eta_{\mathrm{k}}}^*=F_{\eta_{\mathrm{k}}}$ that
	\begin{equation}\label{eq:Sigma-commute-F}
		\Sigma_{\Phi_{\mathrm{k}},\eta_{\mathrm{k}}} F_{\eta_{\mathrm{k}}} = F_{\eta_{\mathrm{k}}} \Sigma_{\Phi_{\mathrm{k}},\eta_{\mathrm{k}}}.
	\end{equation}
	Thus there exists $P\in\left(\mathcal{O}^{N\times N}\right)^{|\mathcal{K}|}$ such that
	$$\Sigma_{\Phi_{\mathrm{k}} P_{\mathrm{k}},P_{\mathrm{k}}^*\eta_{\mathrm{k}} P_{\mathrm{k}}}=P^*_{\mathrm{k}}\Sigma_{\Phi_{\mathrm{k}},\eta_{\mathrm{k}}}P_{\mathrm{k}},\quad F_{P_{\mathrm{k}}^*\eta P_{\mathrm{k}}}=P^*_{\mathrm{k}} F_{\eta_{\mathrm{k}}} P_{\mathrm{k}},\quad P^*_{\mathrm{k}}\Lambda_{\mathrm{k}} P_{\mathrm{k}}$$ are diagonal. Let $\operatorname{Diag}(\varepsilon_{\mathrm{k}1},\ldots,\varepsilon_{\mathrm{k}N})=P_{\mathrm{k}}^*\Lambda_{\mathrm{k}} F_{\eta_{\mathrm{k}}}^{-1} P_{\mathrm{k}}$. We still denote $\Phi_{\mathrm{k}} P_{\mathrm{k}}$ and $P^*_{\mathrm{k}}\eta_{\mathrm{k}} P_{\mathrm{k}}$ by $\Phi_{\mathrm{k}}$  and $\eta_{\mathrm{k}}$, respectively. Consequently, we arrive at \eqref{eq:KS}.
	
	The Kohn-Sham equations \eqref{eq:KS} are usually solved by the SCF iterations which is stated as Algorithm \ref{algo:scf}.
	
	\begin{algorithm}[!htbp]
		\caption{The SCF iteration method for solving ensemble Kohn-Sham DFT}
		\label{algo:scf}
		\begin{algorithmic}[1]
			\STATE Given $\epsilon>0$, $\sigma$ and initial guess of the input density $\rho_{\textup{in}}$. Set $\rho_{\textup{out}}=0$;
			\WHILE{$\|\rho_{\textup{out}}-\rho_{\textup{in}}\|>\epsilon$}
			\STATE Obtain the input density $\rho_{\textup{in}}$ by some mixing schemes from $\rho_{\textup{out}}$ and the density of previous steps;
			\STATE Solve the linear eigenvalue problems
			\[
			H(\rho_{\textup{in}})\phi_{\mathrm{k}i} = \varepsilon_{\mathrm{k}i}\phi_{\mathrm{k}i},
			\]
			to get eigenpairs $(\phi_{\mathrm{k}i},\varepsilon_{\mathrm{k}i})$, $\mathrm{k}\in\mathcal{K},~i=1,2\ldots,N$;
			\STATE Calculate $\mu$ and occupation numbers $f_{\mathrm{k}i}$ corresponding to eigenfunctions $\phi_{\mathrm{k}i}$ such that $\sum\limits_{\mathrm{k}\in\mathcal{K}}w_{\mathrm{k}}\sum\limits_{i=1}^N f_{\mathrm{k}i}=N_e$ and
			\[
			f_{\mathrm{k}i}=f\left(\frac{\varepsilon_{\mathrm{k}i}-\mu}{\sigma}\right);
			\]
			\STATE Calculate output density
			\[
			\rho_{\textup{out}}=\sum_{\mathrm{k}\in\mathcal{K}}w_{\mathrm{k}}\operatorname{tr}((\Psi^*_{\mathrm{k}}\Psi_{\mathrm{k}}+\langle\Psi^*_{\mathrm{k}}M\rangle\mathcal{Q}\langle M^*\Psi_{\mathrm{k}}\rangle) F_{\eta_{\mathrm{k}}}),
			\]
			where $F_{\eta_\mathrm{k}}=\operatorname{Diag}(f_{\mathrm{k}1},f_{\mathrm{k}2},\ldots,f_{\mathrm{k}N})$;
			\ENDWHILE
		\end{algorithmic}
	\end{algorithm}	

	\section*{Acknowledgments}
	The authors would like to thank Professor Zhigang Wang for providing the configurations of the gold clusters, Professor Nicola Marzari for providing the configurations of the multicomponent systems, and Dr. Liwei Zhang for his helpful discussions. 
	 
	\bibliographystyle{siamplain}
	\bibliography{refs}
\end{document}